\documentclass[11pt]{amsart}

\usepackage[T1]{fontenc}
\usepackage[utf8]{inputenc}
\usepackage[english]{babel}

\usepackage[letterpaper,top=3cm,bottom=3.5cm,left=3cm,right=3cm,marginparwidth=1.75cm]{geometry}

\usepackage{textcomp}
\usepackage{amsmath,amssymb,amsthm}
\usepackage{mathtools}
\usepackage{graphicx}
\usepackage{gensymb}
\usepackage{tipa}
\usepackage{wasysym}
\usepackage{tikz}
\usepackage{fancyhdr}
\usepackage{accents}

\usepackage[hidelinks]{hyperref}

\newtheorem{theorem}{Theorem}[section]
\newtheorem{example}[theorem]{Example}
\newtheorem{proposition}[theorem]{Proposition}
\newtheorem{lemma}[theorem]{Lemma}
\newtheorem{corollary}[theorem]{Corollary}
\newtheorem{definition}[theorem]{Definition}

\newtheorem{remark}[theorem]{Remark}

\newlength{\dhatheight}

\title{Bisingular $(\rho,\delta)$ pseudodifferential operators\\ on products of compact Lie groups}
\author{Guido Drei, Serena Federico and Alberto Parmeggiani}

\address{Guido Drei\newline Dipartimento di Matematica, Universit\`a di Bologna, Piazza di Porta San Donato 5, 40126, Bologna, Italy}
\email {guido.drei2@unibo.it}

\address{Serena Federico\newline Dipartimento di Matematica, Universit\`a di Bologna, Piazza di Porta San Donato 5, 40126, Bologna, Italy}
\email{serena.federico2@unibo.it}

\address{Alberto Parmeggiani\newline Dipartimento di Matematica, Universit\`a di Bologna, Piazza di Porta San Donato 5, 40126, Bologna, Italy}
\email{alberto.parmeggiani@unibo.it}

\thanks{The authors are members of the Research Group GNAMPA of INdAM}
\thanks{The second and third authors were partially supported by the Italian Ministry of University and Research, under PRIN2022 (Scorrimento) “Anomalies
in partial differential equations and applications”, 2022HCLAZ8\_002, J53C24002560006.}

\begin{document}

\begin{abstract}
In this paper we develop the $(\rho,\delta)$ calculus of bisingular pseudodifferential operators on the product compact Lie groups, which extends the one that was earlier introduced by the second and third authors
in the $(1,0)$ case. The $(\rho,\delta)$ case is the calculus suitable to construct parametrices of hypoelliptic operators on $G_1\times G_2$. 
\end{abstract}

\maketitle

\renewcommand{\theequation}{\thesection.\arabic{equation}}

\section{Introduction}\label{sec1}

Bisingular operators came naturally to the fore in the celebrated Atiyah-Singer Index Theorem \cite{AS}, as systems of the kind
$$A_1\boxtimes A_2=\left[\begin{array}{cc}A_1\otimes I & -I\otimes A^*_2\\
  I\otimes A_2 & A_1^*\otimes I\end{array}\right],$$
where $A_j$ is a pseudodifferential operator of order $m_j$ on a manifold $M_j.$ The notation $A_1\otimes I$ refers of course to the fact that it is the unique operator that
acts on tensor products of functions $u_1=u_1(x_1)$ on $M_1$ and $u_2=u_2(x_2)$ on $M_2$ as
$$(A_1\otimes I)(u_1\otimes u_2)=(A_1u_1)\otimes u_2.$$
Likewise for $I\otimes A_2.$ It turns out that $A_1\boxtimes A_2$ is a bisingular pseudodifferential operator of bi-order $(m_1,m_2)$ as defined in Rodino \cite{R}, who introduced
and studied the bisingular $(1,0)$-classes (i.e. $\rho=1$ and $\delta=0$) and their calculus properties. See also the paper \cite{NR} by Nicola and Rodino about residues and index for bisingular operators (with $\rho=1$ and $\delta=0$), and the paper \cite{BS} by Borsero and Schulz about a wave-front set tailored to the bisingular calculus (in the $(1,0)$ classes). It is worth recalling that problems related to bisingular integral operators appeared in Pilidi \cite{P}, when
studying a boundary value problem on the distinguished boundary (the Cartesian product of boundaries)
for holomorphic functions in two complex variables. A more recent, very interesting, case of use of the bisingular calculus is found in the paper
\cite{GW} of G\'erard and Wrochna, in their construction of Hadamard states for Klein-Gordon fields in the interior of a light cone in a globally hyperbolic spacetime.
However, the main \textit{raison d'\^{e}tre} of bisingular operators is indeed the class of systems of the kind $A_1\boxtimes A_2$ as introduced by Atiyah and Singer.
Along that line, it is of course interesting to consider general bisingular systems of the kind
$$\mathbf{A}=\left[\begin{array}{cc}A_{11}\otimes I & I\otimes A_{12}\\
I\otimes A_{21} & A_{22}\otimes I\end{array}\right],$$
where $A_{11},A_{22}$ are pseudos of order $m_1$ on some manifold $M_1$ and $A_{12},A_{21}$ are pseudos of order $m_2$ on some other manifold $M_2$,
so that in the calculus for systems such as $\mathbf{A}$ one has to deal with general bisingular entries. Operators of the kind $\mathbf{A}$ also suggest that natural classes of bisingular operators are the classes of symbols in
$$S^{m_1,0}_{(\rho_1,1),(0,0)}+S^{0,m_2}_{(1,\rho_2),(0,0)}\subset S^{m_1,m_2}_{(\rho_1,\rho_2),(0,0)},$$
and, more generally, with entries $jk$ that belong to $S^{m_j,m_k}_{(\rho_j,\rho_k),(\delta_j,\delta_k)}$.

Notice that 
$$(A_1\boxtimes A_2)^*A_1\boxtimes A_2=\left[\begin{array}{cc}A_1^*A_1\otimes I+I\otimes A_2^*A_2 & 0\\
0 & A_1A_1^*\otimes I+I\otimes A_2A_2^*\end{array}\right],$$
which shows that certain (non equivalent) sums of squares may be brought into a product of bisingular operators. Another interesting example, related to the $\bar\partial$ operator is shown in the examples of Section 7.

In \cite{FP}, the second and third authors extended the bisingular calculus to products of compact Lie groups $G_1\times G_2$, in the framework of the \textit{global calculus}
introduced by Ruzhansky and Turunen in \cite{RT}. That calculus was developed in the $(1,0)$ classes. But to deal with parametrices of non-elliptic but
hypoelliptic operators, and add flexibility to the calculus, it is natural to consider also $(\rho,\delta)$-classes. This is our aim in the present paper, in which we
indeed extend the bisingular global calculus on a product of compact Lie groups to $(\rho,\delta)$-classes, the main point being the control of the Schwartz kernel of the operators. As it will be clear in Section 6, where we introduce classes of bihypoelliptic operators and prove the existence of parametrices, in the bisingular case pointwise estimates on the inverse of the symbol do not suffice to obtain the symbol estimates for the inverse of the symbol, as it is in the Euclidean setting. In this case, to get the needed estimates of the inverse of the symbol, one has to assume a certain behavior of the latter with respect to finite-difference derivatives of the symbol itself (see \eqref{horm} below). That gives rise to a new situation, not paralleled in the standard $(\rho,\delta)$-classes.

We end this introduction by giving the plan of the paper. In the next section we will be giving some preliminaries, to summarize the kind of Fourier analysis that we will be using
throughout the paper. In Section \ref{sec3}, we will define bisingular $(\rho,\delta)$ symbols. In Section \ref{sec4}, we will prove some delicate kernel estimates which are fundamental
to the calculus, that will be developed in Section \ref{sec5}. In Section \ref{sec6}, we will define biellipticity and bihypoellipticity, and give the main results about the existence of a bisingular parametrix. In the final Section \ref{sec7} we will give a number of examples, among which an example of a hypoelliptic bisingular system of
the kind $A_1\boxtimes A_2$ with inverse belonging to our $(\rho,\delta)$-classes, and examples that are instances of operators such as the $\bar\partial,$ or the Laplacian, that can be factored as products of bisingular systems.
  
\section{Preliminaries}\label{sec2}
In this section we briefly introduce all the necessary tools to perform our analysis on compact Lie groups. We shall recall the definition of the Fourier transform, of the difference operators (that is, roughly speaking, of the differential operators in the ``frequency variable''), some corresponding crucial properties, and Taylor's formula in this setting.

In the following $G$ will always be a compact Lie group. By $\mathrm{Rep}(G)$ we denote the set of all irreducible unitary representations of $G$, while  $\widehat{G}$ denotes the unitary dual of $G$, i.e.
the set of all equivalence classes of irreducible unitary representations of $G$. Since $G$ is compact, each $\xi\in\mathrm{Rep}(G)$ is finite dimensional and $\widehat{G}$ is countable. 
For any representation $\xi\in\mathrm{Rep}(G)$, the representation space is denoted by $\mathcal{H}_{\xi}$ and $\mathcal{U}(\mathcal{H}_{\xi})$ is the corresponding space of unitary 
operators on $\mathcal{H}_{\xi}$.\ \\
Given a smooth function $f\in C^{\infty}(G)$ and a representation $\xi\in\mathrm{Rep}(G)$, the (matrix-valued) Fourier transform $\hat{f}(\xi)$ of $f$ at $\xi$ is defined by 
$$\hat{f}(\xi)=\int_Gf(x)\xi^*(x)dx,$$ 
where $\xi^*(x):=\overline{^t\xi(x)}$ is the adjoint representation of $\xi$ and $dx$ denotes the Haar measure on the compact group $G$. 
Notice that for a representation $\xi:G\rightarrow\mathcal{U}(\mathcal{H}_{\xi})$ of dimension $d_{\xi}:=\mathrm{dim}(\mathcal{H}_{\xi})$, 
one has $\hat{f}(\xi)\in\mathbb{C}^{d_{\xi}\times d_{\xi}}$ (the complex matrices of dimensions $d_\xi\times d_\xi$). 
We can recover the function $f$ from its Fourier coefficients by 
$$f(x)=\sum_{\xi\in\widehat{G}}d_{\xi}\mathrm{Tr}(\xi(x)\hat{f}(\xi)), \ \ \ x\in G,$$ 
where $\mathrm{Tr}(\cdot)$ is the trace operator. 
The following Parseval identity holds 
$$\lVert f\rVert^2_{L^2(G)}=\sum_{\xi\in\widehat{G}}d_{\xi}\lVert\hat{f}(\xi)\rVert^2_{HS}=:\lVert\hat{f}\rVert^2_{\ell^2(\widehat{G})},$$ 
where $\lVert\hat{f}(\xi)\rVert_{HS}:=(\mathrm{Tr}(\hat{f}(\xi)^*\hat{f}(\xi)))^{\frac{1}{2}}=\biggl(\sum_{i,j=1}^{\mathrm{dim}(\xi)}|\hat{f}(\xi)_{ij}|^2\biggr)^{\frac{1}{2}}$ 
is the Hilbert-Schmidt norm.

Recall that if $u\in\mathcal{D}'(G)$ and $\xi\in\hat{G}$, then $\hat{u}(\xi)=u(\xi^*)=[u(\xi^*_{jk}]_{1\leq j,k\leq d_\xi}$.

\begin{definition}A difference operator $Q_{\xi}$ of order $k\in\mathbb{N}_0$ on $\mathcal{F}(\mathcal{D}'(G))=\{\hat{u}: u\in\mathcal{D}'(G)\}$ is an operator such that 
$$Q_{\xi}\hat{u}(\xi)=\widehat{q_{Q}u}(\xi),$$ 
for some smooth function $q_Q\in C^{\infty}(G)$ vanishing of order $k$ at the identity element $e\in G$, that is, $q_Q(e)=P_xq_Q(e)=0$ for every left-invariant differential operator 
$P_x\in\mathrm{Diff}^{k-1}(G)$ of order $k-1$.
\end{definition}

We have denoted by $\mathrm{Diff}^{l}(G)$ the set of all left-invariant differential operators on $G$ of order $l\in\mathbb{N}_0$.
By $\mathrm{diff}^k(\widehat{G})$, we denote the set of all difference operators of order $k$ on $G$.

\begin{definition}A collection of $n_{\Delta}\geq n=\mathrm{dim}(G)$ difference operators $\Delta_1,\cdots,\Delta_{n_{\Delta}}$ in $\mathrm{diff}^1(\widehat{G})$ is called 
\textit{admissible} if:
\begin{enumerate}
    \item the corresponding functions $q_1,\cdots,q_{n_{\Delta}}\in C^{\infty}(G)$ vanish of order (at least) one at the neutral element, that is $q_1(e)=\cdots=q_{n_{\Delta}}(e)=0$;
    \item $dq_j(e)\neq 0$ for all $j=1,\cdots,n_{\Delta}$;
    \item $\mathrm{rank}(dq_1(e),\cdots,dq_{n_{\Delta}}(e))=n$.
\end{enumerate} 
A collection of admissible difference operators is said to be \textit{strongly admissible} if one has $\bigcap_{j=1}^{n_{\Delta}}\{x\in G;q_j(x)=0\}=\{e\}.$
\end{definition}

Given a fixed family of smooth functions $Q=\{q_j\}_{j=1,\cdots,n_{\Delta}}$, we shall denote by $\Delta_Q$ the associated admissible collection of difference operators and, 
for $\alpha\in\mathbb{N}_0^{n_{\Delta}}$, by  $\Delta_{Q}^{\alpha}:=\Delta^{\alpha_1}_{q_1}\cdots\Delta^{\alpha_{n_{\Delta}}}_{q_{n_{\Delta}}}$ the element in $\mathrm{diff}^{|\alpha|}(\widehat{G})$ 
associated with the smooth function $q^{\alpha}:=q^{\alpha_1}_{1}\cdots q_{n_{\Delta}}^{\alpha_{n_{\Delta}}}$.\ \\
Once the collection of difference operators is fixed, one may find a family of differential operators in $\mathrm{Diff}^{|\alpha|}(G)$, denoted by $\partial_x^{(\alpha)}$, 
such that the following Taylor's formula holds 
$$f(x)=\sum_{|\alpha|<N}\frac{1}{\alpha!}q(x)^{\alpha}\partial_x^{(\alpha)}f(e)+O(|x|^N),\ \ \ h(x)\rightarrow 0,$$ 
for all $f\in C^{\infty}(G)$, where $|x|$ is the geodesic distance from $x$ to $e_G$. 
The differential operators $\partial^{(\alpha)}_x$ can be replaced by $\partial_x^{\alpha}:=\partial_{x_1}^{\alpha_1}\cdots\partial_{x_n}^{\alpha_n}$, 
where $\partial_{x_j},j=1,\cdots,n=\mathrm{dim}(G)$, is a collection in the Lie algebra $\mathfrak{g}$ of left-invariant first order differential operators corresponding 
to some linearly independent left-invariant vector fields on $G$. We shall use the notation $\partial_{x_j}$ and $\widetilde{\partial_{x_{j}}}$ for the left and right invariant 
vector fields, respectively. Once we fix an orthonormal basis for the left-invariant vector fields for $\mathfrak{g}$, then any element of $\mathrm{Diff}^k(G)$ can be written as a 
linear combination in terms of the element of the basis (and the same holds for the right-invariant vector fields).\ \\
The family of functions $\{q_{ij}=\xi_{ij}-\delta_{ij}\}_{\xi\in\widehat{G},1\leq i,j\leq d_{\xi}}$, where $\delta_{ij}$ is the Kronecker's delta, induces a strongly admissible collection 
of difference operators. Therefore, in the sequel, we choose the latter as the fixed admissible collection.

\begin{definition}\label{Leibnizdef} A collection $\Delta_Q$ of difference operators satisfies the \textit{Leibniz-like} property if, for every distribution $\hat{u}_1,\hat{u}_2\in\mathcal{F}(\mathcal{D}'(G))$, 
$$\Delta_{q_j}(\hat{u}_1\hat{u}_2)=\Delta_{q_j}(\hat{u}_1)\hat u_2+\hat{u}_1\Delta_{q_j}(\hat{u}_2)+\sum_{1\leq l,k\leq n_{\Delta}}c^{(j)}_{lk}\Delta_{q_l}(\hat{u}_1)\Delta_{q_k}(\hat{u}_2),
\ \ \ \ j=1,\dots,n_{\Delta},$$ 
for some coefficients $c^{(j)}_{lk}\in\mathbb{C}$ depending only on $l,k,j,\Delta_Q$.\ \\
    If $\Delta_Q$ is a collection satisfying the Leibniz-like formula, then, recursively, for any $\alpha\in\mathbb{N}_0^{n_{\Delta}}$, and 
$\hat{u}_1,\hat{u}_2\in\mathcal{F}(\mathcal{D}'(G))$ one has
$$\Delta_Q^{\alpha}(\hat{u}_1\hat{u}_2)=\sum_{|\alpha_1|,|\alpha_2|\leq|\alpha|\leq|\alpha_1|+|\alpha_2|}c^{\alpha}_{\alpha_1,\alpha_2}\Delta_Q^{\alpha_1}(\hat{u}_1)\Delta_Q^{\alpha_2}(\hat{u}_2),$$ 
for some coefficients $c_{\alpha_1,\alpha_2}^{\alpha}\in\mathbb{C}$, depending on $\alpha_1,\alpha_2,\alpha,\Delta_Q$, with $c^{\alpha}_{\alpha,0}=c^{\alpha}_{0,\alpha}=1$.
\end{definition}

The strong admissible collection $\Delta_Q$ relative to the family $Q=\{q_{ij}=\xi_{ij}-\delta_{ij}\}_{1\leq i,j\leq d_{\xi},\xi\in \widehat{G}}$ satisfies the Leibniz-like property.
    
\begin{definition}For each $\tau,\xi\in\mathrm{Rep}(G)$ we define the linear mapping $\Delta_{\tau}\hat{f}(\xi)$ on $\mathcal{H}_{\tau}\otimes\mathcal{H}_{\xi}$ by 
$$\Delta_{\tau}\hat{f}(\xi)=\hat{f}(\tau\otimes\xi)-\hat{f}(I_{d_{\tau}}\otimes\xi),\ \ \ \hat{f}\in\mathcal F(\mathcal D'(G)).$$
The restriction of $\Delta_{\tau}\hat{f}(\xi)$ to any occurrence of $\rho\in\widehat{G}$ in the decomposition into irreducibles of $\tau\otimes\xi$ defines the same mapping  on $\mathcal{H}_{\rho}$, while the restriction to any $\rho\in\widehat{G}$ not appearing in the decomposition of $\tau\otimes\xi$ is fixed to be zero. 
The operator $\Delta_{\tau}$ is called the difference operator associated with $\tau\in\mathrm{Rep}(G)$.
\end{definition}

Recall that $\hat f(\tau\otimes\xi)$ should be viewed as an element in $\mathrm{End}(\mathcal H_{\tau\otimes\xi})$. Notice that in general $\tau\otimes\xi$ is not irreducible as a representation of $G$, that is $\mathcal H_{\tau\otimes\xi}=\bigoplus_{\rho\in\widehat G}\mathcal H_{\rho}^{m_{\rho}}$, where $m_{\rho}\in\mathbb N_0$. 
    
\begin{definition}\label{defdiff}Let $G=G_1\times G_2$ be a product of compact Lie groups, with $n_i:=\mathrm{dim}(G_i)$ for $ i=1,2$ and $e=(e_1,e_2)$ the neutral element of $G$. A collection of $n_P:=n_{\Delta_P}\geq n_1$ 
difference operators $\Delta_{p_1},\cdots,\Delta_{p_{n_P}}\in\mathrm{diff}^1(\widehat{G})$ is called \textit{admissible} relative to $G_1$ if the corresponding functions 
$p_1,\cdots,p_{n_P}\in C^{\infty}(G)$ are such that \begin{enumerate}
    \item $p_1(e)=\cdots =p_{n_P}(e)=0$,
    \item $dp_j(e)\neq 0$ for all $j=1,\cdots,n_P$,
    \item $\mathrm{rank}(dp_1(e),\cdots,dp_{n_P}(e))=n_1$.
\end{enumerate}  
The collection is said to be \textit{strongly admissible} relative to $G_1$ if $\bigcap_{j=1}^{n_{P}}\{x\in G; p_j(x)=0\}=\{e_1\}\times G_2.$ By reversing the roles of $G_1$ and $G_2$, 
(strong) admissible collections relative to $G_2$ are defined.
\end{definition}

We consider the family of functions  
$$P=\{p_{ij}^{\tau}=(\tau_1\otimes I_{d_{\tau_2}}-I_{d_{\tau}})_{ij};\ \ 1\leq i,j\leq d_{\tau}, \ \ \tau=\tau_1\otimes\tau_2\in\widehat{G}\},$$ $$R=\{r_{ij}^{(\tau)}=(I_{d_{\tau_1}}\otimes\tau_2-I_{d_{\tau}})_{ij};\ \ 1\leq i,j\leq d_{\tau},\ \ \tau=\tau_1\otimes\tau_2\in\widehat{G}\},$$ so that $\Delta_P$ and $\Delta_R$ are strongly admissible collections relative to $G_1\cong G_1\times\{e_2\}\subset G$ and $G_2\cong\{e_1\}\times G_2\subset G$, respectively.
    
\begin{remark}\label{remdiff}
From the families $P$ and $R$, we may select a finite subfamily (by using fundamental representations, which are finitely many for a compact group, see \cite{F}) and, 
calling them again $P$ and $R$, we have
$$P=\{p_k,k=1,\cdots,n_P\}, \ \ R=\{r_k,k=1,\cdots, n_R\},$$
where $p_k,r_k$ are of the form $p_{ij}^{(\tau)},r^{(\tau)}_{ij}$, respectively, for some $\tau\in\widehat{G},i,j\in\{1,\cdots,d_{\tau}\}$.
\end{remark} 

Notice that the functions $p_k,$ for $k=1,\cdots,n_{P}$, are independent of $x_2\in G_2$ and, similarly, the functions $r_k,$ for $k=1,\cdots,n_R$, are independent of $x_1\in G_1$. 
We may define 
$$\Delta^{\alpha,\beta}:=\Delta^{\alpha}_P\Delta^{\beta}_R=\Delta^{\alpha_1}_{p_1}\cdots\Delta^{\alpha_{n_P}}_{p_{n_P}}\Delta^{\beta_1}_{r_1}\cdots\Delta^{\beta_{n_R}}_{r_{n_R}}.$$
By applying the Leibniz formula in Definition \ref{Leibnizdef} separately in the two variables, we obtain the following property for the difference operators $\Delta^{\alpha_1,\alpha_2}$.
\begin{proposition}\label{Leibniz}
Let $G_1,G_2$ be compact Lie groups and $G=G_1\times G_2$. Then, for every $(\alpha_1,\alpha_2)\in\mathbb{N}_0^{n_{P}}\times\mathbb{N}_0^{n_R}$ and for all $\hat{u}_1,\hat{u}_2\in\mathcal{F}(\mathcal{D}'(G))$, 
we have 
$$\Delta^{\alpha_1,\alpha_2}(\hat{u}_1\hat{u}_2)=\sum_{\begin{matrix}_{|\beta_1|,|\gamma_1|\leq|\alpha_1|\leq|\beta_1|+|\gamma_1|}\\ _{|\beta_2|,|\gamma_2|\leq |\alpha_2|\leq |\beta_2|+|\gamma_2| }\end{matrix}}c^{\alpha_1}_{\beta_1,\gamma_1}c^{\alpha_2}_{\beta_2,\gamma_2}\Delta^{\beta_1,\beta_2}(\hat{u}_1)\Delta^{\gamma_1,\gamma_2}(\hat{u}_2),$$ 
for some coefficients $c_{\beta_1,\gamma_1}^{\alpha_1},c^{\alpha_2}_{\beta_2,\gamma_2}\in\mathbb{C}$, such that $c^{\alpha_1}_{\alpha_1,0}=c^{\alpha_1}_{0,\alpha_1}=c^{\alpha_2}_{\alpha_2,0}=c^{\alpha_2}_{0,\alpha_2}=1$.
\end{proposition}

Since the families $P$ and $R$ define admissible collections of difference operators on $\widehat{G}$ relative to $G_1$ and $G_2$ respectively, we may find a family of differential operators 
$\partial^{\alpha,\beta}_x:=\partial^{\alpha}_{x_1}\partial^{\beta}_{x_2},$
such that the following Taylor's formula holds 
$$f(x)=\sum_{|\alpha|<N}\sum_{|\beta|<N}\frac{1}{\alpha!\beta!}q^{\alpha,\beta}(x^{-1})\partial_x^{\alpha,\beta}f(e)+
\sum_{|\alpha+\beta|=N, |\alpha|\geq N\vee |\beta|\geq N}\frac{1}{\alpha!\beta!}q^{\alpha,\beta}(x^{-1})f_{\alpha,\beta}(x), \ \ \ x\in G,$$ 
for each $f\in C^{\infty}(G)$, where $q^{\alpha,\beta}(x):=p(x)^{\alpha}r(x)^{\beta}=p_1(x)^{\alpha_1}\cdots p_{n_P}(x)^{\alpha_{n_P}}r_1(x)^{\beta_1}\cdots r_{n_R}(x)^{\beta_{n_R}}$.\\
The differential operators $\partial_x^{\alpha,\beta}$ are chosen to satisfy $\partial_{x_1}^{\alpha}(p(x)^{\alpha})=\partial_{x_2}^{\beta}(r(x)^{\beta})=1$ for all 
$\alpha,\beta$, such that $|\alpha|=|\beta|=1$. We fix as the basis of the Lie algebra $\mathfrak{g}$ the dual basis 
$(\partial_{x_{1,1}},\cdots,\partial_{x_{1,n_1}},\partial_{x_{2,1}},\cdots,\partial_{x_{2,n_2}})$ of $(dp_1(e),\cdots,dp_{n_1}(e),dr_1(e),\cdots,dr_{n_2}(e))$, 
where the choice of $(p_1,\cdots,p_{n_1},r_1,\cdots,r_{n_2})$ is made in such a way that $\mathrm{rank}(dp_1(e),\cdots,dp_{n_1}(e))=n_1$ and $\mathrm{rank}(dr_1(e),\cdots,dr_{n_2}(e))=n_2$.\ \\
The above formula can be derived by applying Taylor's formula twice, that is, first with respect to the variable $x_1$ using the functions $q^{\alpha,0}(x)=q^{\alpha,0}(x_1)=p(x_1)^{\alpha}$, 
and then expanding again with respect to $x_2$ and using $q^{0,\beta}(x)=q^{0,\beta}(x_2)=r(x_2)^{\beta}$.

\section{Bisingular symbols on \texorpdfstring{$G=G_1\times G_2$}{G=G1xG2}}\label{sec3}
Here, following \cite{FP}, we define the so-called $(\rho,\delta)$-classes of \textit{bisingular} pseudodifferential symbols on the product of compact Lie groups, where $0\leq \delta\leq \rho\leq 1$. The classes we introduce here represent the nontrivial generalization of the $(1,0)$-classes defined in \cite{FP}. The key role of the general $(\rho,\delta)$-classes becomes evident especially when dealing with problems related to hypoellipticity and with the existence of parametrices. Indeed, even in the Euclidean case, parametrices usually belong to $(\rho,\delta)$-classes with $(\rho,\delta)\neq (1,0)$. Corresponding to these classes, we will develop a  pseudodifferential calculus in Section \ref{sec5}.

Recall that $x=(x_1,x_2)$ is an element of $G=G_1\times G_2$, and $\xi=\xi_1\otimes\xi_2$ is an element of $\widehat{G}$, where $\xi_i\in\widehat{G_i},$ for $ i=1,2$. 
Using the definitions above and fixing the families $P$ and $R$, we define $\Delta_1^{\alpha}:=\Delta_P^{\alpha},\ \Delta_2^{\beta}:=\Delta_{R}^{\beta},$ 
$\partial_1^{\alpha}:=\partial^{\alpha}_{x_1}=\partial_{x_{1,1}}^{\alpha_1}\cdots\partial_{x_{1,n_1}}^{\alpha_{n_1}}$ and $\partial^{\beta}_{2}:=\partial^{\beta}_{x_2}
=\partial^{\beta_1}_{x_{2,1}}\cdots\partial^{\beta_{n_2}}_{x_{2,n_2}}$ as above. We set $\partial^{\alpha,\beta}:=\partial_{x_1}^{\alpha}\partial_{x_2}^{\beta}$, and, analogously, for the difference operators, 
$\Delta^{\alpha,\beta}=\Delta_{1}^{\alpha}\Delta_2^{\beta}$.\ \\
Recall also that, given a continuous linear operator $A:C^{\infty}(G)\rightarrow\mathcal{D}'(G)$, its matrix-valued symbol $\sigma_A(x,\xi)\in\mathbb{C}^{d_{\xi}\times d_{\xi}}$ is given by 
$$\sigma_A(x,\xi)=\xi^*(x)(A\xi)(x),$$ 
and the equality 
$$Af(x)=\sum_{\xi\in\widehat{G}}d_{\xi}\mathrm{Tr}(\xi(x)\sigma_A(x,\xi)\hat{f}(\xi)),\ \ \ f\in C^{\infty}(G),$$ 
holds in the sense of distributions.

Now we are ready to give the definition of bisingular $(\rho,\delta)$-classes of symbols.

\begin{definition}
Let $(m_1,m_2)\in\mathbb{R}^2$ and let $\rho=(\rho_1,\rho_2),\delta=(\delta_1,\delta_2)\in\mathbb{R}^2$ be such that $0\leq\delta_i\leq\rho_i\leq 1$ for $i=1,2$.
The class of bisingular symbols of (bi)order  $(m_1,m_2)$ with respect to the parameters $\rho,\delta$, is the set $S^{m_1,m_2}_{\rho,\delta}(G\times\widehat{G})$ of all functions 
$a:G\times\widehat{G}\rightarrow\bigcup_{[\xi]\in\widehat{G}}\mathbb{C}^{d_{\xi}\times d_{\xi}}$ that are smooth in $x\in G$, and such that, for all multi-indices  
$(\alpha_1,\alpha_2)\in\mathbb{N}_0^{n_P}\times\mathbb{N}_0^{n_R}$ and $(\beta_1,\beta_2)\in\mathbb{N}_0^{n_1}\times\mathbb{N}_0^{n_2}$,
$$\lVert\partial_{x_1}^{\beta_1}\partial_{x_2}^{\beta_2}\Delta_{1}^{\alpha_1}\Delta_2^{\alpha_2}a(x_1,x_2,\xi_1,\xi_2)\rVert_{\mathcal{L}(\mathcal{H}_{\xi})}
\leq C_{\alpha_1,\alpha_2,\beta_1,\beta_2}\langle\xi_1\rangle^{m_1-\rho_1|\alpha_1|+\delta_1|\beta_1|}\langle\xi_2\rangle^{m_2-\rho_2|\alpha_2|+\delta_2|\beta_2|},$$ where $\lVert\cdot\rVert_{\mathcal L(\mathcal H_{\xi})}$ is the operator norm on the representation space $\mathcal H_{\xi}$ relative to $\xi$.
A symbol is said to be smoothing when it belongs to $S^{-\infty,-\infty}(G\times\widehat{G}):=\bigcap_{(m_1,m_2)\in\mathbb{R}^2}S^{m_1,m_2}_{\rho,\delta}(G\times\widehat{G})$. 
It is easy to check that $S^{-\infty,-\infty}(G\times\widehat{G})$ does not depend on $\rho$ or $\delta$.

For symbols valued into a (finite dimensional) vector space $V$ we shall use the notation $S^{m_1,m_2}_{\rho,\delta}(G\times\widehat{G};V)$. (Hence, $S^{m_1,m_2}_{\rho,\delta}(G\times\widehat{G};V)=S^{m_1,m_2}_{\rho,\delta}(G\times\widehat{G})\otimes V$.)
\end{definition}

The space $S^{m_1,m_2}_{\rho,\delta}(G\times\widehat{G})$ is a Fréchet space equipped with the following family of seminorms

\begin{align*}
\lVert \sigma\rVert_{S^{m_1,m_2}_{\rho,\delta},a,b}:=\max_{\begin{matrix}_{
    |\alpha_1|\leq a_1,|\alpha_2|\leq a_2,}\\_{ |\beta_1|\leq b_1, |\beta_2|\leq b_2}
\end{matrix}}
\sup_{(x,\xi)\in G\times\widehat{G}}\langle\xi_1&\rangle^{-m_1+\rho_1|\alpha_1|-\delta_1|\beta_1|}\langle\xi_2\rangle^{-m_2+\rho_2|\alpha_2|-\delta_2|\beta_2|}\times \\
&\hspace{3cm}
\times \lVert\partial_{x_1}^{\beta_1}\partial_{x_2}^{\beta_2}\Delta_1^{\alpha_1}\Delta_2^{\alpha_2}\sigma(x,\xi)\rVert_{\mathcal{L}(\mathcal{H}_{\xi})},
\end{align*}
where $a=(a_1,a_2),b=(b_1,b_2)\in\mathbb{N}^2_0$.
If $x\in G$ is fixed, we will use the following notation
\begin{align*}
\lVert \sigma(x,\cdot)\rVert_{S^{m_1,m_2}_{\rho,\delta},a,b}:=\max_{\begin{matrix}_{
    |\alpha_1|\leq a_1,|\alpha_2|\leq a_2,}\\_{ |\beta_1|\leq b_1, |\beta_2|\leq b_2}
\end{matrix}}
\sup_{\xi\in\widehat{G}}\langle\xi_1&\rangle^{-m_1+\rho_1|\alpha_1|-\delta_1|\beta_1|}\langle\xi_2\rangle^{-m_2+\rho_2|\alpha_2|-\delta_2|\beta_2|}\times
\\ &\hspace{3cm}\times\lVert\partial_{x_1}^{\beta_1}\partial_{x_2}^{\beta_2}\Delta_1^{\alpha_1}\Delta_2^{\alpha_2}\sigma(x,\xi)\rVert_{\mathcal{L}(\mathcal{H}_{\xi})}.\end{align*}
It is easy to check that if $m_i\leq m'_i,\  \rho_i\geq \rho'_i$ and $\delta_i\leq\delta'_i$, (where $0\leq\delta_i\leq\rho_i\leq 1$ and $0\leq\delta'_i\leq\rho'_i\leq 1$) for $i=1,2$, then the inclusion
$S^{m_1,m_2}_{\rho,\delta}(G\times\widehat{G})\subset S^{m'_1,m'_2}_{\rho',\delta'}(G\times\widehat{G})$ is true, where $\rho=(\rho_1,\rho_2),\rho'=(\rho'_1,\rho'_2),\delta=(\delta_1,\delta_2),\delta'=(\delta'_1,\delta'_2).$
\medskip

In the following proposition, we state some elementary properties of bisingular $(\rho,\delta)$-classes.
\begin{proposition}
Bisingular $(\rho,\delta)$-classes satisfy the following properties:
\begin{enumerate}
\item If $\sigma\in S^{m_1,m_2}_{\rho,\delta}(G\times\widehat{G})$, then for any 
$\alpha=(\alpha_1,\alpha_2)\in\mathbb{N}_0^{n_P}\times\mathbb{N}_0^{n_R}$ and $\beta=(\beta_1,\beta_2)\in \mathbb{N}_0^{n_1}\times\mathbb{N}_0^{n_2}$, 
the symbol $\partial_x^{\beta}\Delta^{\alpha}\sigma$ belongs to the class $S^{m_1-\rho_1|\alpha_1|+\delta_1|\beta_1|,m_2-\rho_2|\alpha_2|+\delta_2|\beta_2|}_{\rho,\delta}(G\times\widehat{G})$. 
Moreover, for every $a=(a_1,a_2),b=(b_1,b_2)\in\mathbb{N}_0^2$,    
$$\lVert \partial_x^{\beta}\Delta^{\alpha}\sigma\rVert_{S^{m_1-\rho_1|\alpha_1|+\delta_1|\beta_1|,m_2-\rho_2|\alpha_2|+\delta_2|\beta_2|}_{\rho,\delta},a,b}\lesssim_{a,b,\alpha,\beta,m}
\lVert\sigma\rVert_{S_{\rho,\delta}^{m_1,m_2},(a_1+|\alpha_1|,a_2+|\alpha_2|),(b_1+|\beta_1|,b_2+|\beta_2|)}.$$
\item If $\sigma\in S_{\rho,\delta}^{m_1,m_2}(G\times\widehat{G})$, then the symbol $\,^t\bar\sigma=\sigma^*$ belongs to $S^{m_1,m_2}_{\rho,\delta}(G\times\widehat{G})$ 
and for every $a=(a_1,a_2),b=(b_1,b_2)\in\mathbb{N}_0^2$,  
$$\lVert\sigma^*\rVert_{S^{m_1,m_2}_{\rho,\delta},a,b}=\lVert\sigma\rVert_{S^{m_1,m_2}_{\rho,\delta},a,b}.$$
\item If $\sigma_1\in S^{m_1,m_2}_{\rho,\delta}(G\times\widehat{G})$ and $\sigma_2\in S^{m'_1,m_2'}_{\rho,\delta}(G\times\widehat{G})$, then the symbol $\sigma=\sigma_1\sigma_2$ belongs to 
$S^{m_1+m'_1,m_2+m'_2}_{\rho,\delta}(G\times\widehat{G})$ and, for all $a=(a_1,a_2),b=(b_1,b_2)\in\mathbb{N}_0^2$, 
$$\lVert\sigma\rVert_{S^{m_1+m'_1,m_2+m'_2}_{\rho,\delta},a,b}\lesssim_{a,b,m_i,m_i'}\lVert\sigma_1\rVert_{S^{m_1,m_2}_{\rho,\delta},a,b} \lVert\sigma_2\rVert_{S^{m'_1,m'_2}_{\rho,\delta},a,b}.$$
\end{enumerate}
\end{proposition}
To each matrix-valued symbol $a\in S^{m_1,m_2}_{\rho,\delta}(G\times \widehat{G})$ we associate an operator $\mathrm{Op}(a)$ by means of the following quantization formula
$$\mathrm{Op}(a)f(x):=\sum_{\xi\in\widehat{G}}d_{\xi}\mathrm{Tr}\Bigl(\xi(x)a(x, \xi)\hat{f}(\xi)\Bigr)\hspace{6cm}$$
$$\hspace{4cm}=
\sum_{\xi=\xi_1\otimes\xi_2\in\widehat{G}}d_{\xi_1}d_{\xi_2}\mathrm{Tr}\Bigl((\xi_1\otimes\xi_2)(x)a(x,\xi_1,\xi_2)\hat{f}(\xi_1\otimes\xi_2)\Bigr),$$ 
where $f\in C^{\infty}(G)$. We denote by $L^{m_1,m_2}_{\rho,\delta}(G)$ the class of operators
obtained by quantizing symbols in the class $S^{m_1,m_2}_{\rho,\delta}(G\times\widehat{G})$. These operators are called
bisingular operators of bi-order $(m_1,m_2)\in\mathbb{R}^2$ and parameters $\rho=(\rho_1,\rho_2),\delta=(\delta_1,\delta_2)$ on $G=G_1\times G_2$.\\
Moreover, with each symbol $a\in S^{m_1,m_2}_{\rho,\delta}(G\times\widehat{G})$, we associate the following maps
$$G_1\times\widehat G_1\ni (x_1,\xi_1)\mapsto a(x_1,x_2,\xi_1,D_2)\in L^{m_2}_{\rho_2,\delta_2}(G_2),$$
$$G_2\times \widehat G_2\ni(x_2,\xi_2)\mapsto a(x_1,x_2,D_1,\xi_2)\in L^{m_1}_{\rho_1,\delta_1}(G_1),$$
where $L^{m_1}_{\rho_1,\delta_1}(G_1)$ and $L^{m_2}_{\rho_2,\delta_2}(G_2)$ are classes of operators on $G_1$ and $G_2$ respectively, obtained by
means of the quantization formulas
$$a(x_1,x_2,\xi_1,D_2)\psi(x_2)=\sum_{\xi_2\in\widehat G_2}d_{\xi_2}\mathrm{Tr}\Bigl(\bigl(\mathrm{I}_{d_{\xi_1}}\otimes\xi_2(x_2)\bigr)a(x_1,x_2,\xi_1,\xi_2)\,
\bigl(\mathrm{I}_{d_{\xi_1}}\otimes \hat{\psi}(\xi_2)\bigr)\Bigr),\ \ \ \psi\in C^{\infty}(G_2)$$
$$a(x_1,x_2,D_1,\xi_2)\varphi(x_1)=\sum_{\xi_1\in\widehat G_1}d_{\xi_1}\mathrm{Tr}\Bigl(\bigl(\xi_1(x_1)\otimes\mathrm{I}_{d_{\xi_2}}\bigr)a(x_1,x_2,\xi_1,\xi_2)\,
\bigl(\hat{\varphi}(\xi_1)\otimes \mathrm{I}_{d_{\xi_2}}\bigr)\Bigr),\ \ \ \varphi\in C^{\infty}(G_1).$$
The symbol $a\in S^{m_1,m_2}_{\rho,\delta}(G\times\widehat{G})$ is uniquely determined by
one of these maps.\ \\

Given a continuous linear operator $A:C^{\infty}(G)\rightarrow\mathcal{D}'(G)$, its right convolution kernel, denoted by $R_A\in \mathcal{D}'(G\times G)$, is defined via
$$A\varphi(x)=\int_G\varphi(y)R_A(x,y^{-1}x)dy=(R_A(x,\cdot)\ast\varphi)(x).$$ For an operator $A\in L^{m_1,m_2}_{\rho,\delta}(G)$, its symbol $a\in S^{m}_{\rho,\delta}(G\times\widehat{G})$
is the formal Fourier transform of the right convolution kernel associated with $A$, namely, $$a(x,\xi):=(\mathcal{F}_{y\rightarrow \xi}R_A)(x,\xi),$$ and the kernel can be written as 
$$R_A(x,y):=\sum_{\xi\in\widehat{G}}d_{\xi}\mathrm{Tr}\Bigl(\xi(y)a(x,\xi)\Bigr),\ \ \ y=(y_1,y_2)\in G=G_1\times G_2.$$
For fixed $(x_1,\xi_1)\in G_1\times\widehat G_1$, we can write the operator $a(x_1,x_2,\xi_1,D_2)$ as follows
$$a(x_1,x_2,\xi_1,D_2)\varphi(x_2)=(R_a^2(x_1,x_2,\xi_1,\cdot)\ast_{G_2}\varphi)(x_2),$$ 
where $R_a^2$ is the right-convolution kernel obtained by quantizing the second variable only, namely
$$R_a^2(x_1,x_2,\xi_1,y_2):=\sum_{\xi_2\in\widehat{G}_2}d_{\xi_2}\mathrm{Tr}\Bigl((I_{\xi_1}\otimes\xi_2(y_2))a(x_1,x_2,\xi_1,\xi_2)\Bigr).$$
Analogously, for fixed $(x_2,\xi_2)\in G_2\times\widehat G_2,$ the operator $a(x_1,x_2,D_1,\xi_2)$ can be written as
$$a(x_1,x_2,D_1,\xi_2)\varphi(x_1)=(R_a^1(x_1,x_2,\cdot,\xi_2)\ast\varphi)(x_1),$$
where
$$R_a^1(x_1,x_2,y_1,\xi_2)=\sum_{\xi_1\in\widehat G_1}d_{\xi_1}\mathrm{Tr}\Bigl((\xi_1(y_1)\otimes I_{\xi_2})a(x_1,x_2,\xi_1,\xi_2)\Bigr).$$
Since irreducible representations are orthogonal, one easily finds
$$a(x_1,x_2,\xi_1,\xi_2)=\int_{G_1}R_a^1(x_1,x_2,y_1,\xi_2)(\xi_1(y_1)\otimes I_{\xi_2})dy_1\hspace{6cm}$$
$$\hspace{5cm}=\int_{G_2}R_a^2(x_1,x_2,\xi_1,y_2)(I_{\xi_1}\otimes \xi_2(y_2))dy_2.$$

\begin{definition}
Given symbols $a\in S^{m_1,m_2}_{\rho,\delta}(G\times\widehat{G}),b\in S^{m'_1,m'_2}_{\rho,\delta}(G\times\widehat G)$, we will denote by $(a\circ_1 b)(x_1,x_2,\xi_1,\xi_2)$ and 
$(a\circ_2b)(x_1,x_2,\xi_1,\xi_2)$ the symbols in $S^{m_1+m'_1,m_2+m'_2}_{\rho,\delta}(G\times\widehat{G})$ associated with the operators
$$(a\circ_1b)(x_1,x_2,D_1,\xi_2)\varphi(x_1)=a(x_1,x_2,D_1,\xi_2)b(x_1,x_2,D_1,\xi_2)\varphi(x_1),\ \ \ \varphi\in C^{\infty}(G_1),$$
and 
$$(a\circ_2b)(x_1,x_2,\xi_1,D_2)\psi(x_2)=a(x_1,x_2,\xi_1,D_2)b(x_1,x_2,\xi_1, D_2)\psi(x_2),\ \ \ \psi\in C^{\infty}(G_2).$$
\end{definition}

By considering Taylor expansions of right convolution kernels, one shows that
$$(a\circ_1b)(x_1,x_2,\xi_1,\xi_2)\sim\sum_{|\alpha_1|\geq 0}\frac{1}{\alpha_1!}\Delta^{\alpha_1,0}a(x,\xi)\partial^{\alpha_1,0}b(x,\xi),$$
and
$$(a\circ_2b)(x_1,x_2,\xi_1,\xi_2)\sim\sum_{|\alpha_2|\geq 0}\frac{1}{\alpha_2!}\Delta^{0,\alpha_2}a(x,\xi)\partial^{0,\alpha_2}b(x,\xi),$$
which means that for all $N\in\mathbb{N}$, there are suitable $b_{\alpha_1},b_{\alpha_2}\in S^{m'_1,m'_2}_{\rho,\delta}(G\times\widehat{G})$ such that the reminders
\begin{align*}r_N^1(x,\xi):=&(a\circ_1b)(x_1,x_2,\xi_1,\xi_2)-\sum_{|\alpha_1|<N}\frac{1}{\alpha_1!}\Delta^{\alpha_1,0}a(x,\xi)\partial^{\alpha_1,0}b(x,\xi)=\\=&
\sum_{|\alpha_1|=N}\frac{1}{\alpha_1!}\Delta^{\alpha_1,0}a(x,\xi)b_{\alpha_1}(x,\xi)\end{align*} and 
\begin{align*} r_N^2(x,\xi):=&(a\circ_2b)(x_1,x_2,\xi_1,\xi_2)-\sum_{|\alpha_2|<N}\frac{1}{\alpha_2!}\Delta^{0,\alpha_2}a(x,\xi)\partial^{0,\alpha_2}b(x,\xi)=\\=&\sum_{|\alpha_2|=N}\frac{1}{\alpha_2!}\Delta^{0,\alpha_2}a(x,\xi)b_{\alpha_2}(x,\xi)\end{align*} belong to the classes $S^{m_1+m_1'-(N-1)(\rho_1-\delta_1),m_2+m_2'}_{\rho,\delta}(G\times\widehat G)$ and $S^{m_1+m_1',m_2+m_2'-(N-1)(\rho_2-\delta_2)}_{\rho,\delta}(G\times\widehat G)$, respectively.
\medskip

Observe that, for a symbol $a\in S_{\rho,\delta}^{m'_1,m'_2}(G\times\widehat{G})$ and $(x_2,\xi_2)\in G_2\times\widehat G_2$, we have that $\mathrm{Op}(a_{(x_2,\xi_2)})(x_1,D_1)$ $:=a(x_1,x_2,D_1,\xi_2)\in L^{m_1}_{\rho_1,\delta_1}(G_1)$ is an operator on $G_1$, and its adjoint, 
denoted by $\mathrm{Op}(a_{(x_2,\xi_2)})^{*_1}(x_1,D_1):=a(x_1,x_2,D_1,\xi_2)^{*_1}$, is the operator satisfying 
$$(\mathrm{Op}(a_{(x_2,\xi_2)})^{\ast_1}u,v)_{L^2(G_1)}=(u,\mathrm{Op}(a_{(x_2,\xi_2)})v)_{L^2(G_1)},$$ for every $u,v\in\mathcal{D}(G_1).$
Analogously, for all 
$(x_1,\xi_1)\in G_1\times\widehat{G}_1$, we have that the adjoint operator of $\mathrm{Op}(a_{(x_1,\xi_1)})(x_2,D_2):=a(x_1,x_2,\xi_1,D_2)\in L^{m_2}_{\rho_2,\delta_2}(G_2)$, that we denote by $\mathrm{Op}(a_{(x_1,\xi_1)})^{*_2}(x_2,D_2):=a(x_1,x_2,\xi_1,D_2)^{*_2}$, is defined by 
$$(\mathrm{Op}(a_{x_1,\xi_1})^{\ast_2}u,v)_{L^2(G_2)}=(u,\mathrm{Op}(a_{(x_1,\xi_1)})v)_{L^2(G_2)},\ \ \ u,v\in \mathcal{D}(G_2).$$
The symbols of $\mathrm{Op}(a_{(x_2,\xi_2)})^{\ast_1}$ and $\mathrm{Op}(a_{(x_1,\xi_1)})^{\ast_2}$ are denoted, respectively, by $a_{(x_2,\xi_2)}^{(\ast_1)}$ and $a_{(x_1,\xi_1)}^{(\ast_2)}$.\ \\


\noindent\textbf{Sobolev spaces $H^{s_1,s_2}(G)$.}
\medskip

Next, we introduce the Sobolev spaces suitable to our \textit{bisingular} setting. They come into play in several ways, especially in the continuity properties of pseudodifferential operators.
\medskip

Consider the operator $\mathcal L$ on $G=G_1\times G_2$, defined by $$\mathcal L:=(I_1+\mathcal L_{G_1})\otimes(I_2+\mathcal L_{G_2}),$$ where $\mathcal L_{G_i}$ denotes the positive Laplace operator on $G_i$ and $I_i$ is the identity on $G_i$, $i=1,2$. 
The operator $\mathcal L$ is called \textit{biLaplacian} and plays the role of the Laplacian in our setting. The symbol of the biLaplacian is given by 
$$\sigma_{\mathcal L}(\xi)=\sigma_{\mathcal L}(\xi_1\otimes\xi_2)=\langle\xi_1\rangle^2\langle\xi_2\rangle^2I_{d_{\xi}},$$ 
where, for $i=1,2$, the weight $\langle\xi_i\rangle=(1+\lambda_{\xi_i})^{\frac{1}{2}}$ is defined in terms of the eigenvalue $\lambda_{\xi_i}\geq 0$ of $\mathcal L_{G_i}$ relative to the representation 
$\xi_i\in\widehat G_i$, and $I_{d_{\xi}}\in\mathbb{C}^{d_{\xi}\times d_{\xi}}$ is the identity matrix acting on $\mathcal{H}_{\xi}=\mathcal{H}_{\xi_1}\otimes\mathcal{H}_{\xi_2}$.

\begin{definition}[Bisingular Sobolev space of order $(s_1,s_2)$]\ \\
The Sobolev space of order $(s_1,s_2)\in\mathbb{R}^2$ on $G$ is the space 
$$H^{s_1,s_2}(G):=\{u\in\mathcal{D}'(G);(I_1+\mathcal L_{G_1})^{\frac{s_1}{2}}\otimes(I_2+\mathcal L_{G_2})^{\frac{s_2}{2}}u\in L^2(G)\}$$ 
equipped with the norm 
\begin{align*}
\lVert u\rVert_{s_1,s_2}&:=\bigg(\sum_{\xi\in\widehat{G}}d_{\xi}\langle\xi_1\rangle^{2s_1}\langle\xi_2\rangle^{2s_2}\mathrm{Tr}(\hat{u}(\xi)^*\hat{u}(\xi))\biggr)^{\frac{1}{2}}
=\lVert\langle\xi_1\rangle^{s_1}\langle\xi_2\rangle^{s_2}\hat{u}\rVert_{\ell^2(\widehat{G})}\\
&=:\lVert\hat{u}\rVert_{h^{s_1,s_2}(\widehat{G})},
\end{align*}
where 
$$h^{s_1,s_2}(\widehat{G}):=\{\hat{u}\in\mathcal{F}(\mathcal{D}'(G));\langle\xi_1\rangle^{s_1}\langle\xi_2\rangle^{s_2}\hat{u}(\xi)\in\ell^2(\widehat{G})\}$$ 
and $F\in\ell^2(\widehat{G})$ if and only if $\sum_{\xi\in\widehat{G}}d_{\xi}\lVert F(\xi)\rVert^2_{HS}<+\infty$.
\end{definition}
    
The spaces $h^{s_1,s_2}(\widehat{G})$ are complete with respect to the inner product 
$$(u,v)_{s_1,s_2}:=\sum_{\xi\in\widehat{G}}d_{\xi}\langle\xi_1\rangle^{2s_1}\langle\xi_2\rangle^{2s_2}\mathrm{Tr}(\hat{u}(\xi)^*\hat{v}(\xi)).$$ 
Therefore, the Sobolev spaces $H^{s_1,s_2}(G)$ are also complete.

\section{Kernel estimates}\label{sec4}
Before diving into the development of the \textit{bisingular} pseudodifferential calculus, we first have to go through the study of the kernels of bisingular operators. In particular, kernel estimates are at the core of the pseudodifferential calculus on Lie groups. Hence, our scope here will be to derive the estimates satisfied by the kernels that will be repeatedly applied in Section \ref{sec5} to derive the calculus.
\medskip

To begin with, we give a proposition on some general classical properties of pseudodifferential kernels. The proof is omitted since it follows from easy computations.

\begin{proposition}
Let $\sigma=\sigma(x,\xi)\in S^{m_1,m_2}_{\rho,\delta}(G\times\widehat{G})$ be a smooth symbol with an associated kernel 
$\kappa_{\sigma,x}(y)=\mathcal{F}^{-1}_{\xi\mapsto y}\sigma(x,y)$. For each $\alpha=(\alpha_1,\alpha_2)\in\mathbb{N}_0^{n_P}\times\mathbb{N}_0^{n_R}$ and 
$\beta=(\beta_1,\beta_2)\in\mathbb{N}_0^{n_1}\times\mathbb{N}_0^{n_2}$, the symbol $\partial_x^{\beta}\Delta^{\alpha}_{\xi}\sigma(x,\xi)$ belongs to the symbol class 
$S^{m_1-\rho_1|\alpha_1|+\delta_1|\beta_1|, m_2-\rho_2|\alpha_2|+\delta_2|\beta_2|}_{\rho,\delta}(G\times\widehat{G})$  and its associated kernel is 
$\kappa_{\partial^{\beta}\Delta^{\alpha}\sigma,x}(y)=q^{\alpha}(y)\partial^{\beta}_x\kappa_{\sigma,x}(y).$\ \\
The symbol $\sigma^{(*)}(x,\xi)$ of the operator $\mathrm{Op}(\sigma)^*$ belongs to the same class $S^{m_1,m_2}_{\rho,\delta}(G\times\widehat{G})$ and its associated 
kernel is $\kappa_{\sigma^{(*)},x}(y)=\overline{\kappa_{\sigma,x}}(y^{-1})$.\\
If $\sigma_1$ and $\sigma_2$ are smooth symbols with associated kernels $\kappa_{\sigma_1,x}(y)$ and $\kappa_{\sigma_2,x}(y)$, respectively, then the kernel of the symbol 
$\sigma=\sigma_1\sigma_2$ is $\kappa_{\sigma,x}(y)=\kappa_{\sigma_2,x}\ast\kappa_{\sigma_1,x}(y)$.
\end{proposition}

Now we state two results, a proposition which is in analogy with Proposition 4.1 in \cite{FP}, and a lemma, which is in analogy with Lemma 6.4 in \cite{F} and Lemma  4.3 in \cite{FP}.
They will be very useful in what comes next, that is in the proof of kernel estimates.
\begin{proposition}\label{prop4.2}
For any $(m_1,m_2)\in\mathbb{R}^2$ and multi-indices $(\alpha_1,\alpha_2)\in\mathbb{N}_0^{n_P}\times\mathbb{N}_0^{n_R}$, there exist $d\in\mathbb{N}_0$ and $C>0$ such that, for all 
$f_1,f_2\in C^d([0,+\infty[), \xi=\xi_1\otimes\xi_2\in\widehat{G},$ and $ t_1,t_2\in]0,1[$, we have 
$$\lVert\Delta^{\alpha_1}_1\Delta^{\alpha_2}_2f_1(t_1\lambda_{\xi_1})f_2(t_2\lambda_{\xi_2})\rVert_{\mathcal{L}(\mathcal{H}_{\xi})}$$
$$\leq C\Bigl(t_1^{\frac{m_1}{2}}\langle\xi_1\rangle^{m_1-|\alpha_1|}\sup_{\lambda_{\xi_1}\geq 0,\ell_1=0,\cdots,d}|\partial^{\ell_1}_{\lambda_{\xi_1}}f_1(\lambda_{\xi_1})|\Bigr)
\Bigl(t_2^{\frac{m_2}{2}}\langle\xi_2\rangle^{m_2-|\alpha_2|}\sup_{\lambda_{\xi_2}\geq 0,\ell_2=0,\cdots,d}|\partial_{\lambda_{\xi_2}}^{\ell_2}f_2(\lambda_{\xi_2})|\Bigr),$$ 
in the sense that if the suprema on the right-hand side are finite, then the left-hand side is also finite, and the inequality holds.
\end{proposition}
\begin{proof}
The result follows by considering the identity 
$$\rVert\Delta_1^{\alpha_1}\Delta_2^{\alpha_2}f_1(t_1\lambda_{\xi_1})f_2(t_2\lambda_{\xi_2})\rVert_{\mathcal{L}(\mathcal{H}_{\xi})}
=\lVert\Delta_1^{\alpha_1}f_1(t_1\lambda_{\xi_1})\rVert_{\mathcal{L}(\mathcal{H}_{\xi_1})}\lVert\Delta_2^{\alpha_2}f_2(t_2\lambda_{\xi_2})\rVert_{\mathcal{L}(\mathcal{H}_{\xi_2})},$$
and applying Proposition 6.1 in \cite{F} separately to each term of the right-hand side (see Proposition 4.1 in \cite{FP}).
\end{proof}

\begin{lemma}\label{3.2}
Let $k\in\mathcal{D}'(G)$ be a distribution on $G=G_1\times G_2$ and $n_i=\mathrm{dim}(G_i)$, for $i=1,2$. If $s_i>\frac{n_i}{2}$ for $i=1,2$, then 
$$\lVert k\rVert_{L^2(G)}\lesssim\sup_{\xi=\xi_1\otimes\xi_2\in\widehat{G}}\langle\xi_1\rangle^{s_1}\langle\xi_2\rangle^{s_2}\lVert \hat{k}(\xi)\rVert_{\mathcal{L}(\mathcal{H}_{\xi})}.$$ 
In particular, if there exist $s_1>\frac{n_1}{2}$ and $s_2>\frac{n_2}{2}$ such that the right-hand side is finite, then $k\in L^2(G)$.
\end{lemma}
\begin{proof} 
For $i=1,2$, we consider the kernel of the operator $(I_i+\mathcal L_{G_i})^{-\frac{s_i}{2}}$, which is given by
$$B_{s_i}=\frac{1}{\Gamma(\frac{s_i}{2})}\int_{0}^{+\infty}t_i^{\frac{s_i}{2}-1}e^{-t_i}p_{t_i}^{(i)}dt_i$$ 
where $p_{t_i}^{(i)}:=e^{-t_iL_{G_i}}\delta_{e_{i}},$ for $t_i>0$, is the heat kernel, which is a smooth function on $G_i$ that satisfies 
$$\forall s_i,t_i>0 \ \ \ \int_{G_i}p^{(i)}_{t_i}(x_i)dx_i=1,\ \ \ p_{t_i}^{(i)}(x_i^{-1})=p^{(i)}_{t_i}(x_i), \ \ \text{ and }\ \ p^{(i)}_{t_i}\ast p_{s_i}^{(i)}=p^{(i)}_{t_i+s_i}.$$ 
Moreover, $p_{t_i}^{(i)}$ verifies the following estimates $$|p_{t_i}^{(i)}(x_i)|\leq C_iV_i(\sqrt{t_i})^{-1}e^{-\frac{|x_i|_i^2}{C_it_i}},\ \ x_i\in G_i, t_i>0,$$
$$|X_i^{\alpha_i}p^{(i)}_{t_i(x_i)}|\leq C_i\sqrt{t_i}^{-n_i-|\alpha_i|}e^{-\frac{|x_i|^2_i}{C_it_i}},\ \ \ x_i\in G_i,\ \ 0<t_i\leq 1,$$
where $V_i(r)$ denotes the volume of the ball centered at the neutral element $e_{i}$ and of radius $r$ and $C_i$ is a positive constant. By Lemma A.5 in \cite{F}, if $s_i>\frac{n_i}{2}$, then the kernel 
$B_{s_i}$ is square integrable and the continuous inclusion $H^{s_i}(G_i)\subset C(G_i)$ holds for $i=1,2$.\\
The kernel of the operator $(I_1+\mathcal L_{G_1})^{-\frac{s_1}{2}}\otimes(I_2+\mathcal L_{G_2})^{-\frac{s_2}{2}}$ is 
$$B_{s_1,s_2}(y)=B_{s_1}(y_1)\otimes B_{s_2}(y_2),\ \ \ y=(y_1,y_2)\in G=G_1\times G_2,$$ 
and it is square integrable if $s_1>\frac{n_1}{2}$ and $s_2>\frac{n_2}{2}$, since 
$\lVert B_{s_1,s_2}\rVert_{L^2(G)}=\lVert B_{s_1}\rVert_{L^2(G_1)}\lVert B_{s_2}\rVert_{L^2(G_2)}$ (see Lemma A.3 in \cite{FP}). 
Moreover, we have the continuous inclusion $H^{s_1,s_2}(G)\subset C(G)$, since every $f\in H^{s_1,s_2}(G)$ can be written as the convolution 
$$f=\{((I_1+\mathcal L_{G_1})^{\frac{s_1}{2}}\otimes(I_2+\mathcal L_{G_2})^{\frac{s_2}{2}})f\}\ast B_{s_1,s_2}.$$
For $s_1,s_2>0$, we have $k(y)=((I_1+\mathcal L_{G_1})^{\frac{s_1}{2}}\otimes(1+\mathcal L_{G_2})^{\frac{s_2}{2}})(k\ast(B_{s_1}\otimes B_{s_2}))(y)$ 
which gives 
$$\hat{k}(\xi)=\langle\xi_1\rangle^{s_1}\langle\xi_2\rangle^{s_2}\widehat{B_{s_1}\otimes B_{s_2}}(\xi)\hat{k}(\xi).$$
Hence, for $s_1>\frac{n_1}{2},s_2>\frac{n_2}{2}$, we obtain 
$$\lVert k\rVert^2_{L^2(G)}=\sum_{\xi\in\widehat{G}}d_{\xi}\lVert\hat{k}(\xi)\rVert^2_{HS}\leq \sum_{\xi\in\widehat{G}}d_{\xi}\lVert\widehat{B_{s_1}\otimes B_{s_2}}(\xi)\rVert^2_{HS}
\lVert\langle\xi_1\rangle^{s_1}\langle\xi_2\rangle^{s_2}\hat{k}(\xi)\rVert^2_{\mathcal{L}(\mathcal{H}_{\xi})}$$
$$\leq\lVert B_{s_1,s_2}\rVert^2_{L^2(G)}\sup_{\xi\in\widehat{G}}\langle\xi_1\rangle^{2s_1}\langle\xi_2\rangle^{2s_2}\lVert\hat{k}(\xi)\rVert^2_{\mathcal{L}(\mathcal{H}_{\xi})}
\lesssim\sup_{\xi\in\widehat{G}}\langle\xi_1\rangle^{2s_1}\langle\xi_2\rangle^{2s_2}\lVert\hat{k}(\xi)\rVert^2_{\mathcal{L}(\mathcal{H}_{\xi})},$$ 
and this concludes the proof.
\end{proof}

\begin{corollary}\label{cor4.4}
Let $(m_1,m_2)\in\mathbb{R}^2$, $0\leq\delta_i\leq\rho_i\leq 1$ and $n_i=\mathrm{dim}(G_i)$ for $i=1,2$. 
Let $\sigma\in S^{m_1,m_2}_{\rho,\delta}(G\times\widehat{G})$ be a symbol with associated kernel $\kappa_x(y)$. 
Then, for any $\alpha=(\alpha_1,\alpha_2)\in\mathbb{N}_0^{n_P}\times\mathbb{N}_0^{n_R},$ and $\gamma=(\gamma_1,\gamma_2),\theta=(\theta_1,\theta_2)\in\mathbb{N}_0^{n_1}\times\mathbb{N}_0^{n_2}$ 
such that
$$|\theta_i|+m_i+n_i+\delta_i|\gamma_i|<\rho_i|\alpha_i|,\ \ \ i=1,2,$$ 
the function $z\mapsto\partial_x^{\gamma}\partial_z^{\theta}(q^{\alpha}(z)\kappa_x(z))$ is continuous on $G$ and bounded in the following sense 
$$\sup_{z\in G}|\partial^{\gamma}_x\partial^{\theta}_z(q^{\alpha}(z)\kappa_x(z))|\lesssim_{m_i,\Delta,\gamma_i,\theta_i}\lVert\sigma(x,\cdot)
\rVert_{S^{m_1,m_2}_{\rho,\delta},(|\alpha_1|,|\alpha_2|),(|\gamma_1|,|\gamma_2|)},\ \ \ x\in G.$$
\end{corollary}
\begin{proof}
By Lemma \ref{3.2}, for all $s_i\in\mathbb R$ and $s_i'>\frac{n_i}{2}, i=1,2$, we have 
$$\lVert(I+\mathcal L_{G_1})_{z_1}^{\frac{s_1}{2}}\otimes(I+\mathcal L_{G_2})_{z_2}^{\frac{s_2}{2}}\partial_z^{\theta}(q^{\alpha}(z)\partial^{\gamma}_x\kappa_x(z))\rVert_{L^2(G)}\ \ \ \ \ \ \ \ \ \ \ \ \ \ \ \ \ \ \ \ \ \ \ \ \ \ \ \ \ \ \ \ \ \ \ \ \ \ \ \ \ \ \ \ \ \ \ $$ 
$$\ \ \ \ \ \ \ \ \ \ \ \ \ \ \ \ \ \ \ \ \ \  \lesssim_{s_i,\theta}\lVert(I+\mathcal L_{G_1})^{\frac{s_1+|\theta_1|}{2}}\otimes(I+\mathcal L_{G_2})^{\frac{s_2+|\theta_2|}{2}}(q^{\alpha}(z)\partial^{\gamma}_x\kappa_x(\cdot))\rVert_{L^2(G)}\ \ \ \ \ \ \ \ \ \ \ $$ 
$$\ \ \ \ \ \ \ \ \ \ \ \ \ \ \ \ \ \ \ \ \ \ \ \ \ \ \ \ \ \ \ \ \ \ \ \ \ \ \ \ \ \ \ \ \ \lesssim_{s_i'}\sup_{\xi\in\widehat{G}}\langle\xi_1\rangle^{s_1'+s_1+|\theta_1|}\langle\xi_2\rangle^{s_2'+s_2+|\theta_2|}\lVert\partial_x^{\gamma}\Delta^{\alpha}\sigma(x,\xi)\rVert_{\mathcal{L}(\mathcal{H}_{\xi})}.$$
The latter quantity can be estimated by $\lVert\sigma(x,\cdot)\rVert_{S^{m_1,m_2}_{\rho,\delta}(|\alpha_1|,|\alpha_2|),(|\gamma_1|,|\gamma_2|)}$, whenever 
$$\label{cond}s_i'+s_i+|\theta_i|\leq -m_i+\rho_i|\alpha_i|-\delta_i|\gamma_i|,\ \ \ i=1,2.$$
The condition $|\theta_i|+m_i+n_i+\delta_i|\gamma_i|<\rho_i|\alpha_i|$ yields the existence of $s_i>\frac{n_i}{2}$ such that the inequality above is verified.\\
By Sobolev embedding $H^{s_1,s_2}(G)\hookrightarrow C(G)$, the function $z\mapsto\partial_x^{\gamma}\partial^{\theta}_z(q^{\alpha}(z)\kappa_x(z))$ is continuous on $G$ and its 
$L^{\infty}$-norm is bounded, up to a constant, by $\lVert\sigma(x,\cdot)\rVert_{S^{m_1,m_2}_{\rho,\delta},(|\alpha_1|,|\alpha_2|),(|\gamma_1|,|\gamma_2|)}$. This concludes the proof. 
\end{proof}

Corollary \ref{cor4.4} immediately leads to the following proposition.

\begin{proposition}\label{prop4.5}
Let $\sigma\in S^{m_1,m_2}_{\rho,\delta}(G\times\widehat{G})$, with $0\leq\delta_i\leq\rho_i\leq 1, \rho_i\neq 0, i=1,2$. 
Then its associated kernel $(x,y)\mapsto\kappa_x(y)$ is a smooth function on $G\times(G\setminus S)$, where 
$$S:=\{(e_1,x_2)\in G_1\times G_2\}\cup\{(x_1,e_2)\in G_1\times G_2\}$$ is the set over which the admissible functions $q^{\alpha}$ vanish.
If $\sigma\in S^{-\infty,-\infty}(G\times\widehat{G})$ is smoothing, then its associated kernel $(x,y)\mapsto \kappa_x(y)$ is smooth on $G\times G$. Conversely, if the kernel $\kappa_x(y)$, associated with the symbol $\sigma\in S^{m_1,m_2}_{\rho,\delta}(G\times\widehat{G})$, is smooth on $G\times G$, then $\sigma$ belongs to $S^{-\infty,-\infty}(G\times\widehat{G})$.
\end{proposition}

In order to show some estimates for the kernels, we need to work inside dyadic pieces where the frequencies are localized. In that direction, we have the following lemma.
\begin{lemma}\label{lem4.6}
Let $\chi\in C^{\infty}_0(\mathbb{R})$ be a given function with values in $[0,1]$ and $\chi\equiv 1$ in a neighborhood of 0. Let 
$\sigma\in S^{m_1,m_2}_{\rho,\delta}(G\times\widehat{G})$ and let $\kappa_x$ be its associated kernel. For each $l_1,l_2\in\mathbb{N}$ we define
$$\sigma_{l_1,l_2}(x,\xi):=\sigma(x,\xi)\chi(l_1^{-1}\lambda_{\xi_1})\chi(l_2^{-1}\lambda_{\xi_2}).$$ 
Then $\sigma_{l_1,l_2}\in S^{-\infty,-\infty}(G\times\widehat{G})$ is smoothing and for any $\gamma=(\gamma_1,\gamma_2)\in\mathbb{N}_0^{n_1}\times\mathbb{N}_0^{n_2}$, there exists a constant $C>0$ such that 
$$\lVert\sigma_{l_1,l_2}\rVert_{S^{m_1,m_2}_{\rho,\delta},0,(|\gamma_1|,|\gamma_2|)}\leq C_{G,m_1,m_2,\gamma,\rho,\delta}\lVert\sigma\rVert_{S^{m_1,m_2}_{\rho,\delta},0,(|\gamma_1|,|\gamma_2|)}.$$ 
In addition, the kernel $(x,y)\mapsto \kappa_{l_1,l_2,x}(y)$ associated with $\sigma_{l_1,l_2}$ is smooth on $G\times G$, and for $\beta=(\beta_1,\beta_2)\in\mathbb{N}_0^{n_1}\times\mathbb{N}_0^{n_2},$ the convergence $ \partial^{\beta}\kappa_{l_1,l_2,x}\rightarrow \partial^{\beta}\kappa_x$ in $\mathcal{D}'(G)$ holds uniformly in $x\in G$, as $l_1,l_2\rightarrow+\infty$.
\end{lemma}
\begin{proof}
For every $l_1,l_2\in\mathbb{N}$, the invariant symbol $\chi(l_1^{-1}\lambda_{\xi_1})\chi(l_2^{-1}\lambda_{\xi_2})$ is smoothing because it has compact support in $\xi$. 
The properties of the symbol classes imply that the symbol $\sigma_{l_1,l_2}$ is also smoothing and, from Proposition \ref{prop4.5}, its kernel $(x,y)\mapsto\kappa_{l_1,l_2,x}(y)$ is a smooth function 
on $G\times G$. We shall give the proof of the convergence of the kernels. Setting $s_i=\lceil\frac{n_i}{2}\rceil+1$, for $i=1,2$, we have 
\begin{align*} 
\lVert\partial^{\beta}(\kappa_{l_1,l_2,x}&-\kappa_x)\rVert_{H^{-s_1-m_1-\delta_1|\beta_1|-1,-s_2-m_2-\delta_2|\beta_2|-1}(G)}\\
&=\lVert\partial^{\beta}(\sigma_{l_1,l_2}-\sigma)(x,\cdot)\rVert_{h^{-s_1-m_1-\delta_1|\beta_1|-1,-s_2-m_2-\delta_2|\beta_2|-1}(\widehat{G})}\\
&=\lVert(1-\chi(l_1^{-1}\lambda_{\xi_1})\chi(l_2^{-1}\lambda_{\xi_2}))\partial^{\beta}\sigma(x,\cdot)\rVert_{h^{-s_1-m_1-\delta_1|\beta_1|-1,-s_2-m_2-\delta_2|\beta_2|-1}(\widehat{G})}\\ 
&\lesssim\lVert\langle\xi_1\rangle^{-m_1-\delta_1|\beta_1|-1}\langle\xi_2\rangle^{-m_2-\delta_2|\beta_2|-1}(1-\chi(l_1^{-1}\lambda_{\xi_1})\chi(l_2^{-1}\lambda_{\xi_2}))
\partial^{\beta}\sigma(x,\cdot)\rVert_{h^{-s_1,-s_2}(\widehat{G})}\\
&\lesssim\lVert\langle\xi_1\rangle^{-m_1-\delta_1|\beta_1|-1}\langle\xi_2\rangle^{-m_2-\delta_2|\beta_2|-1}(1-\chi(l_1^{-1}\lambda_{\xi_1})\chi(l_2^{-1}\lambda_{\xi_2}))
\partial^{\beta}\sigma(x,\cdot)\rVert_{L^{\infty}(\widehat{G})}.
\end{align*}
The last inequality above follows from the fact that 
$$\lVert\sigma\rVert_{h^{-s_1,-s_2}(\widehat{G})}\leq C_{s_1,s_2}\lVert\sigma\rVert_{L^{\infty}(\widehat{G})}$$ 
where $C_{s_1,s_2}:= \lVert(1+\lambda_{\xi_1})^{-\frac{s_1}{2}}(1+\lambda_{\xi_2})^{-\frac{s_2}{2}}\rVert_{h^0(\widehat{G})}$ is finite, since $s_i>\frac{n_i}{2}$, $i=1,2$.\\
By hypothesis, there are $\epsilon$ and $c$ such that $0<\epsilon<c$ and the function $\chi$ is identically equal to 1 on $[0,\epsilon]$ and equal to 0 on 
$[c,+\infty[$. Therefore, $(1-\chi(l_1^{-1}\lambda_{\xi_1})\chi(l_2^{-1}\lambda_{\xi_2}))\neq 0$ in the following three cases 
 
\begin{enumerate}
\item $\lambda_{\xi_1}>\epsilon l_1,\lambda_{\xi_2}>\epsilon l_2$,
\item $\lambda_{\xi_1}>\epsilon l_1,\lambda_{\xi_2}\leq \epsilon l_2$,
\item $\lambda_{\xi_1}\leq \epsilon l_1,\lambda_{\xi_2}>\epsilon l_2$.
\end{enumerate} In the first case, we have
\begin{align*}
\lVert\partial^{\beta}&(\kappa_{l_1,l_2,x}-\kappa_x)\rVert_{H^{-s_1-m_1-\delta_1|\beta_1|-1,-s_2-m_2-\delta_2|\beta_2|-1}(G)}\\ &
\leq\max_{\xi\in\widehat{G},\lambda_{\xi_1}>\epsilon l_1,\lambda_{\xi_2}>\epsilon l_2}\lVert(1-\chi(l_1^{-1}\lambda_{\xi_1})\chi(l_2^{-1}\lambda_{\xi_2}))\partial^{\beta}\sigma(x,\xi)
\rVert_{h^{-s_1-m_1-\delta_1|\beta_1|-1,-s_2-m_2-\delta_2|\beta_2|-1}(\widehat{G})}\\
&\leq (1+\epsilon l_1)^{-\frac{1}{2}}(1+\epsilon l_2)^{-\frac{1}{2}}\langle\xi_1\rangle^{-m_1-\delta_1|\beta_1|}\langle\xi_2\rangle^{-m_2-\delta_2|\beta_2|}\lVert\partial^{\beta}\sigma
\rVert_{L^{\infty}(\widehat{G})}\\
&\lesssim (1+\epsilon l_1)^{-\frac{1}{2}}(1+\epsilon l_2)^{-\frac{1}{2}}\lVert\sigma\rVert_{S^{m_1,m_2}_{\rho,\delta},0,(|\beta_1|,|\beta_2|)}.
\end{align*}
This implies uniform convergence $\partial^{\beta}\kappa_{l_1,l_2,x}\xrightarrow{\mathcal{D}'(G)}\partial^{\beta}\kappa_x$, as $l_1,l_2\rightarrow+\infty$. In the second case (and in the third, by reversing the role of $l_1$ and $l_2$), we have 
\begin{align*}&
\lVert\partial^{\beta}(\kappa_{l_1,l_2,x}-\kappa_x)\rVert_{H^{-s_1-m_1-\delta_1|\beta_1|-1,-s_2-m_2-\delta_2|\beta_2|-1}(G)}\\ &
\leq\max_{\xi\in\widehat{G},\lambda_{\xi_1}>\epsilon l_1,\lambda_{\xi_2}\leq\epsilon l_2}\lVert(1-\chi(l_1^{-1}\lambda_{\xi_1})\chi(l_2^{-1}\lambda_{\xi_2}))\partial^{\beta}\sigma(x,\xi)
\rVert_{h^{-s_1-m_1-\delta_1|\beta_1|-1,-s_2-m_2-\delta_2|\beta_2|-1}(\widehat{G})}\\&\leq\max_{\lambda_{\xi_1}>\epsilon l_1,\lambda_{\xi_2}\leq\epsilon l_2}\langle\xi_1\rangle^{-m_1-\delta_1|\beta_1|-1}
\langle\xi_2\rangle^{-m_2-\delta_2|\beta_2|-1}\lVert(1-\chi(l_1^{-1}\lambda_{\xi_1})\chi(l_2^{-1}\lambda_{\xi_2}))\partial^{\beta}\sigma(x,\xi)\rVert_{h^{-s_1,-s_2}(\widehat{G})}\\
&\leq (1+\epsilon l_1)^{-\frac{1}{2}}\langle\xi_1\rangle^{-m_1-\delta_1|\beta_1|}\langle\xi_2\rangle^{-m_2-\delta_2|\beta_2|}\lVert\partial^{\beta}\sigma(x,\xi)\rVert_{L^{\infty}(\widehat{G})}\\
&\lesssim (1+\epsilon l_1)^{-\frac{1}{2}}\lVert\sigma\rVert_{S^{m_1,m_2}_{\rho,\delta},0,(|\beta_1|,|\beta_2|)},
 \end{align*} 
yielding the uniform convergence in $\mathcal{D}'(G)$ and ending the proof.
\end{proof}

\begin{lemma}\label{lem4.7}
Let $\eta\in C^{\infty}_0(G)$ and $\sigma\in S^{m_1,m_2}_{\rho,\delta}(G\times\widehat{G})$, with $0\leq \delta_i\leq\rho_i\leq 1$ for $i=1,2$. For any $t_1,t_2\in]0,1[$,
the symbol $\sigma_{t_1,t_2}$ is defined by $\sigma_{t_1,t_2}(x,\xi):=\sigma(x,\xi)\eta(t_1\lambda_{\xi_1})\eta(t_2\lambda_{\xi_2})$. 
Then, for every $m_1',m_2'\in\mathbb{R}$ and $a,b\in\mathbb{N}_0$, there exists $C=C_{m_1,m_2,m_1',m_2',\gamma,\eta}>0$ such that 
$$\lVert\sigma_{t_1,t_2}\rVert_{S^{m_1',m_2'}_{\rho,\delta}0,(|\gamma_1|,|\gamma_2|)}\leq Ct_1^{\frac{m'_1-m_1}{2}}t_2^{\frac{m_2'-m_2}{2}}\lVert\sigma\rVert_{S^{m_1,m_2}_{\rho,\delta},0,(|\gamma_1|,|\gamma_2|)}.$$
\end{lemma}
\begin{proof}
This follows from Proposition \ref{prop4.2}, together with the Leibniz property for the difference operators (see Lemma 6.8 in \cite{F} and Lemma 4.7 in \cite{FP}).
\end{proof}
    
Recall that if $e_1\in G_1, e_2\in G_2$ are the neutral elements of $G_1$ and $G_2$, respectively, then $S$ denotes the following subset of $G=G_1\times G_2$ 
$$S:=\{(x_1,e_2)\in G; x_1\in G_1\}\cup\{(e_1,x_2)\in G; x_2\in G_2\}\subset G=G_1\times G_2.$$

We are finally ready to state and prove the main theorem of this section about kernel estimates. 
\begin{theorem}\label{3.8}
Let $\sigma\in S^{m_1,m_2}_{\rho,\delta}(G\times\widehat{G})$, with $0\leq\delta_i\leq\rho_i\leq 1, \rho_i\neq 0$ for $i=1,2$ (where $\rho=(\rho_1,\rho_2)$ and $\delta=(\delta_1,\delta_2)$) and let $\kappa_x$ its associated kernel.  
Then, for $n_i=\mathrm{dim}(G_i),i=1,2$, the following estimates hold:
\begin{enumerate}
\item If $n_i+m_i>0$ for $i=1,2$, then there exist $C>0$ and $ a,b\in \mathbb{N}^2_0$ (independent of $\sigma$) such that for every $y\notin S$
$$|\kappa_x(y)|\leq C\lVert\sigma\rVert_{S^{m_1,m_2}_{\rho,\delta},a,b}|y_1|^{-2\frac{n_1+m_1}{\rho_1}}|y_2|^{-2\frac{n_2+m_2}{\rho_2}},\ \ \ \ x\in G.$$
\item If $n_i+m_i=0$ for $i=1,2$, then there exist $C>0$ and $a,b\in\mathbb{N}^2_0$ (independent of $\sigma$) such that for all $y\notin S$ 
$$|\kappa_x(y)|\leq C\lVert\sigma\rVert_{S^{m_1,m_2}_{\rho,\delta},a,b}|\ln|y_1|||\ln|y_2||,\ \ \ x\in G.$$
\item If $n_i+m_i<0$ for $i=1,2$, then $\kappa_x$ is continuous on $G$ and for all $y\notin S$ $$|\kappa_x(y)|\leq C\lVert\sigma\rVert_{S^{m_1,m_2}_{\rho,\delta},0,0},\ \ \ x\in G.$$
\item If $n_i+m_i>0$ and $n_j+m_j=0$ for $i,j\in\{1,2\}, i\neq j$, then there exist $C>0$ and $a,b\in\mathbb{N}^2_0$ (independent of $\sigma$) such that for all $y\notin S$ 
$$|\kappa_x(y)|\leq C\lVert\sigma\rVert_{S^{m_1,m_2}_{\rho,\delta},a,b}|y_i|^{-2\frac{n_i+m_i}{\rho_i}}|\ln|y_j||,\ \ \ x\in G.$$
\item If $n_i+m_i<0$ and $n_j+m_j=0$ for $i,j\in\{1,2\},i\neq j$, then there exist $C>0$ and $\gamma_k\in\mathbb{N}_0^2$ (independent of $\sigma$) being of the form $\gamma_k=(a_k,0)$ 
or $\gamma_k=(0,a_k)$, such that for all $y\notin S$ 
$$|\kappa_x(y)|\leq C\lVert\sigma\rVert_{S^{m_1,m_2}_{\rho,\delta},\gamma_k,0}|\ln|y_j||,\ \ \ x\in G.$$
\item If $n_i+m_i>0$ and $n_j+m_j<0$ for $i,j\in\{1,2\},i\neq j$, then there exist $C>0$ and $\gamma_k\in\mathbb{N}_0^2$ (independent of $\sigma$, of the same form as above) such that for all 
$y\notin S$ 
$$|\kappa_x(y)|\leq C\lVert\sigma\rVert_{S^{m_1,m_2}_{\rho,\delta},\gamma_k,0}|y_i|^{-2\frac{n_i+m_i}{\rho_i}},\ \ \ x\in G.$$
\end{enumerate}
\end{theorem}
\begin{proof}
Let $\eta_0,\eta_1\in C^{\infty}_0(\mathbb R)$ be smooth functions supported in $[-1,1]$ and $[\frac{1}{2},2]$ respectively, taking values in $[0,1]$, and such that $$\forall\lambda>0\ \ \ \sum_{l=0}^{+\infty}\eta_l(\lambda)= 1, \text{ where }\ \  \eta_l(\lambda):=\eta_1(2^{-(l-1)}\lambda),l\geq 1.$$ For each $l_1,l_2\in \mathbb N_0$, we define $\sigma_{l_1,l_2}(x,\xi):=\sigma(x,\xi)\eta_{l_1}(\lambda_{\xi_1})\eta_{l_2}(\lambda_{\xi_2})$ and denote by $\kappa_{l_1,l_2,x}$ the corresponding kernel. Since $\eta_{l_1}(\lambda_{\xi_1})\eta_{l_2}(\lambda_{\xi_2})$ is smoothing, then $\sigma_{l_1,l_2}$ is smoothing too. Moreover, also the mapping $(x,y)\mapsto \kappa_{l_1,l_2,x}(y)= \kappa_x\ast\eta_{l_1}(\mathcal L_{G_1})\eta_{l_2}(\mathcal L_{G_2})\delta_{e_1}\otimes\delta_{e_2}$ is smooth, as $(x,y)\mapsto\kappa_x(y)$ is smooth on $G\times (G\setminus S)$ and $\eta_{l_1}(\mathcal L_{G_1})\eta_{l_2}(\mathcal L_{G_2})\delta_{e_1}\otimes\delta_{e_2}$  is smooth on $G$ (recall that $\mathcal L_{G_1}$ and $\mathcal L_{G_2}$ denote the Laplacians on $G_1$ and $G_2$, respectively). One has the following convergence in $C^{\infty}(G\times (G\setminus S))$  $$\kappa_x(y) =\lim_{N_1,N_2\rightarrow +\infty}\sum_{l_1=0}^{N_1}\sum^{N_2}_{l_2=0}\kappa_{l_1,l_2,x}(y)=\lim_{N_1,N_2\rightarrow+\infty}\biggl(\kappa_x\ast\sum_{l_1=0}^{N_1}\sum_{l_2=0}^{N_2}\eta_{l_1}(\mathcal L_{G_1})\eta_{l_2}(\mathcal L_{G_2})\delta_{e_1}\otimes\delta_{e_2}\biggr)(y),$$ and that the following bound holds for $y\notin S$ $$|\kappa_x(y)|\leq\sum_{l_1,l_2=0}^{+\infty}|\kappa_{l_1,l_2,x}(y)|.$$ By Corollary \ref{cor4.4} and Lemma \ref{lem4.7}, for any given $(\alpha_1,\alpha_2)\in \mathbb N^{n_P}\times\mathbb N^{n_R}$ and for any given $m_1',m_2'\in\mathbb R$ such that $m'_i +n_i<\rho_i|\alpha_i|$, we have the following \begin{equation}\label{eq4.1}\sup_{z\in G}|q^{\alpha_1,\alpha_2}(z)\kappa_{l_1,l_2,x}(z)| \lesssim\sup_{\xi\in\widehat G}\lVert\sigma_{l_1,l_2}(x,\xi)\rVert_{S^{m'_1,m'_2}_{\rho,\delta}(|\alpha_1|,|\alpha_2|),0}\end{equation} $$ \lesssim \lVert\sigma\rVert_{S^{m_1,m_2}_{\rho,\delta}(|\alpha_1|,|\alpha_2|),0}2^{-(l_1-1)\frac{m'_1-m_1}{2}}2^{-(l_2-1)\frac{m'_2-m_2}{2}}.$$
Notice that, for all $z\in G$ and for all $a_1,a_2\in \mathbb N_0$,  $$|z_1|^{a_1}|z_2|^{a_2}\lesssim\sum_{|\alpha_1|=a_1,|\alpha_2|=a_2}|q^{\alpha_1,\alpha_2}(z)|.$$ Therefore, for all $a_1,a_2\in \mathbb N_0$ and $m'_1, m'_2$ such that $m'_i+n_i< \rho_ia_i, i = 1,2$, (\ref{eq4.1}) implies \begin{equation}\label{eq4.2}|z_1|^{a_1}|z_2|^{a_2}|\kappa_{l_1,l_2,x}(z)|\lesssim\lVert\sigma\rVert_{S^{m_1,m_2}_{\rho,\delta}(a_1,a_2),0}2^{l_1\frac{m_1-m'_1}{2}}2^{l_2\frac{m_2-m'_2}{2}}.\end{equation} The previous estimate is meaningful in a neighborhood $U=(U_1 \times G_2)\cup (G_1\times U_2)$ of $S$, where $U_1$ and $U_2$ are geodesic neighborhoods of $e_1$ and $e_2$, respectively. Outside that neighborhood, the estimates in the statement are straightforward, because of the smoothness of the kernel. Since we want to study the behavior of $\kappa_x(y)$ close to the set $S$, we will consider each of the following situations \begin{itemize} \item $|z_1| < 1$ and $|z_2| < 1$; \item $|z_1| < 1$ and $|z_2|\geq 1$ (resp. $|z_1|\geq 1$ and $|z_2| < 1$). \end{itemize}\ \\ 
\textit{Case $n_i+m_i>0$ for all $i= 1,2$.} When $|z_1|<1$ and $|z_2|< 1$, we can choose $l_{0_i}\in\mathbb N_0$ such that $ 2^{-\frac{\rho_i}{2}l_{0_i}}\leq |z_i|\leq 2^{-\frac{\rho_i}{2}l_{0_i}+1}, \ i = 1,2.$ In order to derive the desired estimate we write
$$\sum^{N_1}_{l_1=0}\sum^{N_2}_{l_2=0}=\sum_{l_1\leq l_{0_1},l_2\leq l_{0_2}}+\sum_{l_1\leq l_{0_1},N_2\geq l_2>l_{0_2}}+\sum_{N_1\geq l_1>l_{0_1},l_2\leq l_{0_2}}+\sum_{N_1\geq l_1>l_{0_1},N_2\geq l_2>l_{0_2}}$$ and study the behavior of $\kappa_{l_1,l_2,x}$ separately in the cases \begin{itemize}
    \item $l_i\leq l_{0_i}$ for $i=1,2$, \item $l_i>l_{0_i}$ for $i=1,2$, \item $l_1\leq l_{0_1}$ and $l_2>l_{0_2}$ (resp. $l_2\leq l_{0_2}$ and $l_1> l_{0_1}$).\end{itemize}
When $l_i\leq l_{0_i}$ for $i = 1,2$, from (\ref{eq4.2}) we get $$\sum_{l_1\leq l_{0_1}
,l_2\leq l_{0_2}}|\kappa_{l_1,l_2,x}(z)|\lesssim\lVert\sigma\rVert_{S^{m_1,m_2}_{\rho,\delta}(a_1,a_2),0}|z_1|^{-a_1}2^{l_{1}\frac{m_1-m'_1}{2}}|z_2|^{-a_2}2^{l_{2}\frac{m_2-m'_2}{2}}.$$ Since $m_i+n_i>0$, there exist $a_i\in \mathbb N_0$ and $m'_i\in\mathbb R$, for $i = 1,2$, such that \begin{equation}\label{eq4.3}a_i<2\frac{m_i+n_i}{\rho_i}<a_i+2 \ \ \ \text{ and }\ \ \ \frac{m_i-m'_i}{\rho_i}=2\frac{m_i+n_i}{\rho_i}-a_i>0.\end{equation} Notice that $m_i'+n_i=-m_i-n_i+\rho_ia_i<\rho_i a_i$,  which yields $$\sum_{l_1\leq l_{0_1},
l_2\leq l_{0_2}}|\kappa_{l_1,l_2,x}(z)|\lesssim\rVert\sigma\rVert_{S^{m_1,m_2}_{\rho,\delta}(a_1,a_2),0}|z_1|^{-a_1}2^{l_{0_1}\frac{m_1-m'_1}{2}}|z_2|^{-a_2}2^{l_{0_2}\frac{m_2-m'_2}{2}}$$
$$\lesssim\lVert\sigma\rVert_{S^{m_1,m_2}_{\rho,\delta}(a_1,a_2),0}|z_1|^{-2\frac{m_1+n_1}{\rho_1}}|z_2|^{-2\frac{m_2+n_2}{\rho_2}}.$$ 
When $l_i>l_{0_i}$ ($l_i\leq N_i$) we make a different choice for $a_i$ and $m'_i$ in (\ref{eq4.3}) that we call $a'_i,m''_i$ in order to keep the notation $a_i,m'_i$ for the choices we made in the previous case $l_i\leq l_{0_i}$. We now choose $$a'_i=a_i+2\ \ \ \ \ \ \text{ and } \ \ \ \ \ \ \frac{m_i-m''_i}{\rho_i}=2\frac{m_i+n_i}{\rho_i}-a'_i< 0,\ \ \  i =1,2.$$ 
Notice that $m''_i+n_i=-m_i-n_i+\rho_ia_i'<\rho_ia_i'$ for $i=1,2$. Thus, we have that 
$$|\kappa_{l_1,l_2,x}(z)|\lesssim \sum_{N_1\geq l_1>l_{0_1},N_2\geq l_2>l_{0_2}}\lVert\sigma\rVert_{S^{m_1,m_2}_{\rho,\delta}(a_1',a_2'),0}|z_1|^{-a_1'}2^{l_{0_1}\frac{m_1-m''_1}{2}}|z_2|^{-a_2'}2^{l_{0_2} \frac{m_2-m''_2}{2}}$$
$$\lesssim\rVert\sigma\rVert_{S^{m_1,m_2}_{\rho,\delta}(a_1',a_2'),0}|z_1|^{-2\frac{m_1+n_1}{\rho_1}}|z_2|^{-2\frac{m_2+n_2}{\rho_2}}.$$ 
When $l_1\leq l_{0_1}$ and $l_2>l_{0_2}$  (resp. $l_2\leq l_{0_2}$ and $l_1> l_{0_1}$) we make a different choice of $a_i$ and $m'_i$ that we call $a''_i,m'''_i$ in order to keep the previous notation for the other cases. By choosing $a''_1=a_1,m'''_1=m'_1,a''_2=a'_2$ and $m'''_2=m''_2$ we get, once again from (\ref{eq4.2}), that $$\sum_{l_1\leq l_{0_1},
N_2\geq l_2>l_{0_2}}|\kappa_{l_1,l_2,x}(z)|\lesssim\rVert\sigma\rVert_{S^{m_1,m_2}_{\rho,\delta}(a_1'',a_2''),0}|z_1|^{-a''_1}2^{l_{0_1} \frac{m_1-m'''_1}{2}}|z_2|^{-a''_2}2^{l_{0_2}\frac{m_2-m'''_2}{2}}$$
$$\lesssim\lVert\sigma\rVert_{S^{m_1,m_2}_{\rho,\delta}(a_1'',a_2''),0}|z_1|^{-2\frac{m_1+n_1}{\rho_1}}|z_2|^{-2\frac{m_2+n_2}{\rho_2}}.$$
Similarly, the estimate in the case $l_2\leq l_{0_2}$ and $l_1 > l_{0_1}$ follows by exchanging the role of $l_1$ and $l_2$. Collecting the (four) estimates together we get the desired result (keeping the biggest seminorm) in the case when $|z_1|<1$ and $|z_2|<1$. \ \\
In the case when $|z_1|<1$ and $|z_2|\geq 1$, we can choose $l_{0_1}$ as before, and, once again, split the analysis into the cases $l_1 \leq l_{0_1}$ and $l_1 > l_{0_1}$. We do not split the sum in $l_2$ in this case, so we will make a single choice for $a_2'$ and $m''_2$ in such a way that $a_2'>2\frac{m_2+n_2}{\rho_2}, m''_2=-m_2-2n_2+\rho_2a_2'$ (so that $m''_2+n_2<\rho_2a_2'$ and $m_2<m_2''$ are still verified). By choosing $a_1, m'_1$ as before when $l_1\leq l_{0_1}$ and $a'_1, m''_1$ as before when $l_{1}>l_{0_1}$, we will get the result (again in terms of the biggest seminorm). Finally, the case $|z_1|\geq 1$ and $|z_2|< 1$ is proved as the last one by reversing the role of $z_1$ and $z_2$. Collecting all the estimates above, we obtain the result in terms of the biggest seminorm. \ \\ \\
\textit{Case $m_i +n_i = 0$ for $i = 1,2$.} When $|z_1|< 1$ and $|z_2|<1$ we fix $l_{0_i}$ as before and consider separately the cases as above, depending on the relation between $l_i$ and $l_{0_i}$.\ \\ When $l_i\leq l_{0_i}$ for $i = 1,2$, we choose $a_i = 0, m'_i = -n_i-\rho_i\epsilon_i$, where $\epsilon_i>0$ can be chosen arbitrarily small, for all $i=1,2$. Since $m_i'+n_i<\rho_ia_i$ and $m_i>m_i'$, we get $$\sum_{l_1\leq l_{0_1},l_2\leq l_{0_2}}|\kappa_{l_1,l_2,x}(z)|\lesssim\lVert\sigma\rVert_{S_{\rho,\delta}^{m_1,m_2}(a_1,a_2),0} 2^{l_{0_1}\frac{m_1-m_1'}{2}}2^{l_{0_2}\frac{m_2-m_2'}{2}}\lesssim\lVert \sigma\rVert_{S_{\rho,\delta}^{m_1,m_2}(a_1,a_2),0}|z_1|^{-\epsilon_1}|z_2|^{-\epsilon_2}.$$  Since $\epsilon_i>0$ can be chosen arbitrarily small, we obtain $$\sum_{l_1\leq l_{0_1},l_2\leq l_{0_2}}|\kappa_{l_1,l_2,x}(z)|\lesssim\lVert \sigma\rVert_{S_{\rho,\delta}^{m_1,m_2}(a_1,a_2),0}|\ln|z_1|| |\ln|z_2||.$$
When $l_i > l_{0_i}$ for $i = 1,2$, we choose $a'_i=2$ and $m''_i=m_i+\rho_i(2-\epsilon_i)$ for all $i=1,2$ (where $2>\epsilon_i>0$ is arbitrarily small as before). Notice that $m_i''+n_i<\rho_ia_i'$ and $m_i<m_i''$, so that from (\ref{eq4.2}) with $a'_i,m''_i$, we have
$$\sum_{N_1\geq l_1>l_{0_1},N_2 \geq l_2> l_{0_2}}|\kappa_{l_1,l_2,x}(z)|\lesssim\lVert\sigma\rVert_{S_{\rho,\delta}^{m_1,m_2}(a_1',a_2'),0} |z_1|^{-2}|z_2|^{-2}2^{l_{0_1}\frac{m_1-m_1''}{2}}2^{l_{0_2}\frac{m_2-m_2''}{2}}$$
$$\lesssim\lVert\sigma\rVert_{S_{\rho,\delta}^{m_1,m_2}(a_1',a_2'),0}|z_1|^{-\epsilon_1}|z_2|^{-\epsilon_2}.$$ Again, since $\epsilon_i>0$ can be chosen arbitrarily small, we get $$\sum_{N_1\geq l_1>l_{0_1},N_2 \geq l_2> l_{0_2}}|\kappa_{l_1,l_2,x}(z)|\lesssim\lVert \sigma\rVert_{S_{\rho,\delta}^{m_1,m_2}(a_1',a_2'),0}|\ln|z_1|| |\ln|z_2||.$$
When $l_1\leq l_{0_1}$ and $l_2 > l_{0_2}$ (resp. $l_2\leq l_{0_2}$ and $l_1 > l_{0_1}$), by choosing $a''_1=a_1, m'''_1=m'_1$ and $a''_2=a'_2, m'''_2=m_2''$, we obtain
$$\sum_{l_1\leq l_{0_1},N_2\geq l_2>l_{0_2}}|\kappa_{l_1,l_2,x}(z)|\lesssim\lVert\sigma\rVert_{S^{m_1,m_2}_{\rho,\delta}(a_1'',a_2''),0}|\ln |z_1|||\ln|z_2||.$$ Collecting the estimates together, the result when $|z_1|< 1$ and $|z_2|< 1$ follows. \ \\ When $|z_1|< 1$ and $|z_2|\geq 1$ we fix again $l_{0_1}$ as before. Recall that now we do not split the sum in $l_2$ and that we will make a single choice for $a_2$ and $m'_2$ in (\ref{eq4.2}). Then, using estimate (\ref{eq4.2}) with $a_1$ and $m'_1$ (when $l_1\leq l_{0_1}$), and $a'_1$ and $m''_1$ (when $l_1 > l_{0_1}$) as in the previous case, the result follows by choosing $a_2'=2$ and $m''_2 =m_2+\rho_2(2-\epsilon_2)$ (where $m_2'' +n_2 < \rho_2a_2'$ and $m_2<m_2''$ are still satisfied). The case $|z_1|\geq 1$ and $|z_2| < 1$ is treated as the previous one reversing the roles of $z_1$ and $z_2$. Collecting all the cases above, we get the result in terms of the biggest seminorm.\ \\ \\
\textit{Case $n_i+m_i< 0$.} The estimate in this case is a direct consequence of Corollary \ref{cor4.4}.\ \\ \\
\textit{Case $n_i + m_i > 0, n_j +m_j = 0$ for $i,j \in \{1,2\}, i\neq j$.} Without loss of generality we suppose $n_1+ m_1 > 0$ and $n_2+m_2 = 0$, since the other case is treated analogously. We then combine the strategies used in the cases $n_i + m_i > 0$ for all $i = 1,2$ and $n_i +m_i = 0$ for all $i = 1,2$. \ \\When $|z_1|< 1$ and $|z_2| < 1$ we fix again $l_{0_i}$ such that $|z_i|\sim 2^{-\frac{\rho_i}{2}l_{0_i}} , i = 1,2$. Then, for $l_i\leq l_{0_i}$, we choose $a_i\in \mathbb N_0$ and $m_i\in\mathbb R$, for all $i = 1,2$, such that $$a_1<2\frac{m_1 +n_1}{\rho_1}<a_1+2\ \ \ \text{ and }\ \ \ \frac{m_1-m'_1}{\rho_1}=2\frac{m_1+n_1}{\rho_1}-a_1>0,$$
$$a_2 = 0,\ \ \ \  m'_2 =-n_2-\rho_2\epsilon_2,$$ so that, from (\ref{eq4.2}), we obtain $$\sum_{l_1\leq l_{0_1},l_2\leq l_{0_2}}|\kappa_{l_1,l_2,x}(z)|\lesssim\lVert\sigma\rVert_{S_{\rho,\delta}^{m_1,m_2}(a_1,a_2),0}|z_1|^{-2\frac{m_1+n_1}{\rho_1}}|\ln |z_2||.$$
For $l_i > l_{0_i}$, for all $i = 1,2$, we apply (\ref{eq4.2}) with $a'_1 = a_1 + 2$ and $m''_1$ satisfying the same conditions as $m'_1$ with $a'_1$ in place of $a_1$ (where, recall, $a_1,m'_1$ are the parameters used for $l_i\leq l_{0_i}$), $a'_2 = 2$ and $m''_2 =m_2+\rho_2(2-\epsilon_2)$. We then have $$\sum_{N_1\geq l_1>l_{0_1},N_2\geq l_2>l_{0_2}}|\kappa_{l_1,l_2,x}(z)|\lesssim\lVert\sigma\rVert_{S_{\rho,\delta}^{m_1,m_2}(a_1',a_2'),0}|z_1|^{-2\frac{n_1+m_1}{\rho_1}}|\ln|z_2||.$$
For $l_1\leq l_{0_1}$ and $l_2> l_{0_2}$ ($l_2\leq l_{0_2}$ and $l_1> l_{0_1}$) we repeat the method used before, namely, we choose $a''_1=a_1,m'''_1=m'_1,a''_2=a'_2$, and $m'''_2=m''_2$ in (\ref{eq4.2}) and get
$$\sum_{l_1\leq l_{0_1},
N_2\geq l_2>l_{0_2}}|\kappa_{l_1,l_2,x}(z)|\lesssim\lVert\sigma\rVert_{S_{\rho,\delta}^{m_1,m_2}(a_1'',a_2''),0}|z_1|^{-2\frac{n_1+m_1}{\rho_1}}|\ln|z_2||.$$
Hence, collecting all the estimates we get the result when $|z_1|< 1$ and $|z_2|< 1$. \ \\ When $|z_1|< 1$ and $|z_2|\geq 1$ the proof follows by considering again only the two cases $l_1\leq l_{0_1}$ and $l_1>l_{0_1}$ (here we do not split the sum in $l_2$, so $0\leq l_2 \leq N_2$). Using the same choices as before for $a_1,a'_1,m'_1,m''_1$, and choosing $a_2'$ and $m''_2$ in (\ref{eq4.2}), (where, recall, $a_1,m'_1$ are the parameters used when $l_1\leq l_{0_1}$, while $a'_1, m''_1$ are those used for $l_1> l_{0_1}$), we obtain the desired estimates when $|z_1|< 1$ and $|z_2|\geq 1$. When $|z_1|\geq 1$ and $|z_2| < 1$ the result is proved by reversing the roles of $z_1$ and $z_2$ in the last case.\ \\ \\
\textit{Cases $n_i +m_i < 0$ and $n_j +m_j =0,\ i\neq j$; cases $n_i+m_i > 0$ and $n_j +m_j <0,\  i\neq j$.} These cases can be treated as the last one, that is, by combing the strategies used for the other cases in the different regions $|z_i|< 1,|z_j|\geq 1, i,j = 1,2 \ i\neq j$. 
\end{proof}

\section{Calculus of bisingular pseudodifferential operators}\label{sec5}
Here we state and prove the asymptotic formula for the $(\rho,\delta)$-symbol of the composition of bisingular operators as well as the one for the $(\rho,\delta)$-symbol of the adjoint of a  bisingular operator. 
These formulas constitute what we call \textit{bisingual calculus}, since they give some ``rule of calculations'' for symbols of bisingular operators, of course up to errors. As in the Euclidean and the other group settings, the calculus is needed to attack problems related with PDEs, such as, among many others, solvability, hypoellipticity, existence of parametrices.

We wish to remark that in the derivation of the calculus one of the main differences with the classical Euclidean setting, where the symplectic structure is exploited, is that a deep use of the kernel estimates is required.

\begin{theorem}[Composition formula]\label{4.1}
Let $\sigma_A\in S^{m_1,m_2}_{\rho_,\delta} (G\times\widehat{G})$ and $\sigma_B\in S^{m'_1,m'_2}_{\rho,\delta}(G\times\widehat{G})$, with 
$\rho=(\rho_1,\rho_2),\delta=(\delta_1,\delta_2)$ and $0\leq \delta_i<\rho_i\leq 1$, for $i=1,2$. Let $A:=\mathrm{Op}(a)$ and $B=\mathrm{Op}(b)$ be the corresponding pseudodifferential operators. 
For short, we write $\eta_1=\rho_1-\delta_1$ and $\eta_2=\rho_2-\delta_2$. Then the
symbol $\sigma_{AB}$ of $AB$ is asymptotically given by
$$a\#b(x,\xi):=\sigma_{AB}(x,\xi)\sim \sum_{j\geq 0}c_{m_1+m'_1-j\eta_1,m_2+m'_2-j\eta_2}(x,\xi),$$
where $c_{m_1+m'_1-j\eta_1,m_2+m'_2-j\eta_2}\in S^{m_1+m'_1-j\eta_1,m_2+m'_2-j\eta_2}_{\rho,\delta}(G\times\widehat{G})$ is equal to the sum of three terms
$$d'_{m_1+m'_1-j\eta_1,m_2+m'_2-j\eta_2}(x,\xi)+d''_{m_1+m'_1-j\eta_1,m_2+m'_2-(j+1)\eta_2}(x,\xi)+d'''_{m_1+m'_1-(j+1)\eta_1,m_2+m'_2-j\eta_2}(x,\xi),$$
which are given by the following formulas:

\begin{align*}
d'_{m_1+m'_1-j\eta_1,m_2+m'_2-j\eta_2}(x,\xi)=&
\sum_{|\alpha_1|=|\alpha_2|=j}
\frac{1}{\alpha_1!\alpha_2!}
(\Delta^{\alpha_1,\alpha_2} \sigma_A(x,\xi))\partial^{\alpha_1,\alpha_2}\sigma_B(x,\xi),\\  d''_{m_1+m'_1-j\eta_1,m_2+m'_2-(j+1)\eta_2}(x,\xi)=&\sum_{|\alpha_1|=j}\frac{1}{\alpha_1!}
\biggl(
\Delta^{\alpha_1,0}\sigma_A\circ_{\xi_2}\partial^{\alpha_1,0}\sigma_B(x,\xi) +\\ &\hspace{2cm}- \sum_{|\alpha_2|\leq j}
\frac{1}{\alpha_2!}(\Delta^{\alpha_1,\alpha_2} \sigma_A(x,\xi))\partial^{\alpha_1,\alpha_2} \sigma_B(x,\xi)\biggr),\\
d'''_{m_1+m'_1-(j+1)\eta_1,m_2+m'_2-j\eta_2}(x,\xi)=&\sum_{|\alpha_2|=j}
\frac{1}{\alpha_2!}
\biggl(\Delta^{0,\alpha_2} \sigma_A \circ_{\xi_1} \partial^{0,\alpha_2}\sigma_B(x,\xi)+\\ &\hspace{2cm}-\sum_{|\alpha_1|\leq j}\frac{1}{\alpha_1!}(\Delta^{\alpha_1,\alpha_2} \sigma_A(x,\xi))
\partial^{\alpha_1,\alpha_2}\sigma_B(x,\xi)\biggr).
\end{align*}
In particular, for any given $N\in\mathbb{N}$, the reminder
$$r_N=\sigma_{AB}-\sum_{j<N}c_{m_1+m'_1-j\eta_1,m_2+m'_2-j\eta_2}$$ belongs to the symbol class $S^{m_1+m'_1-N\eta_1,m_2+m'_2-N\eta_2}_{\rho,\delta}(G\times\widehat{G}).$
\end{theorem}
\begin{proof}
Let $A$ and $B$ be the operators above. By using right convolution kernels, we may write
\begin{align*}ABf(x)=& \int_G(Bf)(xz)R_A(x,z^{-1})dz=\int_Gf(xy^{-1})\int_GR_B(xz,yz)R_A(x,z^{-1})dzdy=\\ =&\int_Gf(y)\int_GR_B(xz,y^{-1}xz)R_A(x,z^{-1})dzdy=\int_Gf(y)R_{AB}(x,y^{-1}x)dy,\end{align*} 
where $R_{AB}(x,y):=\int_GR_B(xz,yz)R_A(x,z^{-1})dz.$ Then we compute the symbol of $AB$ 
\begin{align*}\sigma_{AB}(x,\xi)=&\widehat{R_{AB}}(x,\xi)=\int_G\int_GR_A(x,z^{-1})R_B(xz,yz)\xi^*(y)dzdy=\\
=&\int_G\int_GR_A(x,z^{-1})\xi^*(z^{-1})R_B(xz,yz)\xi^*(yz)dzdy.\end{align*}
We now consider the Taylor expansion of $R_B(xz,yz)=R_B(x_1z_1,x_2z_2,y_1z_1,y_2z_2)$ with respect to the first variable at $z_1=e_1$, 
\begin{align*}
R_B(xz,yz)&=\sum_{|\alpha_1|<N}\frac{1}{\alpha_1!}q^{\alpha_1,0}(z_1^{-1},x_2z_2)\partial^{\alpha_1,0}R_B(x_1,x_2z_2,yz)+\\ &+\sum_{|\alpha_1|=N}\frac{1}{\alpha_1!}q^{\alpha_1,0}(z_1^{-1},x_2z_2)(R_B)_{\alpha_1}(x_1z_1,x_2z_2,yz),\end{align*} 
where $q^{\alpha_1,0}(x)=p^{\alpha_1}(x_1)$ is constant with respect to $x_2$. Taking into account that $q^{\alpha_1,0}(x_1,x_2)$ is independent of the second variable (and, analogously, 
$q^{0,\alpha_2}(x)=r^{\alpha_2}(x_2)$ is independent of the first one), we expand the previous quantity with respect to the second variable at $z_2=e_2$ and find 
\begin{align*}
R_B(xz,yz)=&\sum_{|\alpha_1|<N,|\alpha_2|<N}\frac{1}{\alpha_1!\alpha_2!}q^{\alpha_1,\alpha_2}(z_1^{-1},z_2^{-1})\partial^{\alpha_1,\alpha_2}R_B(x_1,x_2,yz)\\
&+\sum_{|\alpha_2|=N,|\alpha_1|<N}\frac{1}{\alpha_1!\alpha_2!}q^{\alpha_1,\alpha_2}(z_1^{-1},z_2^{-1})(\partial^{\alpha_1,0}R_B)_{\alpha_2}(x_1,x_2z_2,yz)\\
&+\sum_{|\alpha_2|<N,|\alpha_1|=N}\frac{1}{\alpha_1!\alpha_2!}q^{\alpha_1,\alpha_2}(z_1^{-1},z_2^{-1})\partial^{0,\alpha_2}(R_B)_{\alpha_1}(x_1z_1,x_2,yz)\\
&+\sum_{|\alpha_2|=N,|\alpha_1|=N}\frac{1}{\alpha_1!\alpha_2!}q^{\alpha_1,\alpha_2}(z_1^{-1},z_2^{-1})(R_B)_{\alpha_1,\alpha_2}(x_1z_1,x_2z_2,yz).
\end{align*}
Therefore, we have 
\begin{align*}&\sigma_{AB}(x,\xi)=\\
&=\sum_{|\alpha_1|<N,|\alpha_2|<N}\frac{1}{\alpha_1!\alpha_2!}\int_{G\times G}\xi^*(z^{-1})q^{\alpha_1,\alpha_2}(z^{-1})R_A(x,z^{-1})\xi^*(yz)\partial^{\alpha_1,\alpha_2}R_B(x,yz)dzdy\\ 
&\hspace{0.5cm}+\sum_{|\alpha_1|<N}\frac{1}{\alpha_1!}\int_{G\times G}\biggl(q^{\alpha_1,0}(z^{-1})R_A(x,z^{-1})\xi^*(z^{-1})\partial^{\alpha_1,0}R_B(x_1,x_2z_2,yz)\xi^*(yz)\\
&\hspace{2.5cm}-\sum_{|\alpha_2|<N}\frac{1}{\alpha_2!}q^{\alpha_1,\alpha_2}(z^{-1})R_A(x,z^{-1})\xi^*(z^{-1})\partial^{\alpha_1,\alpha_2}R_B(x,yz)\xi^*(yz)\biggr)dzdy \\
&\hspace{0.5cm}+\sum_{|\alpha_2|<N}\frac{1}{\alpha_2!}\int_{G\times G}\biggl(q^{0,\alpha_2}(z^{-1})R_A(x,z^{-1})\xi^*(z^{-1})\partial^{0,\alpha_2}R_B(x_1z_1,x_2,yz)\xi^*(yz)\\
&\hspace{2.5cm}-\sum_{|\alpha_1|<N}\frac{1}{\alpha_1!}q^{\alpha_1,\alpha_2}(z^{-1})R_A(x,z^{-1})\xi^*(z^{-1})\partial^{\alpha_1,\alpha_2}R_B(x,yz)\xi^*(yz)\biggr)dzdy\\ & +\sum_{|\alpha_1|=N,|\alpha_2|=N}\frac{1}{\alpha_1!\alpha_2!}\int_{G\times G}q^{\alpha_1,\alpha_2}(z^{-1})R_A(x,z^{-1})\xi^*(z^{-1})(R_B)_{\alpha_1,\alpha_2}(xz,yz)\xi^*(yz)dzdy.\end{align*}
By computing the integrals above, we obtain 
\begin{align*}
&\sigma_{AB}(x,\xi)=\sum_{|\alpha_1|<N,|\alpha_2|<N}\frac{1}{\alpha_1!\alpha_2!}(\Delta^{\alpha_1,\alpha_2}\sigma_A(x,\xi))\partial^{\alpha_1,\alpha_2}\sigma_B(x,\xi)+\\
&
    +\sum_{|\alpha_1|<N}\frac{1}{\alpha_1!}\biggl((\Delta^{\alpha_1,0}\sigma_A\circ_{\xi_2}\partial^{\alpha_1,0}\sigma_B)(x,\xi)-\sum_{|\alpha_2|<N}\frac{1}{\alpha_2!}(\Delta^{\alpha_1,\alpha_2}\sigma_A(x,\xi))
\partial^{\alpha_1,\alpha_2}\sigma_B(x,\xi)\biggr)\\ 
&+\sum_{|\alpha_2|<N}\frac{1}{\alpha_2!}\biggl((\Delta^{0,\alpha_2}\sigma_A\circ_{\xi_1}\partial^{0,\alpha_2}\sigma_B)(x,\xi)-\sum_{|\alpha_1|<N}\frac{1}{\alpha_1!}(\Delta^{\alpha_1,\alpha_2}\sigma_A(x,\xi))
\partial^{\alpha_1,\alpha_2}\sigma_B(x,\xi)\biggr)\\&+\sum_{|\alpha_1|=N,|\alpha_2|=N}\frac{1}{\alpha_1!\alpha_2!}\int_{G\times G}q^{\alpha_1,\alpha_2}(z^{-1})R_A(x,z^{-1})\xi^*(z^{-1})(R_B)_{\alpha_1,\alpha_2}(xz,yz)\xi^*(yz)dzdy.\end{align*}
Rearranging the terms, we get
\begin{align*}
&\sigma_{AB}(x,\xi)=\sum_{|\alpha_1|=|\alpha_2|<N}\frac{1}{\alpha_1!\alpha_2!}(\Delta^{\alpha_1,\alpha_2}\sigma_A(x,\xi))\partial^{\alpha_1,\alpha_2}\sigma_B(x,\xi)+\\
&+\sum_{|\alpha_1|<N}\frac{1}{\alpha_1!}\biggl((\Delta^{\alpha_1,0}\sigma_A\circ_{\xi_2}\partial^{\alpha_1,0}\sigma_B)(x,\xi)-
\sum_{|\alpha_2|\leq|\alpha_1|}\frac{1}{\alpha_2!}\Delta^{\alpha_1,\alpha_2}\sigma_A(x,\xi)\partial^{\alpha_1,\alpha_2}\sigma_{B}(x,\xi)\biggr)\\
&+\sum_{|\alpha_2|<N}\frac{1}{\alpha_2!}\biggl((\Delta^{0,\alpha_2}\sigma_A\circ_{\xi_1}\partial^{0,\alpha_2}\sigma_B)(x,\xi)-
\sum_{|\alpha_1|\leq|\alpha_2|}\frac{1}{\alpha_1!}\Delta^{\alpha_1,\alpha_2}\sigma_A(x,\xi)\partial^{\alpha_1,\alpha_2}\sigma_B(x,\xi)\biggr)\\
&+\sum_{|\alpha_1|=N,|\alpha_2|=N}\frac{1}{\alpha_1!,\alpha_2!}\int_{G\times G}q^{\alpha_1,\alpha_2}(z^{-1})R_A(x,z^{-1})\xi^*(z^{-1})(R_B)_{\alpha_1,\alpha_2}(xz,yz)\xi^*(yz)dzdy\end{align*}
$$=\sum_{j<N}\biggl(d'_{m_1+m_1'-j\eta_1,m_2+m_2'-j\eta_2}+d''_{m_1+m_1'-j\eta_1,m_2+m_2'-(j+1)\eta_2}+d'''_{m_1+m_1'-(j+1)\eta_1,m_2+m_2'-j\eta_2}\biggr)+r_N.$$
Now, we need to show that for every $N\in\mathbb{N}$, the remainder $r_N$ belongs to the bisingular class $S^{m_1+m_1'-N\eta_1,m_2+m_2'-N\eta_2}_{\rho,\delta}(G\times\widehat{G})$, that is, 
$$\sup_{x\in G}\lVert\partial^{\gamma_1,\gamma_2}\Delta^{\beta_1,\beta_2}r_N(x,\xi)\rVert_{\mathcal{L}(\mathcal{H}_{\xi})}
\lesssim\langle\xi_1\rangle^{m_1+m_1'-N\eta_1-\rho_1|\beta_1|+\delta_1|\gamma_1|}\langle\xi_2\rangle^{m_2+m_2'-N\eta_2-\rho_2|\beta_2|+\delta_2|\gamma_2|},$$ 
for any $(\gamma_1,\gamma_2)\in\mathbb{N}_0^{n_1}\times\mathbb{N}_0^{n_2},(\beta_1,\beta_2)\in\mathbb{N}_0^{n_P}\times\mathbb{N}_0^{n_R}$, 
where we recall that $\eta_i=\rho_i-\delta_i$ for $i=1,2$. 
We will consider the simpler case $\gamma_1=\gamma_2=\beta_1=\beta_2=0$. \\ Write $\xi^*(z)=\langle\xi_1\rangle^{-2s_1}\langle\xi_2\rangle^{-2s_2}(I_1+\mathcal L_{G_1})^{s_1}_{z_1}\otimes(I_2+\mathcal L_{G_2})^{s_2}_{z_2}\xi^*(z)$, for some $s_1,s_2\in\mathbb{N}_0$, to be determined. 
Performing the integration in $y$ and then integrating by parts, we get 

\begin{align*}r_N(x,\xi)=&\sum_{|\alpha_1|=N,|\alpha_2|=N}\frac{1}{\alpha_1!\alpha_2!}\int_{G\times G}q^{\alpha_1,\alpha_2}(z^{-1})R_A(x,z^{-1})\xi^*(z^{-1})(\widehat{R_B})_{\alpha_1,\alpha_2}(xz,\xi)dz\\
=&\langle\xi_1\rangle^{-2s_1}\langle\xi_2\rangle^{-2s_2}\sum_{|\alpha_1|=N,|\alpha_2|=N}\sum_{\begin{matrix}
    _{|\epsilon_1|+|\tau_1|=2s_1}\\_{|\epsilon_2|+|\tau_2|=2s_2}
\end{matrix}}c_{\epsilon_1,\epsilon_2,\tau_1,\tau_2}\frac{1}{\alpha_1!\alpha_2!}\times \\ &\hspace{2.5cm}\times\int_{G}
\partial_{z}^{\epsilon_1,\epsilon_2}(q^{\alpha_1,\alpha_2}(z^{-1})R_A(x,z^{-1}))\xi^*(z^{-1})\partial^{\tau_1,\tau_2}_z(\widehat{R_B})_{\alpha_1,\alpha_2}(xz,\xi)dz\\ 
=&\langle\xi_1\rangle^{-2s_1}\langle\xi_2\rangle^{-2s_2}\sum_{|\alpha_1|=N,|\alpha_2|=N}\sum_{\begin{matrix}_{|\epsilon_1|+|\tau_1|=2s_1}\\_{|\epsilon_2|+|\tau_2|=2s_2}
\end{matrix}}c_{\epsilon_1,\epsilon_2,\tau_1,\tau_2}\frac{1}{\alpha_1!\alpha_2!}\times \\ &\hspace{2.5cm}\times\int_{G}\widetilde{\partial}_{z^{-1}}^{\epsilon_1,\epsilon_2}(q^{\alpha_1,\alpha_2}(z^{-1})
R_A(x,z^{-1}))\xi^*(z^{-1})\partial^{\tau_1,\tau_2}_{z'=xz}(\widehat{R_B})_{\alpha_1,\alpha_2}(z',\xi)dz,\end{align*}
where $c_{\epsilon_1,\epsilon_2,\tau_1,\tau_2}$ are constants and, in the last equality, we applied the relation between left-invariant and right-invariant vector fields given by 
$\partial^{\alpha,\beta}(\phi(\cdot^{-1}))(x)=(-1)^{|\alpha|+|\beta|}(\widetilde{\partial}^{\alpha,\beta}\phi)(x^{-1})$. 
Since $(R_B)_{\alpha_1,\alpha_2}$ is the kernel of a symbol in $S^{m_1',m_2'}_{\rho,\delta}(G\times\widehat{G})$, we find 
\begin{align*}\lVert r_N(x,\xi)\rVert_{\mathcal{L}(\mathcal{H}_{\xi})}&\lesssim\sum_{|\alpha_1|=N,|\alpha_2|=N}\sum_{\begin{matrix}_{|\epsilon_1|+|\tau_1|=2s_1}\\_{|\epsilon_2|+|\tau_2|=2s_2}
\end{matrix}}\langle\xi_1\rangle^{-2s_1}\langle\xi_2\rangle^{-2s_2}\frac{1}{\alpha_1!\alpha_2!}\times \\ &\hspace{0.2cm}\times\int_G|\widetilde{\partial}^{\epsilon_1,\epsilon_2}_{z^{-1}}(q^{\alpha_1,\alpha_2}(z^{-1})R_A(x,z^{-1}))|dz\ 
\mathrm{sup}_{z'\in G}\lVert\partial^{\tau_1,\tau_2}_{z'}\widehat{(R_B)}_{\alpha_1,\alpha_2}(z',\xi)\rVert_{\mathcal{L}(\mathcal{H}_{\xi})}\\ &\lesssim\sum_{|\alpha_1|=N,|\alpha_2|=N}\sum_{\begin{matrix}_{|\epsilon_1|+|\tau_1|=2s_1}\\_{|\epsilon_2|+|\tau_2|=2s_2}
\end{matrix}}\langle\xi_1\rangle^{m_1'-2s_1+\delta_1|\tau_1|}\langle\xi_2\rangle^{m_2'-2s_2+\delta_2|\tau_2|}\frac{1}{\alpha_1!\alpha_2!}\times \\ & \hspace{0.2cm}\times\int_G|
\widetilde{\partial}_{z^{-1}}^{\epsilon_1,\epsilon_2}(q^{\alpha_1,\alpha_2}(z^{-1})R_A(x,z^{-1}))|dz\lVert\widehat{(R_B)}_{\alpha_1,\alpha_2}\rVert_{S^{m_1',m_2'}_{\rho,\delta},0,(2s_1,2s_2)}.\end{align*}
Next, we fix $N\in\mathbb{N}$ and show that there exist $\widetilde{N}\in\mathbb{N}$, $N<\widetilde{N}$, such that 
$$ r_{\widetilde{N}}(x,\xi)\in S^{m_1+m_1'-N\eta_1,m_2+m_2'-N\eta_2}_{\rho,\delta}(G\times\widehat{G}),$$ where $\eta_i=\rho_i-\delta_i, i=1,2$.\\ 
By the previous computations, we obtain \begin{align*}\lVert r_{\widetilde{N}}(x,\xi)\rVert_{\mathcal{L}(\mathcal{H}_{\xi})}&\lesssim\sum_{|\alpha_1|=\widetilde{N},|\alpha_2|=\widetilde{N}}\sum_{\begin{matrix}_{|\epsilon_1|+|\tau_1|=2s_1}\\_{|\epsilon_2|+|\tau_2|=2s_2}
\end{matrix}}\langle\xi_1\rangle^{m_1'-2s_1+\delta_1|\tau_1|}\langle\xi_2\rangle^{m_2'-2s_2+\delta_2|\tau_2|}\frac{1}{\alpha_1!\alpha_2!}\times \\ &\hspace{1cm}\times\int_G|
\widetilde{\partial}_{z^{-1}}^{\epsilon_1,\epsilon_2}(q^{\alpha_1,\alpha_2}(z^{-1})R_A(x,z^{-1}))|dz\lVert\widehat{(R_B)}_{\alpha_1,\alpha_2}\rVert_{S^{m_1',m_2'}_{\rho,\delta},0,(2s_1,2s_2)}.\end{align*}
Suppose that the above integral converges, then we want 
\begin{equation}\label{1}\begin{cases}
m_1'+\delta_1|\tau_1|-2s_1\leq m_1'+m_1-N\eta_1\\ m_2'+\delta_2|\tau_2|-2s_2\leq m_2'+m_2-N\eta_2,
\end{cases}\end{equation} 
so that 
$$\langle\xi_1\rangle^{m_1'+\delta_1|\tau_1|-2s_1}\langle\xi_2\rangle^{m_2'+\delta|\tau_2|-2s_2}\leq \langle\xi_1\rangle^{m_1'+m_1-N\eta_1}\langle\xi_2\rangle^{m_2'+m_2-N\eta_2},$$
as desired. We now prove that indeed the integral 
\begin{equation}\label{int2}
\int_{G}|\widetilde{\partial}^{\epsilon_1,\epsilon_2}_{z}(q^{\alpha_1,\alpha_2}(z)R_A(x,z))|dz\end{equation} 
converges. 

According to Theorem \ref{3.8}, we can estimate the function of $z$ inside the integral, up to a constant depending on $\lVert\sigma_{A}\rVert_{S^{m_1,m_2}_{\rho,\delta}}$, by 
\begin{itemize}
\item $|z_1|^{-2\frac{m_1+n_1+|\epsilon_1|-\rho_1\widetilde{N}}{\rho_1}}|z_2|^{-2\frac{m_2+n_2+|\epsilon_2|-\rho_2\widetilde{N}}{\rho_2}} \text{ if } m_j+n_j+|\epsilon_j|-\rho_j\widetilde{N}>0 \text{ for } j=1,2,$
\item $|z_1|^{-2\frac{m_1+n_1+|\epsilon_1|-\rho_1\widetilde{N}}{\rho_1}}|\ln|z_2|| \ \ \ \ \ \ \ \ \ \ 
 \ \ \ \ \ \text{ if } m_1+n_1+|\epsilon_1|-\rho_1\widetilde{N}>0, m_2+n_2+|\epsilon_2|-\rho_2\widetilde{N}=0,$
\item $|\ln|z_1|||z_2|^{-2\frac{m_2+n_2+|\epsilon_2|-\rho_2\widetilde{N}}{\rho_2}} \ \ \ \ \ \ \ \ \ \ \ \ \ \ \ \text{ if } m_1+n_1+|\epsilon_1|-\rho_1\widetilde{N}=0, m_2+n_2+|\epsilon_2|-\rho_2\widetilde{N}>0,$
\item $|z_1|^{-2\frac{m_1+n_1+|\epsilon_1|-\rho_1\widetilde{N}}{\rho_1}} \ \ \ \ \ \ \ \ \ \ \ \ \ \ \ \ \ \ \ \ \ \ \ \ \ \text{ if } m_1+n_1+|\epsilon_1|-\rho_1\widetilde{N}>0, m_2+n_2+|\epsilon_2|-\rho_2\widetilde{N}<0,$
\item $|z_2|^{-2\frac{m_2+n_2+|\epsilon_2|-\rho_2\widetilde{N}}{\rho_2}}  \ \ \ \ \ \ \ \ \ \ \ \ \ \ \ \ \ \ \ \ \ \ \ \ \ \text{ if } m_1+n_1+|\epsilon_1|-\rho_1\widetilde{N}<0, m_2+n_2+|\epsilon_2|-\rho_2\widetilde{N}>0,$
\item $|\ln|z_1|||\ln|z_2|| \ \ \ \ \ \ \ \ \ \ \ \ \ \ \ \ \ \ \ \ \ \ \ \ \  \text{ if } m_j+n_j+|\epsilon_j|-\rho_j\widetilde{N}=0 \text{ for } j=1,2$
\item $|\ln|z_1||  \ \ \ \ \ \ \ \ \ \ \ \ \ \ \ \ \ \ \ \ \ \ \ \ \ \ \ \ \ \ \ \ \ \ \ \ \ \ \ \text{ if } m_1+n_1+|\epsilon_1|-\rho_1\widetilde{N}=0, m_2+n_2+|\epsilon_2|-\rho_2\widetilde{N}<0,$
\item $|\ln|z_2||  \ \ \ \ \ \ \ \ \ \ \ \ \ \ \ \ \ \ \ \ \ \ \ \ \ \ \ \ \ \ \ \ \ \ \ \ \ \ \ \text{ if } m_1+n_1+|\epsilon_1|-\rho_1\widetilde{N}<0, m_2+n_2+|\epsilon_2|-\rho_2\widetilde{N}=0,$
\item $1 \ \ \ \ \ \ \ \ \ \ \ \ \ \ \ \ \ \ \ \ \ \ \ \ \ \ \ \ \ \ \ \ \ \ \ \ \ \ \ \ \ \text{ if } m_j+n_j+|\epsilon_j|-\rho_j\widetilde{N}<0 \text{ for } j=1,2.$
\end{itemize}

Moreover, in order to have the integral (\ref{int2}) convergent, we shall require the conditions 
\begin{equation}\label{3}\begin{cases}
    2(m_1+n_1+|\epsilon_1|-\rho_1\widetilde{N})<\rho_1n_1,\\
    2(m_2+n_2+|\epsilon_2|-\rho_2\widetilde{N})<\rho_2n_2.
\end{cases}
\end{equation}
Since $\delta_i<1$ for $i=1,2$, we may find nonnegative integers $s_1$ and $s_2$ (depending on $N$, which is fixed), such that 
$$\begin{cases}N\eta_1-m_1\leq 2s_1(1-\delta_1),\\
N\eta_2-m_2\leq 2s_2(1-\delta_2),\end{cases}$$ 
so that the conditions in (\ref{1}) hold for any $|\tau_1|\leq 2s_1$ and $|\tau_2|\leq 2s_2$. Then, we choose $\widetilde{N}\in\mathbb{N}$ (depending on $s_1,s_2$) large enough such that 
$$\begin{cases}2s_1+m_1+n_1(1-\frac{\rho_1}{2})<\rho_1\widetilde{N}\\2s_2+m_2+n_2(1-\frac{\rho_2}{2})<\rho_2\widetilde{N},\end{cases}$$ 
so that the conditions in (\ref{3}) are satisfied for any $|\epsilon_1|\leq 2s_1$ and $|\epsilon_2|\leq 2s_2$. 
This shows that the remainder $r_{\widetilde{N}}$ belongs to the symbol class $S^{m_1+m_1'-N\eta_1,m_2+m_2'-N\eta_2}_{\rho,\delta}(G\times\widehat{G})$\\
Now, we must show that $r_{N}\in S^{m_1+m_1'-N\eta_1,m_2+m_2'-N\eta_2}_{\rho,\delta}(G\times\widehat{G})$. Notice that 

\begin{align*}
    r_N(x,\xi)&=\sigma_{AB}(x,\xi)-\sum_{j<N}c_{m_1+m_1'-j\eta_1,m_2+m_2'-j\eta_2}(x,\xi)\\ &=\sigma_{AB}(x,\xi)-\sum_{j<\widetilde{N}}c_{m_1+m_1'-j\eta_1,m_2+m_2'-j\eta_2}(x,\xi)
+\sum_{N\leq j<\widetilde{N}}c_{m_1+m_1'-j\eta_1,m_2+m_2'-j\eta_2}(x,\xi)\\ &=r_{\widetilde{N}}(x,\xi)+\sum_{N\leq j<\widetilde{N}}c_{m_1+m_1'-j\eta_1,m_2+m_2'-j\eta_2}(x,\xi).
\end{align*}
Therefore, due to the inclusions
$$r_{\widetilde{N}}\in S^{m_1+m_1'-N\eta_1,m_2+m_2'-N\eta_2}_{\rho,\delta}(G\times\widehat{G}),$$ and
$$\sum_{N\leq j< \widetilde{N}}c_{m_1+m_1'-j\eta_1,m_2+m_2'-j\eta_2}\in S^{m_1+m_1'-N\eta_1,m_2+m_2'-N\eta_2}_{\rho,\delta}(G\times\widehat{G}),$$ 
we get $r_N\in S^{m_1+m_1'-N\eta_1,m_2+m_2'-N\eta_2}_{\rho,\delta}(G\times\widehat{G})$. Finally, since $N\in\mathbb{N}$ is arbitrary, we get the desired property.
\end{proof}

\begin{theorem}[Formula for the adjoint]
Let $\sigma\in S^{m_1,m_2}_{\rho,\delta}(G\times\widehat{G})$, with $\rho=(\rho_1,\rho_2),\delta=(\delta_1,\delta_2)$ and $0\leq \delta_i< \rho_i\leq 1$, for $i=1,2$. 
Let us denote $\eta_1=\rho_1-\delta_1$ and $\eta_2=\rho_2-\delta_2$. Then the symbol $\sigma^{(*)}$ of the operator $\mathrm{Op}(\sigma)^{\ast}$ is asymptotically given by 
$$\sigma^{(*)}(x,\xi)\sim\sum_{j\geq 0}c_{m_1-j\eta_1,m_2-j\eta_2}(x,\xi),$$ 
where $c_{m_1-j\eta_1,m_2-j\eta_2}\in S^{m_1-j\eta_1,m_2-j\eta_2}_{\rho,\delta}(G\times\widehat{G})$ is equal to the sum of three terms 
$$d_{m_1-j\eta_1,m_2-j\eta_2}'(x,\xi)+d''_{m_1-j\eta_1,m_2-(j+1)\eta_2}(x,\xi)+d'''_{m_1-(j+1)\eta_1,m_2-j\eta_2}(x,\xi),$$ 
which are given by the following formulas:
\begin{align*}
d'_{m_1-j\eta_1,m_2-j\eta_2}(x,\xi)=
&\sum_{|\alpha_1|=|\alpha_2|=j}\frac{1}{\alpha_1!\alpha_2!}\Delta^{\alpha_1,\alpha_2}\partial^{\alpha_1,\alpha_2}\sigma(x,\xi)^{\ast},\\d''_{m_1-j\eta_1,m_2-(j+1)\eta_2}(x,\xi)=
&\sum_{|\alpha_1|=j}\frac{1}{\alpha_1!}\biggl(\Delta^{\alpha_1,0}\partial^{\alpha_1,0}\sigma^{(\ast_2)}(x,\xi)-\sum_{|\alpha_2|\leq |\alpha_1|}\frac{1}{\alpha_2!}\Delta^{\alpha_1,\alpha_2}
\partial^{\alpha_1,\alpha_2}\sigma(x,\xi)^{\ast}\biggr),\\ d'''_{m_1-(j+1)\eta_1,m_2-j\eta_2}(x,\xi)=
&\sum_{|\alpha_2|=j}\frac{1}{\alpha_2!}\biggl(\Delta^{0,\alpha_2}\partial^{0,\alpha_2}\sigma^{(\ast_1)}(x,\xi)-\sum_{|\alpha_1|\leq |\alpha_2|}\frac{1}{\alpha_1!}\Delta^{\alpha_1,\alpha_2}
\partial^{\alpha_1,\alpha_2}\sigma(x,\xi)^{\ast}\biggr).
\end{align*} 
In particular, for any given $N\in\mathbb{N}$, the reminder
$$r_N=\sigma^{(*)}-\sum_{j<N}c_{m_1-j\eta_1,m_2-j\eta_2}$$ belongs to the symbol class $S_{\rho,\delta}^{m_1-N\eta_1,m_2-N\eta_2}(G\times\widehat{G}).$
\end{theorem}
\begin{proof}
Since the kernel of $\sigma^{(*)}$ satisfies $\kappa_{\sigma^{(*)}}(x,y)=\overline{\kappa_{\sigma}}(xy^{-1},y^{-1})$, by taking the Fourier transform with respect to the second variable we have 
$$\sigma^{(*)}(x,\xi)=\int_G\overline{\kappa_{\sigma}}(xy^{-1},y^{-1})\xi_1^*(y_1)\otimes\xi_2^*(y_2)dy.$$
We shall expand $\overline{\kappa_{\sigma}}(xy^{-1},y^{-1})=\overline{\kappa_{\sigma}}(x_1y_1^{-1},x_2y_2^{-1},y^{-1})$ in the first variable at $y_1=e_1$ and then, in the second variable at $y_2=e_2$:
\begin{align*}
\overline{\kappa_{\sigma}}(xy^{-1},y^{-1})
&=\sum_{|\alpha_1|<N,|\alpha_2|<N}\frac{1}{\alpha_1!\alpha_2!}q^{\alpha_1,\alpha_2}(y)\partial^{\alpha_1,\alpha_2}\overline{\kappa_{\sigma}}(x,y^{-1})\\ 
&+\sum_{|\alpha_1|<N,|\alpha_2|=N}\frac{1}{\alpha_1!\alpha_2!}q^{\alpha_1,\alpha_2}(y)\overline{(\partial^{\alpha_1,0}\kappa_{\sigma})_{\alpha_2}}(x_1,x_2y_2^{-1},y^{-1})\\ 
&+\sum_{|\alpha_1|=N,|\alpha_2|<N}\frac{1}{\alpha_1!\alpha_2!}q^{\alpha_1,\alpha_2}(y)\overline{\partial^{0,\alpha_2}(\kappa_{\sigma})_{\alpha_1}}(x_1y_1^{-1},x_2,y^{-1})\\ 
&+\sum_{|\alpha_1|=N,|\alpha_2|=N}\frac{1}{\alpha_1!\alpha_2!}q^{\alpha_1,\alpha_2}(y)\overline{(\kappa_{\sigma})_{\alpha_1,\alpha_2}}(x_1y_1^{-1},x_2y_2^{-1},y^{-1})\\ 
&= I+II+III+IV.
\end{align*}
The first term is the kernel of the operator whose symbol is given by
$$\sum_{|\alpha_1|<N,|\alpha_2|<N}\frac{1}{\alpha_1!\alpha_2!}\Delta^{\alpha_1,\alpha_2}\partial^{\alpha_1,\alpha_2}\sigma(x,\xi)^*.$$ 
We may write $II$ as 
$$II=\sum_{|\alpha_1|<N}\frac{1}{\alpha_1!}\biggl(q^{\alpha_1,0}(y_1)\overline{\partial^{\alpha_1,0}\kappa_{\sigma}}(x_1,x_2y_2^{-1},y^{-1})-\sum_{|\alpha_2|<N}\frac{1}{\alpha_2!}q^{\alpha_1,\alpha_2}(y)
\overline{\partial^{\alpha_1,\alpha_2}\kappa_{\sigma}}(x_1,x_2,y^{-1})\biggr),$$ 
namely $II$ is the kernel of the pseudodifferential operator with symbol 
$$\sum_{|\alpha_1|<N}\frac{1}{\alpha_1!}\biggl(\Delta^{\alpha_1,0}\partial^{\alpha_1,0}\sigma^{(\ast_2)}(x,\xi)-\sum_{|\alpha_2|<N}\frac{1}{\alpha_2!}\Delta^{\alpha_1,\alpha_2}\partial^{\alpha_1,\alpha_2}
\sigma(x,\xi)^{\ast}\biggr).$$ 
By similar arguments, one shows that $III$ is the kernel of the operator with symbol
$$\sum_{|\alpha_2|<N}\frac{1}{\alpha_2!}\biggl(\Delta^{0,\alpha_2}\partial^{0,\alpha_2}\sigma^{(\ast_1)}(x,\xi)-\sum_{|\alpha_1|<N}\frac{1}{\alpha_1!}
\Delta^{\alpha_1,\alpha_2}\partial^{\alpha_1,\alpha_2}\sigma(x,\xi)^{\ast}\biggr).$$ 
By rearranging the terms, we find 
\begin{align*}
\sigma^{(*)}(x,\xi)
&=\sum_{|\alpha_1|=|\alpha_2|<N}\frac{1}{\alpha_1!\alpha_2!}\Delta^{\alpha_1,\alpha_2}\partial^{\alpha_1,\alpha_2}\sigma(x,\xi)^{\ast}\\ 
&+\sum_{|\alpha_1|<N}\frac{1}{\alpha_1!}\biggl(\Delta^{\alpha_1,0}\partial^{\alpha_1,0}\sigma^{(\ast_2)}(x,\xi)-\sum_{|\alpha_2|\leq|\alpha_1|}\frac{1}{\alpha_2!}\Delta^{\alpha_1,\alpha_2}
\partial^{\alpha_1,\alpha_2}\sigma(x,\xi)^{\ast}\biggr)\\
&+\sum_{|\alpha_2|<N}\frac{1}{\alpha_2!}\biggl(\Delta^{0,\alpha_2}\partial^{0,\alpha_2}\sigma^{(\ast_1)}(x,\xi)-\sum_{|\alpha_1|\leq|\alpha_2|}\frac{1}{\alpha_1!}\Delta^{\alpha_1,\alpha_2}
\partial^{\alpha_1,\alpha_2}\sigma(x,\xi)^{\ast}\biggr)\\
&+\sum_{|\alpha_!|=N,|\alpha_2|=N}\frac{1}{\alpha_1!\alpha_2!}\int_Gq^{\alpha_1,\alpha_2}(y)\overline{(\kappa_{\sigma})_{\alpha_1,\alpha_2}}(x_1y_1^{-1},x_2y_2^{-1},y^{-1})\xi^*_1(y_1)\otimes\xi^*_2(y_2)dy\\
&=\sum_{j<N}\biggl(d'_{m_1-j\eta_1,m_2-j\eta_2}+d''_{m_1-j\eta_1,m_2-(j+1)\eta_2}+d'''_{m_1-(j+1)\eta_1,m_2-j\eta_2}\biggr)+r_N.
\end{align*}
Now we show that the reminder $r_N$ belongs to the class $S^{m_1-N\eta_1,m_2-N\eta_2}_{\rho,\delta}(G\times\widehat{G})$, that is, 
$$\sup_{x\in G}\lVert\partial^{\gamma_1,\gamma_2}\Delta^{\beta_1,\beta_2}r_N(x,\xi)\rVert_{\mathcal{L}(\mathcal{H}_{\xi})}\lesssim\langle\xi_1\rangle^{m_1-N\eta_1-\rho_1|\beta_1|+\delta_1|\gamma_1|}
\langle\xi_2\rangle^{m_2-N\eta_2-\rho_2|\beta_2|+\delta_2|\gamma_2|},$$ 
for any $(\gamma_1,\gamma_2)\in\mathbb{N}_0^{n_1}\times\mathbb{N}_0^{n_2},(\beta_1,\beta_2)\in\mathbb{N}_0^{n_P}\times\mathbb{N}_0^{n_R}$, where $\eta_i=\rho_i-\delta_i$ for $i=1,2$. \\
Write $\xi^*(y)=\langle\xi_1\rangle^{-2s_1}\langle\xi_2\rangle^{-2s_2}(I_1+\mathcal L_{G_1})^{s_1}_{y_1}\otimes(I_2+\mathcal L_{G_2})^{s_2}_{y_2}\xi^*(y)$ for some 
$s_1,s_2\in\mathbb{N}_0$ to be chosen and integrate by parts, to obtain
\begin{align*}\partial^{\gamma_1,\gamma_2}\Delta^{\beta_1,\beta_2}r_N(x,\xi)&=\sum_{|\alpha_1|=N,|\alpha_2|=N}\frac{1}{\alpha_1!\alpha_2!}\times\\ &\hspace{1cm}\times\int_{G\times G}q^{\alpha_1+\beta_1,\alpha_2+\beta_2}(y^{-1})
\partial^{\gamma_1,\gamma_2}_x\overline{(\kappa_{\sigma})_{\alpha_1,\alpha_2}}(xy^{-1},y^{-1})\xi^*(y^{-1})dy\\ &=\langle\xi_1\rangle^{-2s_1}\langle\xi_2\rangle^{-2s_2}\sum_{|\alpha_1|=N,|\alpha_2|=N}\sum_{|\epsilon_1|=2s_1,|\epsilon_2|=2s_2}c_{\epsilon_1,\epsilon_2}\frac{1}{\alpha_1!\alpha_2!}\times \\ &\hspace{1cm}\times\int_{G}\widetilde{\partial}_{y^{-1}}^{\epsilon_1,\epsilon_2}
(q^{\alpha_1+\beta_1,\alpha_2+\beta_2}(y^{-1})\partial_{x}^{\gamma_1,\gamma_2}\overline{(\kappa_{\sigma})_{\alpha_1,\alpha_2}}(xy^{-1},y^{-1}))\xi^{\ast}(y^{-1})dy,\end{align*}
where, in the last equality, we applied the relation between left-invariant and right-invariant vector fields given by 
$\partial^{\alpha,\beta}(\phi(\cdot^{-1}))(x)=(-1)^{|\alpha|+|\beta|}(\widetilde{\partial}^{\alpha,\beta}\phi)(x^{-1})$.
Now we fix $N\in\mathbb{N}$ and show that there exists $\widetilde{N}\in\mathbb{N}$, $N<\widetilde{N}$, such that 
$$ \partial^{\gamma_1,\gamma_2}\Delta^{\beta_1,\beta_2}r_{\widetilde{N}}(x,\xi)\in S^{m_1-N\eta_1-\rho_1|\beta_1|+\delta_1|\gamma_1|,m_2-N\eta_2-\rho_2|\beta_2|+\delta_2|\gamma_2|}_{\rho,\delta}(G\times\widehat{G}),$$ 
where $\eta_i=\rho_i-\delta_i, i=1,2$. By the previous computations, we obtain \begin{align*}\lVert \partial^{\gamma_1,\gamma_2}\Delta^{\beta_1,\beta_2}r_{\widetilde{N}}(x,\xi)\rVert_{\mathcal{L}(\mathcal{H}_{\xi})}&\lesssim\sum_{|\alpha_1|=\widetilde{N},|\alpha_2|=\widetilde{N}}\sum_{|\epsilon_1|=2s_1,|\epsilon_2|=2s_2}\langle\xi_1\rangle^{-2s_1}\langle\xi_2\rangle^{-2s_2}\frac{1}{\alpha_1!\alpha_2!}\times \\ &\hspace{0.5cm}\times\int_G|\widetilde{\partial}_{y^{-1}}^{\epsilon_1,\epsilon_2}(q^{\alpha_1+\beta_1,\alpha_2+\beta_2}(y^{-1})
\partial^{\gamma_1,\gamma_2}_x\overline{(\kappa_{\sigma})_{\alpha_1,\alpha_2}}(xy^{-1},y^{-1}))|dy.\end{align*}
Suppose that the above integral converges. Then we require
\begin{equation}\label{4}
\begin{cases}
-2s_1\leq m_1-N\eta_1-\rho_1|\beta_1|+\delta_1|\gamma_1|\\ -2s_2\leq m_2-N\eta_2-\rho_2|\beta_2|+\delta_2|\gamma_2|,
\end{cases}
\end{equation} 
so that 
$$\langle\xi_1\rangle^{-2s_1}\langle\xi_2\rangle^{-2s_2}\leq \langle\xi_1\rangle^{m_1-N\eta_1-\rho_1|\beta_1|+\delta_1|\gamma_1|}\langle\xi_2\rangle^{m_2-N\eta_2-\rho_2|\beta_2|+\delta_2|\gamma_2|}.$$
Next, we prove that indeed the integral 
\begin{equation}\label{int5}
\int_{G}|\widetilde{\partial}^{\epsilon_1,\epsilon_2}_{y}(q^{\alpha_1+\beta_1,\alpha_2+\beta_2}(y)\partial^{\gamma_1,\gamma_2}_x\overline{(\kappa_{\sigma})_{\alpha_1,\alpha_2}}(xy,y))|dy
\end{equation} 
converges. 
 
According to Theorem \ref{3.8}, close to the set $S$ (where we are interested because elsewhere the kernel is smooth) we can estimate the function of $y$ inside the integral, up to a constant depending on $\lVert\sigma\rVert_{S^{m_1,m_2}_{\rho,\delta}}$, by 
$$|y_1|^{-2\frac{m_1+\delta_1|\gamma_1|+n_1+|\epsilon_1|-\rho_1\widetilde{N}-\rho_1|\beta_1|}{\rho_1}}|y_2|^{-2\frac{m_2+\delta_2|\gamma_2|+n_2+|\epsilon_2|-\rho_2\widetilde{N}-\rho_2|\beta_2|}{\rho_2}}.$$
In order to make the integral in (\ref{int5}) convergent, we require that 
\begin{equation}\label{6}
\begin{cases}
    2(m_1+\delta_1|\gamma_1|+n_1+|\epsilon_1|-\rho_1\widetilde{N}-\rho_1|\beta_1|)<\rho_1n_1,\\
    2(m_2+\delta_2|\gamma_2|+n_2+|\epsilon_2|-\rho_2\widetilde{N}-\rho_2|\beta_2|)<\rho_2n_2.
\end{cases}
\end{equation} Hence we have to choose $\widetilde N$ suitably. To do that, we first observe that we may find nonnegative integers $s_1$ and $s_2$ (depending on $N$, which is fixed) such that 
$$\begin{cases}N\eta_1-m_1+\rho_1|\beta_1|-\delta_1|\gamma_1|\leq 2s_1,\\
N\eta_2-m_2+\rho_2|\beta_2|-\delta_2|\gamma_2|\leq 2s_2,\end{cases}$$ 
so that the conditions in (\ref{4}) hold. Then, we may choose $\widetilde{N}\in\mathbb{N}$ (depending on $s_1,s_2$) large enough such that 
$$\begin{cases}2s_1+m_1+\delta_1|\gamma_1|+n_1(1-\frac{\rho_1}{2})-\rho_1|\beta_1|<\rho_1\widetilde{N}\\2s_2+m_2+\delta_2|\gamma_2|+n_2(1-\frac{\rho_2}{2})-\rho_2|\beta_2|<\rho_2\widetilde{N},\end{cases}$$ 
so that the conditions in (\ref{6}) are satisfied for any $|\epsilon_1|= 2s_1$ and $|\epsilon_2|= 2s_2$. This shows that $\partial^{\gamma_1,\gamma_2}\Delta^{\beta_1,\beta_2}r_{\widetilde{N}}$ 
belongs to the symbol class $S^{m_1-N\eta_1-\rho_1|\beta_1|+\delta_1|\gamma_1|,m_2-N\eta_2-\rho_2|\beta_2|+\delta_2|\gamma_2|}_{\rho,\delta}(G\times\widehat{G})$\\
Now, we must show that $\partial^{\gamma_1,\gamma_2}\Delta^{\beta_1,\beta_2}r_{N}\in S^{m_1-N\eta_1-\rho_1|\beta_1|+\delta_1|\gamma_1|,m_2-N\eta_2-\rho_2|\beta_2|+\delta_2|\gamma_2|}_{\rho,\delta}(G\times\widehat{G})$. 
Notice that 
\begin{align*}
    \partial&^{\gamma_1,\gamma_2}\Delta^{\beta_1,\beta_2}r_N
=\\ &=\partial^{\gamma_1,\gamma_2}\Delta^{\beta_1,\beta_2}\sigma^{(*)}-\sum_{j<N}\partial^{\gamma_1,\gamma_2}\Delta^{\beta_1,\beta_2}c_{m_1-j\eta_1,m_2-j\eta_2}\\ 
&=\partial^{\gamma_1,\gamma_2}\Delta^{\beta_1,\beta_2}\sigma^{(*)}-\sum_{j<\widetilde{N}}\partial^{\gamma_1,\gamma_2}\Delta^{\beta_1,\beta_2}c_{m_1-j\eta_1,m_2-j\eta_2}+\sum_{N\leq j<\widetilde{N}}\partial^{\gamma_1,\gamma_2}\Delta^{\beta_1,\beta_2}c_{m_1-j\eta_1,m_2-j\eta_2}\\ 
&=\partial^{\gamma_1,\gamma_2}\Delta^{\beta_1,\beta_2}r_{\widetilde{N}}+\sum_{N\leq j<\widetilde{N}}\partial^{\gamma_1,\gamma_2}\Delta^{\beta_1,\beta_2}c_{m_1-j\eta_1,m_2-j\eta_2}.
\end{align*}
From the proved inclusions
\begin{align*} \partial^{\gamma_1,\gamma_2}\Delta^{\beta_1,\beta_2}r_{\widetilde{N}}&\in S^{m_1-N\eta_1-\rho_1|\beta_1|+\delta_1|\gamma_1|,m_2-N\eta_2-\rho_2|\beta_2|+\delta_2|\gamma_2|}_{\rho,\delta}(G\times\widehat{G}),\\ \sum_{N\leq j< \widetilde{N}}\partial^{\gamma_1,\gamma_2}\Delta^{\beta_1,\beta_2}c_{m_1-j\eta_1,m_2-j\eta_2}&\in 
S^{m_1-N\eta_1-\rho_1|\beta_1|+\delta_1|\gamma_1|,m_2-N\eta_2-\rho_2|\beta_2|+\delta_2|\gamma_2|}_{\rho,\delta}(G\times\widehat{G}),\end{align*}
we obtain that $\partial^{\gamma_1,\gamma_2}\Delta^{\beta_1,\beta_2}r_N$ belongs to $S^{m_1-N\eta_1-\rho_1|\beta_1|+\delta_1|\gamma_1|,m_2-N\eta_2-\rho_2|\beta_2|+\delta_2|\gamma_2|}_{\rho,\delta}(G\times\widehat{G})$, for all $(\gamma_1,\gamma_2)\in\mathbb{N}^{n_1}_0\times\mathbb{N}_0^{n_2}$ 
and $(\beta_1,\beta_2)\in\mathbb{N}_0^{n_P}\times\mathbb{N}_0^{n_R}$. This proves that $r_N\in S^{m_1-N\eta_1,m_2-N\eta_2}_{\rho,\delta}(G\times\widehat{G})$ for all $N\in\mathbb{N}$ and concludes the proof.
\end{proof}

\begin{theorem}[Asymptotic expansion]
Let $\{\sigma_j\}_{j\in\mathbb{N}_0}$ be a sequence of symbols in the classes $S^{m_j',m_j''}_{\rho,\delta}(G\times\widehat{G})$, with $m_j',m_j''$ being sequences decreasing to $-\infty$. Then, there exists a symbol
$\sigma\in S^{m_0',m_0''}_{\rho,\delta}(G\times\widehat{G})$, unique modulo $S^{-\infty,-\infty}(G\times\widehat{G})$, such that 
$$\sigma-\sum_{j=0}^M\sigma_j\in S^{m'_{M+1},m''_{M+1}}(G\times\widehat{G}),\ \ \ \forall M\in\mathbb{N}.$$
\end{theorem}
\begin{proof}
Let $\psi\in C^{\infty}(\mathbb{R},[0,1])$ be such that $\psi|_{]-\infty,\frac{1}{2}[}=0$ and $\psi|_{]1,+\infty[}=1$. 
Using the Leibniz-like property and Proposition \ref{prop4.2}, for any given $(\widetilde{m_1},\widetilde{m_2})\in\mathbb{R}^2$,
$(\alpha,\beta)\in\mathbb{N}_0^{n_P}\times\mathbb{N}_{0}^{n_R},(\gamma_1,\gamma_2)\in\mathbb{N}_0^{n_1}\times\mathbb{N}_0^{n_2}$ and $t_1,t_2\in]0,1[$, we have 
\begin{align*}
&\lVert\Delta^{\alpha,\beta}\partial^{\gamma_1,\gamma_2}\sigma_j(x,\xi)\psi(t_1\lambda_{\xi_1})\psi(t_2\lambda_{\xi_2})\rVert_{\mathcal{L}(\mathcal{H}_{\mathcal{\xi}})}\\
&\lesssim\sum_{\begin{matrix}_{|\alpha|\leq|\alpha_1|+|\alpha_2|\leq 2|\alpha|}\\_{|\beta|\leq |\beta_1|+|\beta_2|\leq 2|\beta|}
\end{matrix}}\lVert\Delta^{\alpha_1,\beta_1}\partial^{\gamma_1,\gamma_2}\sigma_j(x,\xi)\rVert_{\mathcal{L}(\mathcal{H}_{\xi})}\lVert\Delta^{\alpha_2,\beta_2}\psi(t_1\lambda_{\xi_1})
\psi(t_2\lambda_{\xi_2})\rVert_{\mathcal{L}(\mathcal{H}_{\xi})}\\ 
&\lesssim\rVert\sigma_j\rVert_{S^{m_j',m_j''}_{\rho,\delta},(2|\alpha|,2|\beta|),(|\gamma_1|,|\gamma_2|)}
\sum_{\begin{matrix}_{|\alpha|\leq|\alpha_1|+|\alpha_2|\leq 2|\alpha|}\\_{|\beta|\leq |\beta_1|+|\beta_2|\leq 2|\beta|}
\end{matrix}}\langle\xi_1\rangle^{m_j'-\rho_1|\alpha_1|+\delta_1|\gamma_1|}\langle\xi_2\rangle^{m_j''-\rho_2|\beta_1|+\delta_2|\gamma_2|}\times \\ &
\hspace{10cm}\times t_1^{\frac{\widetilde{m_1}}{2}}\langle\xi_1\rangle^{\widetilde{m_1}-|\alpha_2|}t_2^{\frac{\widetilde{m_2}}{2}}\langle\xi_2\rangle^{\widetilde{m_2}-|\beta_2|}.
\end{align*}
We then choose $\widetilde{m_1}=m_0'-m_j'$ and $\widetilde{m_2}=m_0''-m_j''$, so that 
\begin{align*}
\lVert\Delta^{\alpha,\beta}&\partial^{\gamma_1,\gamma_2}\sigma_j(x,\xi)\psi(t_1\lambda_{\xi_1})\psi(t_2\lambda_{\xi_2})\rVert_{\mathcal{L}(\mathcal{H}_{\xi})}\\
& \lesssim\lVert\sigma_j\rVert_{S^{m_j',m_j''}_{\rho,\delta},(2|\alpha|,2|\beta|),(|\gamma_1|,|\gamma_2|)}t_1^{\frac{m_0'-m_j'}{2}}t_2^{\frac{m_0''-m_j''}{2}}\langle\xi_1
\rangle^{m_0'-\rho_1|\alpha|+\delta_1|\gamma_1|}\langle\xi_2\rangle^{m_0''-\rho_2|\beta|+\delta_2|\gamma_2|},
\end{align*}
which gives, for any given $a=(a_1,a_2), b=(b_1,b_2)\in \mathbb{N}_0^2$ and $t_1,t_2\in ]0,1[$ 
$$\lVert\sigma_j(x,\xi)\psi(t_1\lambda_{\xi_1})\psi(t_2\lambda_{\xi_2})\rVert_{S^{m_0',m_0''}_{\rho,\delta},a,b}\leq C_{a,b,m_j',m_j'',\sigma_j}t_1^{\frac{m_0'-m_j'}{2}}t_2^{\frac{m_0''-m_j''}{2}}.$$
We now choose a decreasing sequence $t_j$, such that for all $j\in\mathbb{N}_0$, we have 
$$t_j\in ]0,2^{-j}[ \text{ and } C_{(j,j),(j,j),m_j',m_j'',\sigma_j}t_1^{\frac{m_0'-m_j'}{2}}t_2^{\frac{m_0''-m_j''}{2}}\leq 2^{-j},$$
and define $\widetilde{\sigma_j}(x,\xi):=\sigma_j(x,\xi)\psi(t_j\lambda_{\xi_1})\psi(t_j\lambda_{\xi_2})$.
For any $l\in\mathbb{N}_0$, the sum 
$$\sum_{j=0}^{+\infty}\lVert\widetilde{\sigma_j}(x,\xi)\rVert_{S^{m_0',m_0''}_{\rho,\delta},(l,l),(l,l)}\leq\sum_{j=0}^l\lVert\widetilde{\sigma_j}(x,\xi)
\rVert_{S_{(\rho,\delta)}^{m_0',m_0''},(l,l),(l,l)}+\sum_{j=l+1}^{+\infty}2^{-j}<+\infty.$$
As $S^{m_0',m_0'}_{\rho,\delta}(G\times\widehat{G})$ is a Fréchet space, we find that $\sigma:=\sum_{j=0}^{+\infty}\widetilde{\sigma_j}$ belongs to
$S^{m_0',m_0''}_{\rho,\delta}(G\times\widehat{G})$. Starting the summation at $j=M+1$, the same proof gives $\sum_{j=M+1}^{+\infty}\widetilde{\sigma_j}\in S^{m_{M+1}',m_{M+1}''}_{\rho,\delta}(G\times\widehat{G})$. 
Hence, we have that
$$\sigma-\sum_{j=0}^{M}\sigma_j=\sum_{j=0}^M(\psi(t_j\lambda_{\xi_1})\psi(t_j\lambda_{\xi_2})-1)\sigma_j+\sum_{j=M+1}^{+\infty}\widetilde{\sigma_j}$$ 
belongs to $S^{m_{M+1}',m_{M+1}''}_{\rho,\delta}(G\times\widehat{G})$ for every $M\in\mathbb{N}$, since $\psi(t_j\lambda_{\xi_1})\psi(t_j\lambda_{\xi_2})-1$ is smoothing. 
To show that $\sigma$ is unique up to smoothing operators, we take another symbol $\tau$ with the same asymptotic expansion as $\sigma$ and observe that, for any given $M\in\mathbb{N}$,
$$\sigma-\tau=\biggl(\sigma-\sum_{j=0}^M\sigma_j\biggr)-\biggl(\tau-\sum_{j=0}^M\sigma_j\biggr)\in S^{m'_{M+1},m''_{M+1}}_{\rho,\delta}(G\times\widehat{G}),$$ 
which shows that $\sigma=\tau$ modulo $S^{-\infty,-\infty}(G\times\widehat{G})$ and proves the result.
\end{proof}

\section{Bi(hypo)ellipticity}\label{sec6}
In this section, we give a description of bielliptic and bihypoelliptic symbols on $G=G_1\times G_2$. We start with a proposition, which gives a necessary condition for the inverse symbol to belong to a certain bisingular class. At first, we briefly set some notations that we will use through this section.

\begin{definition}
Let $G=G_1\times G_2$ be compact.
Given a symbol $a\in S^{m_1,m_2}_{\rho,\delta}(G\times\widehat{G})$, we say that it is \text{\rm invertible except for finitely many representations $\xi\in\widehat{G}$ uniformly in $x\in G$} if the set
$$\{\xi\in\widehat{G};\,\,a(x,\xi)\,\,\text{\rm is not invertible}\}$$
has a finite cardinality uniformly with respect to $x\in G$.

We say that the operator $a(x_1,x_2,D_1,\xi_2)$ is \text{\rm exactly invertible except for finitely many} $\xi_2\in\widehat{G}_2$ \text{\rm uniformly in $x_2\in G_2$} if the set
$$\{\xi_2\in\widehat{G}_2;\,\,a(x_1,x_2,D_1,\xi_2)\,\,\text{\rm is non invertible as an operator in $L^{m_1}_{\rho_1,\delta_1}(G_1)$}\}$$
has a finite cardinality uniformly with respect to $x_2\in G_2$. Analogously for $x_1\in G_1$.
\end{definition}
Let $a\in S^{m_1,m_2}_{\rho,\delta}(G\times\widehat{G})$ be invertible except for finitely many representations $\xi\in\widehat G$ uniformly in $x\in G$. There exist finite subsets $\widehat U_1\subset \widehat G_1 $ and $\widehat U_2\subset\widehat G_2$ such that the product $\widehat U_1\times\widehat U_2$ contains the trivial representation $1_G=1_{G_1}\otimes 1_{G_2}\in\widehat U_1\times\widehat U_2$ and every irreducible representation $\xi=\xi_1\otimes\xi_2$ where the symbol $a(x_1,x_2,\xi_1,\xi_2)$ is noninvertible (for some $(x_1,x_2)\in G_1\times G_2$). Thus, $(\widehat U_1\times\widehat G_2)\cup(\widehat G_1\times\widehat U_2)\subset \widehat G_1\times\widehat G_2$ is a set containing all the (finite) points of noninvertibility of $a$ and the subset $\widehat S$, where $$\widehat S:=(\{1_{G_1}\}\times \widehat G_2)\cup(\widehat G_1\times\{1_{G_2}\})\subset \widehat G_1\times\widehat G_2.$$ For $j=1,2$ we denote by $\chi_j=\chi_j(\xi_j)\in S^{0}_{\rho_j,\delta_j}(G_j\times\widehat G_j)$ an excision function such that $\chi_j|_{\widehat U_j}=0$ and $\chi_j=1$ outside a larger subset of $\widehat G_j$ containing $\widehat U_j$, which will be denoted by $\widehat V_j$. Thus, we have $$\chi:=\chi(\xi_1,\xi_2)=\chi_1(\xi_1)\chi_2(\xi_2)\in S^{0,0}_{\rho,\delta}(G\times\widehat G)$$ is an excision function which is equal to zero on $(\widehat U_1\times\widehat G_2)\cup(\widehat G_1\times\widehat U_2)$ and equal to one outside $(\widehat V_1\times\widehat G_2)\cup(\widehat G_1\times\widehat V_2)$. Also, notice that $1-(\chi_1+\chi_2-\chi)\in S^{-\infty,-\infty}$.

\begin{proposition}\label{a-1}Let $a\in S^{m_1,m_2}_{\rho,\delta}(G\times\widehat{G}), 1\geq \rho_i>\delta_i\geq 0, i=1,2$ and let $(m'_1,m'_2)\in\mathbb R^2$ be such that $m_i'\leq m_i$ for $i=1,2$. Assume that the symbol $a(x_1,x_2,\xi_1,\xi_2)$ is invertible except for finitely many representations $\xi\in\widehat G$ uniformly in $x\in G$ and that it satisfies 
$$\lVert a(x_1,x_2,\xi_1,\xi_2)^{-1}\rVert_{\mathcal L(\mathcal H_{\xi})}\lesssim \langle\xi_1\rangle^{-m_1'}\langle\xi_2\rangle^{-m_2'}.$$ 
In addition, if $m_i'\neq m_i$ for some $i=1,2$, assume that (in the points of invertibility) the following inequality holds 
\begin{equation}\label{horm}\begin{split}\lVert a(x_1,x_2,\xi_1,\xi_2)^{-1}\Delta^{\alpha_1,\alpha_2}\partial^{\beta_1,\beta_2}a(x_1,x_2,\xi_1,\xi_2)\rVert_{\mathcal L(\mathcal H_{\xi})}\hspace{4cm}\\
\hspace{4cm}\lesssim_{\alpha_1,\alpha_2,\beta_1,\beta_2}\langle\xi_1\rangle^{\delta_1|\beta_1|-\rho_1|\alpha_1|}\langle\xi_2\rangle^{\delta_2|\beta_2|-\rho_2|\alpha_2|},\end{split}
\end{equation} 
for all $(\alpha_1,\alpha_2)\in\mathbb N_0^{n_P}\times\mathbb N_0^{n_R}$ and $(\beta_1,\beta_2)\in\mathbb N_0^{n_1}\times\mathbb N_0^{n_2}$. Then the symbol $\chi a^{-1}$ (with the notation introduced above) belongs to $S^{-m_1',-m_2'}_{\rho,\delta}(G\times\widehat G)$.
\end{proposition}
\begin{proof}
Let us first denote $$b(x_1,x_2,\xi_1,\xi_2):=\begin{cases}
    a(x_1,x_2,\xi_1,\xi_2)^{-1} \ \ \ \text{ when }\ \ \ (\xi_1,\xi_2)\notin\widehat U_1\times\widehat U_2\\
0\ \ \ \ \ \ \ \ \ \ \ \ \ \ \ \ \ \ \ \ \ \ \ \ \text{ when } \ \ \ (\xi_1,\xi_2)\in\widehat U_1\times\widehat U_2\end{cases}$$ and estimate $\partial^{\beta_1,\beta_2}b$ by induction. Assume that we have proved 
$$\lVert\partial^{\beta_1,\beta_2}b(x_1,x_2,\xi_1,\xi_2)\rVert_{\mathcal L(\mathcal H_{\xi})}\leq C_{\beta_1,\beta_2}\langle\xi_1\rangle^{-m_1'+\delta_1|\beta_1|}\langle\xi_2\rangle^{-m_2'+\delta_2|\beta_2|},$$ 
whenever $|\beta_1|,|\beta_2|\leq k$. Since $\partial^{\beta_1+e_{j_1},\beta_2}(a(x_1,x_2,\xi_1,\xi_2)b(x_1,x_2,\xi_1,\xi_2))=0$, by the usual Leibniz formula we get 
$$a\partial^{\beta_1+e_{j_1},\beta_2}b=-\sum_{\begin{matrix}_{
    \gamma_1+\epsilon_1=\beta_1+e_{j_1},|\epsilon_1|\leq|\beta_1|}\\_{\gamma_2+\epsilon_2=\beta_2}
\end{matrix}}C_{\gamma_1,\gamma_2,\epsilon_1,\epsilon_2}(\partial^{\gamma_1,\gamma_2}a)(\partial^{\epsilon_1,\epsilon_2}b),$$ where $|e_{j_1}|=1$. 
Using (\ref{horm}) and the inductive hypothesis, we obtain 
\begin{align*}
&\lVert \partial^{\beta_1+e_{j_1},\beta_2}b(x_1,x_2,\xi_1,\xi_2)\rVert_{\mathcal L(\mathcal H_{\xi})}
\lesssim \sum_{\begin{matrix}_{
    \gamma_1+\epsilon_1=\beta_1+e_{j_1},|\epsilon_1|\leq k}\\_{\gamma_2+\epsilon_2=\beta_2}
\end{matrix}}\lVert a^{-1}\partial^{\gamma_1,\gamma_2}a\rVert_{\mathcal L(\mathcal H_{\xi})}\lVert\partial^{\epsilon_1,\epsilon_2}b\rVert_{\mathcal L(\mathcal H_{\xi})} \\  & \lesssim\sum_{\gamma_1,\gamma_2,\epsilon_1,\epsilon_2}\langle\xi_1\rangle^{\delta_1|\gamma_1|}\langle\xi_2\rangle^{\delta_2|\gamma_2|}\langle\xi_1\rangle^{-m_1'+\delta_1|\epsilon_1|}\langle\xi_2\rangle^{-m_2'+\delta_2|\epsilon_2|}\lesssim\langle\xi_1\rangle^{-m_1'+\delta_1(k+1)}\langle\xi_2\rangle^{-m_2'+\delta_2k}.
\end{align*}
and, analogously, for $|e_{j_2}|=1$, we have $$\lVert\partial^{\beta_1,\beta_2+e_{j_2}}b(x_1,x_2,\xi_1,\xi_2)\rVert_{\mathcal L(\mathcal H_{\xi})}\lesssim\langle\xi_1\rangle^{-m_1'+\delta_1k}\langle\xi_2\rangle^{-m_2'+\delta_2(k+1)}.$$ Now, we want to estimate $\partial^{\beta_1+e_{j_1},\beta_2+e_{j_2}}b$. Since $\partial^{\beta_1+e_{j_1},\beta_2+e_{j_2}}(ab)=0$, by the usual Leibniz rule we get $$a\partial^{\beta_1+e_{j_1},\beta_2+e_{j_2}}b=-\sum_{\begin{matrix}_{
    \gamma_1+\epsilon_1=\beta_1+e_{j_1},|\epsilon_1|\leq k}\\_{\gamma_2+\epsilon_2=\beta_2+e_{j_2},|\epsilon_2|\leq k}
\end{matrix}}C_{\gamma_1,\gamma_2,\epsilon_1,\epsilon_2}(\partial^{\gamma_1,\gamma_2}a)(\partial^{\epsilon_1,\epsilon_2}b)+$$ 
$$-\sum_{\gamma_2+\epsilon_2=\beta_2+e_{j_2},|\epsilon_2|\leq k}C_{\gamma_2,\epsilon_1,\epsilon_2}(\partial^{0,\gamma_2}a)(\partial^{\beta_1+e_{j_1},\epsilon_2}b)-\sum_{\gamma_1+\epsilon_1=\beta_1+e_{j_1},|\epsilon_1|\leq k}C_{\gamma_1,\epsilon_1,\epsilon_2}(\partial^{\gamma_1,0}a)(\partial^{\epsilon_1,\beta_2+e_{j_2}}b),$$ so that we obtain the estimate 
$$\lVert \partial^{\beta_1+e_{j_1},\beta_2+e_{j_2}}b(x_1,x_2,\xi_1,\xi_2)\rVert_{\mathcal L(\mathcal H_{\xi})} \lesssim \langle\xi_1\rangle^{-m_1'+\delta_1(k+1)}\langle\xi_2\rangle^{-m_2'+\delta_2(k+1)},$$ which is the desired inequality $$\lVert\partial^{\widetilde\beta_1,\widetilde\beta_2}b(x_1,x_2,\xi_1,\xi_2)\rVert_{\mathcal L(\mathcal H_{\xi})}\lesssim\langle\xi_1\rangle^{-m_1'+\delta_1|\widetilde\beta_1|}\langle\xi_2\rangle^{-m_2'+\delta_2|\widetilde\beta_2|},$$ for any $\widetilde\beta_1,\widetilde\beta_2$, where $|\widetilde\beta_1|=|\widetilde\beta_2|\leq k+1$.

Let us now estimate $\Delta^{\alpha_1,\alpha_2}b$ when $|\alpha_1|=k_1,|\alpha_2|=k_2$ are arbitrary but fixed. Our purpose is to use an inductive argument on $k_1+k_2$. Since outside $\widehat U_1\times\widehat U_2$ we have the identity $a(x_1,x_2,\xi_1,\xi_2)b(x_1,x_2,\xi_1,\xi_2)=I$, then $\Delta^{\alpha_1,\alpha_2}(ab)=0$ and, by using the finite Leibniz formula in Proposition \ref{Leibniz}, we get
$$\sum_{\begin{matrix}_{|\beta_1|,|\gamma_1|\leq k_1\leq|\beta_1|+|\gamma_1|}\\ _{ |\beta_2|,|\gamma_2|\leq k_2\leq |\beta_2|+|\gamma_2|}\end{matrix}}c^{\alpha_1}_{\beta_1,\gamma_1}c^{\alpha_2}_{\beta_2,\gamma_2}\Delta^{\beta_1,\beta_2}a\Delta^{\gamma_1,\gamma_2}b=0,$$
which, more precisely, can be written as $$\Delta^{\alpha_1,\alpha_2}b=-\sum_{\begin{matrix}_{|\gamma_1|=k_1,|\gamma_2|=k_2}\\ _{|\beta_1|+|\beta_2|>0}\end{matrix}}c^{\alpha}_{\beta,\gamma}\ a^{-1}\Delta^{\beta_1,\beta_2}a\Delta^{\gamma_1,\gamma_2}b-\sum_{\begin{matrix}_{|\gamma_1|+|\gamma_2|<k_1+k_2}\\ _{|\beta_1|,|\beta_2|\geq0}\end{matrix}}c^{\alpha}_{\beta,\gamma}\  a^{-1}\Delta^{\beta_1,\beta_2}a\Delta^{\gamma_1,\gamma_2}b.$$
Using the inductive hypothesis and applying (\ref{horm}), we estimate each term in the latter summation by $$\lVert a^{-1}\Delta^{\beta_1,\beta_2}a\Delta^{\gamma_1,\gamma_2}b\rVert\lesssim\langle\xi_1\rangle^{-m_1'-\rho_1(|\beta_1|+|\gamma_1|)}\langle\xi_2\rangle^{-m_2'-\rho_2(|\beta_2|+|\gamma_2|)}\lesssim\langle\xi_1\rangle^{-m_1'-\rho_1|\alpha_1|}\langle\xi_2\rangle^{-m_2'-\rho_2|\alpha_2|},$$ 
thus obtaining
$$\lVert\Delta^{\alpha_1,\alpha_2}b\rVert\lesssim\sum_{|\beta_1|+|\beta_2|>0}\langle\xi_1\rangle^{-\rho_1|\beta_1|}\langle\xi_2\rangle^{-\rho_2|\beta_2|}\sup_{|\gamma_1|=k_1,|\gamma_2|=k_2}\lVert\Delta^{\gamma_1,\gamma_2}b\rVert+ \langle\xi_1\rangle^{-m_1'-\rho_1|\alpha_1|}\langle\xi_2\rangle^{-m_2'-\rho_2|\alpha_2|}.$$ Now, we should take the supremum in $\alpha=(\alpha_1,\alpha_2)$ with fixed height $(k_1,k_2)$ to get \begin{align*}&\sup_{|\alpha_1|=k_1,|\alpha_2|=k_2}\lVert\Delta^{\alpha_1,\alpha_2}b\rVert_{\mathcal L(\mathcal H_{\xi})}\\ &\lesssim\sum_{|\beta_1|+|\beta_2|>0}\langle\xi_1\rangle^{-\rho_1|\beta_1|}\langle\xi_2\rangle^{-\rho_2|\beta_2|}\sup_{|\gamma_1|=k_1,|\gamma_2|=k_2}\lVert\Delta^{\gamma_1,\gamma_2}b\rVert_{\mathcal L(\mathcal H_{\xi})}+\langle\xi_1\rangle^{-m_1'-\rho_1k_1}\langle\xi_2\rangle^{-m_2'-\rho_2k_2}.\end{align*} 
When $\langle\xi_1\rangle^{-\rho_1}+\langle\xi_2\rangle^{-\rho_2}$ is small enough, we obtain the desired estimate  $$\sup_{|\alpha_1|=k_1,|\alpha_2|=k_2}\lVert\Delta^{\alpha_1,\alpha_2}b\rVert_{\mathcal L(\mathcal H_{\xi})}\lesssim\langle\xi_1\rangle^{-m_1'-\rho_1k_1}\langle\xi_2\rangle^{-m_2'-\rho_2k_2}.$$ Hence, we get the result when $\langle\xi_1\rangle^{-\rho_1}+\langle\xi_2\rangle^{-\rho_2}<c$ for some positive constant $c>0$. We multiply $a^{-1}$ by an excision function $\chi=\chi(\xi_1,\xi_2)\in S^{0,0}_{\rho,\delta}(G\times\widehat G)$, which is equal to zero on $(\widehat U_1\times\widehat G_2)\cup(\widehat G_1\times\widehat U_2)$ and equal to $1$ outside  $(\widehat V_1\times\widehat G_2)\cup(\widehat G_1\times\widehat V_2)$ (see the beginning of the section). By applying the finite Leibniz formula we obtain (notice that $\chi a^{-1}=\chi b$)  $$\lVert\partial^{\beta_1,\beta_2}\Delta^{\alpha_1,\alpha_2}(\chi a^{-1})\rVert_{\mathcal L(\mathcal H_{\xi})}\lesssim\sum_{\begin{matrix}_{|\epsilon_1|,|\gamma_1|\leq|\alpha_1|\leq|\epsilon_1|+|\gamma_1|}\\_{|\epsilon_2|,|\gamma_2|\leq|\alpha_2|\leq|\epsilon_2|+|\gamma_2|}\end{matrix}}\lVert\Delta^{\epsilon_1,\epsilon_2}\chi\rVert_{\mathcal L(\mathcal H_{\xi})}\lVert\partial^{\beta_1,\beta_2}\Delta^{\gamma_1,\gamma_2}b\rVert_{\mathcal L(\mathcal H_{\xi})}$$ 

$$\lesssim \langle\xi_1\rangle^{-m_1'-\rho_1|\alpha_1|+\delta_1|\beta_1|}\langle\xi_2\rangle^{-m_2'-\rho_2|\alpha_2|+\delta_2|\beta_2|}.$$ This means that the symbol $\chi(\xi_1,\xi_2) a^{-1}(x_1,x_2,\xi_1,\xi_2)$ belongs to the class $S^{-m_1',-m_2'}_{\rho,\delta}(G\times\widehat G)$ and the proof of Proposition \ref{a-1} is complete.
\end{proof}

We are now ready to define bielliptic and bihypoelliptic operators and to show the existence of parametrices in such cases.

\begin{definition}\label{defbiell} Let $a\in S^{m_1,m_2}_{\rho,\delta}(G\times\widehat{G}), 1\geq \rho_i>\delta_i\geq 0, i=1,2$ and $A=\mathrm{Op}(a)\in L^{m_1,m_2}_{\rho,\delta}(G)$. We say that A is bielliptic if 
\begin{enumerate} 
\item  the symbol $a(x_1,x_2,\xi_1,\xi_2)$ is invertible for all $(x,\xi)\in G\times\widehat{G}$, and satisfies  $$\lVert a(x_1,x_2,\xi_1,\xi_2)^{-1}\rVert_{\mathcal{L}(\mathcal{H}_{\xi})}\lesssim\langle\xi_1\rangle^{-m_1}\langle\xi_2\rangle^{-m_2};$$  
\item the operator $a(x_1,x_2,D_1,\xi_2)$ is exactly invertible for all $(x_2,\xi_2)\in G_2\times\widehat{G}_2$ as an operator in $L^{m_1}_{\rho_1,\delta_1}(G_1)$ with inverse in $L^{-m_1}_{\rho_1,\delta_1} (G_1)$, in particular $$(a\circ_1 a^{-1})(x_1,x_2,D_1,\xi_2)=I_{\mathcal D'(G_1)},$$
\item the operator $a(x_1,x_2,\xi_1,D_2)$ is exactly invertible for all $(x_1,\xi_1)\in G_1\times\widehat{G}_1$ as an operator in $L^{m_2}_{\rho_2,\delta_2}(G_2)$, with inverse in $L^{-m_2}_{\rho_2,\delta_2}(G_2)$, in particular $$(a\circ_2 a^{-1})(x_1,x_2,\xi_1,D_2)=I_{\mathcal D'(G_2)}.$$
\end{enumerate} \end{definition}

\begin{theorem}\label{TheoParBiell} Let $A\in L^{m_1,m_2}_{\rho,\delta}(G),1\geq\rho_i>\delta_i\geq 0,i=1,2$, be bielliptic. Then, there exists $B\in L^{-m_1,-m_2}_{\rho,\delta}(G)$ such that 
$$AB=I+K,\ \ \ BA=I+H,$$ 
where $I=I_{\mathcal D'(G)}$ is the identity map on $G$ and $K,H\in L^{-\infty,-\infty}(G)$ are smoothing bisingular operators.
\end{theorem}
The proof of Theorem \ref{TheoParBiell} is omitted, since the theorem is a particular case of Theorem \ref{parametrixbihypo} below.

\begin{definition}\label{bihypoell} Let $a\in S^{m_1,m_2}_{\rho,\delta}(G\times\widehat{G}), 1\geq \rho_i>\delta_i\geq 0, i=1,2$, and $A=\mathrm{Op}(a)\in L^{m_1,m_2}_{\rho,\delta}(G)$. We say that A is bihypoelliptic of order $(m'_1,m'_2)\in\mathbb{R}^2$, with $m'_j\leq m_j$, $j=1,2,$, if the following conditions on its symbol are verified:
\begin{enumerate} 
\item except for a finite number of representations $\xi=\xi_1\otimes\xi_2\in\widehat G$, uniformly in $(x_1,x_2)\in G_1\times G_2$, the symbol $a(x_1,x_2,\xi_1,\xi_2)$ is invertible and the inverse $a(x_1,x_2,\xi_1,\xi_2)^{-1}$ satisfies $$\lVert a(x_1,x_2,\xi_1,\xi_2)^{-1}\rVert_{\mathcal{L}(\mathcal{H}_{\xi})}\lesssim\langle\xi_1\rangle^{-m'_1}\langle\xi_2\rangle^{-m'_2};$$ 
\item except for finitely many representations $\xi_2\in\widehat G_2$, uniformly in $x_2\in G_2$, the operator $a(x_1,x_2,D_1,\xi_2)$ is exactly invertible as an operator in 
$L^{m_1}_{\rho_1,\delta_1}(G_1)$ with inverse in $L^{-m'_1}_{\rho_1,\delta_1} (G_1)$ and, in particular
$$(a\circ_1a^{-1}\chi_1)(x_1,x_2,D_1,\xi_2)=I_{\mathcal D'(G_1)}-E_1(x_1,x_2,D_1,\xi_2),$$ 
where $E_1\in S^{-\infty,-\infty}$ has a finite support in $(\xi_1,\xi_2)$, uniformly in $(x_1,x_2)$; 
\item except for finitely many representations $\xi_1\in\widehat G_1$, uniformly in $x_1\in G_1$, the operator $a(x_1,x_2,\xi_1,D_2)$ is exactly invertible as an operator in  $L^{m_2}_{\rho_2,\delta_2}(G_2)$ with inverse in $L^{-m'_2}_{\rho_2,\delta_2}(G_2)$ and, in particular $$(a\circ_2a^{-1}\chi_2)(x_1,x_2,\xi_1,D_2)=I_{\mathcal D'(G_2)}-E_2(x_1,x_2,\xi_1,D_2),$$ where $E_2\in S^{-\infty,-\infty}$ has a finite support in $(\xi_1,\xi_2)$, uniformly in $(x_1,x_2)$;
\item for all $\xi=\xi_1\otimes\xi_2\in \widehat G$ where the symbol $a(x_1,x_2,\xi_1,\xi_2)$ is invertible, and for all multi-indices $(\alpha_1,\alpha_2)\in\mathbb N_0^{n_P}\times\mathbb N_0^{n_R},(\beta_1,\beta_2)\in\mathbb N_0^{n_1}\times\mathbb N_0^{n_2}$, the following inequality holds $$\lVert a(x_1,x_2,\xi_1,\xi_2)^{-1}\Delta^{\alpha_1,\alpha_2}\partial^{\beta_1,\beta_2}a(x_1,x_2,\xi_1,\xi_2)\rVert_{\mathcal{L}(\mathcal H_{\xi})} \lesssim \langle\xi_1\rangle^{\delta_1|\beta_1|-\rho_1|\alpha_1|}\langle\xi_2\rangle^{\delta_2|\beta_2|-\rho_2|\alpha_2|}.$$  \end{enumerate}  \end{definition} 

\begin{remark}
    We wish to emphasize that, even though the class of bisingular operators is much wider than that of tensor products of operators, it is specifically targeted to study the latter cases. We will see this more in detail through the next Example \ref{exBiell-Bihyp} of bihypoelliptic bisingular operators of different types, that is, of the form $A=A_1+A_2$ and $A=A_1\otimes A_2$. We will also see that in the two different situations, the noninvertibility over a finite set of representations is more often possible in the first case and hardly in the second case.
\end{remark}

\begin{example}\label{exBiell-Bihyp}
{\rm 
    There exist bihypoellptic operators of the form $A=A_1+A_2\in L^{m_1,m_2}(G_1\times G_2),$ where  $A_1\in L^{m_1,0}, A_2\in L^{0,m_2}$, whose symbol is non-invertible over a finite set. For instance, consider the Laplace operator $L=\partial_x^2+\partial_y^2$ on $\mathbb T^1_x\times\mathbb T^1_y$ and the heat operator $H=\partial_t-\partial_x^2$ on $\mathbb S^1_t\times\mathbb T^1_x$. In both of these examples, the global symbols $$\sigma_{L}(k,h)=\sigma_{L_1}(k)+\sigma_{L_2}(h)=k^2+h^2 \ \ \ \text{ and }\ \ \ \sigma_H(\tau,k)=\sigma_{H_1}(\tau)+\sigma_{H_2}(\kappa)=i\tau+\kappa^2$$ vanish over a finite set (actually at one point in both cases) and out of that set satisfy bihypoelliptic estimates of order $(-1,-1)$ and $(-\frac{1}{2},-1)$, respectively. Note also that $L$ is elliptic as a standard pseudodifferential operator, while is only bihypoelliptic as a bisingular operator. This puts into light the fact that operators of the form $A_1+A_2$ are more suitable to be studied via the standard pseudodiffrential calculus than via the bisingular one.\ 
    
    On the other hand, bihypoelliptic operators of product type $A_1\otimes A_2\in L^{m_1,m_2}_{\rho,\delta}(G_1\times G_2)$, where $A_1\in L^{m_1}_{\rho_1,\delta_1}(G_1), A_2\in L^{m_2}_{\rho_2,\delta_2}(G_2)$, must have a global symbol which is invertible everywhere. Let us call $\widehat H_1\subset \widehat G_1$ and $\widehat H_2\subset \widehat G_2$ the (possibly empty) sets of noninvertibility of $\sigma_{A_1}$ and $\sigma_{A_2}$, respectively.  In this case, the set of noninvertibility of the global symbol $$\sigma_{A_1\otimes A_2}=\sigma_{A_1}\otimes \sigma_{A_2}$$ contains $(\widehat G_1\times \widehat H_2)\cup (\widehat H_1\times \widehat G_2)$, which is finite if and only if it is empty. The condition on the finiteness of the set where the global symbol is non invertible is necessary to have the existence of a parametrix, so for tensors $A_1\otimes A_2$ the invertibility is needed. 
    
    Let us finally comment that in the standard pseudodiffrential calculus on compact Lie groups, $A_1\otimes A_2$ would belong to $S^{m_1+m_2}(G_1\times G_2)$. Here the singularity and regularity properties should be studied in the whole $G_1\times G_2$, causing a loss of information due to the structure of the operator. The use of bisingular classes for tensor products of operators on single groups, instead, allows to capture the behavior on each group separately, and get a refined characterization of singularity and regularity.}
\end{example}

The meaning of Example \ref{exBiell-Bihyp}  is that, depending on the form of the operator, one has an optimal pseudodiffrential setting to work with, that is, the classical or the bisingular one.

\vspace{.3cm}

We next show that bihypoelliptic operators possess a parametrix.

\begin{theorem}\label{parametrixbihypo}
Let $A\in L^{m_1,m_2}_{\rho,\delta}(G),1\geq\rho_i>\delta_i\geq 0,i=1,2$, be bihypoelliptic of order $(m_1',m_2')$, with $m_i'\leq m_i,$ for $i=1,2$. 
Then, there exists $B\in L^{-m'_1,-m'_2}_{\rho,\delta}(G)$ such that 
$$AB=I_{\mathcal D'(G)}+K,\ \ \ BA=I_{\mathcal D'(G)}+H,$$ 
where $K,H\in L^{-\infty,-\infty}(G)$ are smoothing bisingular operators. Consequently, we have 
$$\mathrm{singsupp}(Au)=\mathrm{singsupp}(u),\ \ \ \text{ for all } u\in\mathcal D'(G).$$
\end{theorem}
\begin{proof}
Observe that by (4) in Definition \ref{bihypoell} and Proposition \ref{a-1}, we have that $a^{-1}\chi\in S_{\rho,\delta}^{-m'_1,-m'_2}(G\times\widehat G)$. By our hypotheses (2) and (3) in Definition \ref{bihypoell}, 
we also have $$a\circ_1a^{-1}\chi_1=1-E_1\ \ \ \text{ and }\ \ \ \ a\circ_2a^{-1}\chi_2=1-E_2.$$ 
Taking $b_0(x_1,x_2,\xi_1,\xi_2)=a(x_1,x_2,\xi_1,\xi_2)^{-1}\chi(\xi_1,\xi_2)$ and using the asymptotic composition formula, 
we have 
$$r_1=1-a\#b_0=1-ab_0-(a\circ_1b_0-ab_0)-(a\circ_2b_0-ab_0)-\sum_{j\geq 1}c_{-j(\rho_1-\delta_1),-j(\rho_2-\delta_2)}=$$
$$=\underbrace{1-(\underbrace{(a\circ_1 a^{-1})\chi}_{=\chi_2(1-E_1)}+\underbrace{(a\circ_2a^{-1})\chi}_{=\chi_1(1-E_2)}-\chi)}_{=1-(\chi_1+\chi_2-\chi)+\chi_2E_1+\chi_1E_2\in S^{-\infty,-\infty}}-\sum_{j\geq 1}c_{-j(\rho_1-\delta_1),-j(\rho_2-\delta_2)}\in S^{-(\rho_1-\delta_1),-(\rho_2-\delta_2)}_{\rho,\delta}(G\times\widehat G).$$ 
Then we define $b_j:=b_0\# r_j$, with $r_j=r_1\# r_{j-1}\in S^{-j(\rho_1-\delta_1),-j(\rho_2-\delta_2)}_{\rho,\delta}(G\times \widehat G)$ for $j\geq 2$, so that $a\# b_j=(1-r_1)\#r_j$. 
Thus, for all $k\in\mathbb N$, we have 
$$a\#\sum_{j<k}b_j=(1-r_1)\#\biggl(1+\sum_{0<j<k}r_j\biggr)=1+\sum_{0<j<k}r_j-r_1-r_1\#\sum_{0<j<k}r_j=1-r_k,$$ 
where $r_k\in S^{-k(\rho_1-\delta_1),-k(\rho_2-\delta_2)}_{\rho,\delta}(G\times\widehat G)$. Setting $b\sim\sum_{j\geq 0}b_j$, we get $a\# b-1\in S^{-\infty,-\infty}(G\times\widehat G)$. 
To prove the existence of a left parametrix $B$, we take $b_0=\chi a^{-1}$ and $s_1=1-b_0\# a\in S^{-(\rho_1-\delta_1),-(\rho_2-\delta_2)}_{\rho,\delta}(G\times\widehat G)$. 
Then, we set $s_j:=s_{j-1}\# s_1$ for $j\geq 2$, and $b_k=s_k\# b_0$ for $k\geq 1$. The result then follows for $b\sim\sum_{k\geq 0}b_k$, which concludes the proof.
\end{proof}

\section{Bisingular hypoelliptic and elliptic examples, and further examples}\label{sec7}

We devote this final section to some other examples of bielliptic and bihypoelliptic operators on $G=G_1\times G_2$.
\medskip

\noindent {\bf Example 1.} The following is an example of  bihypoelliptic evolution  operator with constant coefficients on a (triple) product group $\mathbb S^1_t\times \mathbb T^n_x\times \mathrm{SU(2)}_y$. We consider
$$P=\partial_t-(1-(\partial_{x_1}^2+\dots+\partial_{x_n}^2))\otimes(1+\mathcal L_{\mathrm{SU(2)}}),$$ where the time is compactified,  as $t\in\mathbb S^1$, and $\mathcal L_{\mathrm{SU(2)}}$ is the positive subelliptic Laplacian on $\mathrm{SU(2)}$. 
To check the bihypoellipticity we check the conditions in Definition \ref{bihypoell}. 

Observe that the symbol $\sigma$ of $P$, that is 
$$\sigma(t,x,y,\tau,k,h)=i\tau+(1+\underbrace{k_1^2+\dots +k_n^2}_{:=|k|^2})(1+h(h+1))I_{2h+1},$$
is invertible for every $(\tau,k,h)\in\mathbb Z\times\mathbb Z^{n}\times\frac{1}{2}\mathbb N_0$ and belongs to the (three factors) class $S^{1,2,2}_{(1,1,1),(0,0,0)}(\mathbb S^1_t\times \mathbb T^n_x\times\mathrm{SU}(2)_y)$. The inequality $$\biggl|\frac{1}{i\tau+(1+|k|^2)(1+h+h^2)}\biggr|\lesssim\langle \tau\rangle^{-m_1'}\langle k\rangle^{-m_2'}\langle h\rangle^{-m_3'}$$  is verified for $(m_1',m_2',m_3')=(\frac{1}{2},1,1)$, and for every multi-indices $\alpha_1,\alpha_2,\alpha_3$, 
the estimate $$\biggl|\frac{\Delta^{\alpha_1,\alpha_2,\alpha_3}(i\tau+(1+|\kappa|^2)(1+h+h^2))}{i\tau+(1+|\kappa|^2)(1+h+h^2)}\biggr|\lesssim_{\alpha_1,\alpha_2,\alpha_3}\langle\tau\rangle^{-|\alpha_1|\rho_1}\langle k\rangle^{-|\alpha_2|\rho_2}\langle h\rangle^{-|\alpha_3|\rho_3}$$ is true when $(\rho_1,\rho_2,\rho_3)=(\frac{1}{4},\frac{1}{2},\frac{1}{2})$. Then, since (1) and (4) in Definition \ref{bihypoell} hold and (2) and (3) in Definition \ref{bihypoell} are trivially satisfied, we conclude that $P$ is bihypoelliptic (with bihypoelliptic order $(\frac{1}{2},1,1)$) and admits a parametrix whose symbol belongs to the class $S^{-\frac{1}{2},-1,-1}_{(\frac{1}{4},\frac{1}{2},\frac{1}{2}),(0,0,0)}(\mathbb S^1\times\mathbb T^n\times\mathrm{SU}(2))$.
\vspace{.3cm}

\noindent{\bf Example 2.}
Let $P=P(\xi)$ be a hypoelliptic polynomial of degree $m$ (H\"ormander \cite{H1}). Let $\rho\in]0,1]$ be the constant appearing in the characterizing estimates
\begin{equation}
  |P^{(\alpha)}(\xi)|\leq C|P(\xi)||\xi|^{-\rho|\alpha|},\quad|\xi|\geq R\geq 1.
  \label{eqHypPol}\end{equation}
In particular,
\begin{equation}
  |P(\xi)|\geq C_0|\xi|^{\rho m},\quad|\xi|\geq R.
  \label{eqHypPol1}\end{equation}
Let $G=\mathbb{R}^n/2\pi\mathbb{Z}^n.$ By \cite{RT} we know that $P\in S^m_{1,0}(G)$ since it belongs to $S^m_{1,0}(\mathbb{R}^n)$,
and by \cite{H} we have that
$$p(\xi)=\frac{\chi(\xi)}{P(\xi)}\in S^{-\rho m}_{\rho,0}(\mathbb{R}^n),$$
where $\chi$ is an excision function that cuts away the zero-set of $P$ (which is compact). Next, we deform $P$ in such a way that the resulting operator, considered in $S^m_{1,0}(G)$, is
exactly invertible in the bisingular sense. Let
$$\mathsf{Z}_R=B_R(0)\cap\mathbb{Z}^n,$$
which is a finite set. Notice that $P^{-1}(0)\cap\mathbb{Z}^n\subset\mathsf{Z}_R.$
We may hence find $c_0$ such that $c_0\not\in P(\mathsf{Z}_R)$ and may also choose it so that $|c_0|\leq C_0/2.$ We therefore have that
for suitable $C',C''>0$ and $R'\geq R$
$$C''|P(k)|\geq|P(k)-c_0|\geq C'|P(k)|,\quad|k|\geq R',\,\,k\in\mathbb{Z}^n.$$
Hence $P(k)-c_0$ satisfies the following properties:
$$P(k)-c_0\not=0,\,\,\forall k\in\mathbb{Z}^n,\leqno(i)$$
$$|P(k)-c_0|\gtrsim|k|^{\rho m}\,\,\,\forall k\in\mathbb{Z}^n,\,\,|k|\gtrsim 1,\leqno(ii)$$
$$|(P-c_0)^{(\alpha)}(\xi)|\lesssim|P(\xi)-c_0||\xi|^{-\rho m},\quad|\xi|\gtrsim 1.\leqno(iii)$$

Let now $G_j=\mathbb{R}^{n_j}/2\pi\mathbb{Z}^{n_j}$, $j=1,2,$ and $G=G_1\times G_2$.
We may therefore consider $a_j=a_j(\xi_j)\in S^{m_j}_{1,0}(G_j)$ for which the properties $(i)$ to $(iii)$ above hold. Hence, the relative operators $A_j=\mathrm{Op}(a_j)$ are exactly invertible with
$A_j^{-1}\in L^{-m_j}_{\rho_j,0}(G_j),$ $j=1,2.$ Finally, consider
$$A_1\boxtimes A_2=\left[\begin{array}{cc}A_1\otimes I & -I\otimes A_2^*\\ I\otimes A_2 & A_1^*\otimes I\end{array}\right]\in S^{m_1,0}_{(1,1),(0,0)}(G)\boxtimes S^{0,m_2}_{(1,1),(0,0)}(G)
\subset S^{m_1,m_2}_{(\rho_1,\rho_2),(0,0)}(G;\mathsf{Mat}_2(\mathbb{C})),$$
since the $(1,0)$-classes are contained in the $(\rho,\delta)$-classes. The system is bihypoelliptic of order  $(m_1(2\rho_1-1),m_2(2\rho_2-1))$, for one has that for all $k_2\in\mathbb{Z}^{n_2}$
$$A_1\boxtimes A_2(D_1,k_2)=\left[\begin{array}{cc}A_1&-\overline{A_2(k_2)}\\ A_2(k_2)&A_1^*\end{array}\right]$$
whose inverse is
$$(A_1\boxtimes A_2(D_1,k_2))^{-1}=\left[\begin{array}{cc}(A_1^*A_1+|A_2(k_2)|^2)^{-1}&0\\0&(A_1A_1^*+|A_2(k_2)|^2)^{-1}\end{array}\right]
\left[\begin{array}{cc}A_1^*&\overline{A_2(k_2)}\\ -A_2(k_2)&A_1\end{array}\right],$$
which belongs to $L^{-m_1(2\rho_1-1)}_{\rho_1,0}(G_1;\mathsf{Mat}_2(\mathbb{C}))$.
Similarly, for $A_1\boxtimes A_2(k_1,D_2),$ $k_1\in\mathbb{Z}^{n_1},$ for which $(A_1\boxtimes A_2(k_1,D_2))^{-1}\in L^{-m_2(2\rho_2-1)}_{\rho_2,0}(G_2;\mathsf{Mat}_2(\mathbb{C}))$.
Therefore, $A_1\boxtimes A_2$ is bihypoelliptic of order $(m_1(2\rho_1-1),m_2(2\rho_2-1))$, with a parametrix $B\in L^{-m_1(2\rho_1-1),-m_2(2\rho_2-1)}_{(\rho_1,\rho_2),(0,0)}(G;\mathsf{Mat}_2(\mathbb{C})).$

\vspace{.3cm}

\noindent{\bf Example 3.}
The next one, is an elliptic example of a bisingular, which is a realization of $\bar\partial$, more precisely of $\bar\partial +\frac{1}{2}(1+i)$. For $(x_1,x_2)\in\mathbb{T}^2
=\mathbb{R}/2\pi\mathbb{Z}\times\mathbb{R}/2\pi\mathbb{Z},$ consider
$$P=\frac12\left[\begin{array}{cc}(\partial_{x_1}+1)\otimes I&I\otimes(\partial_{x_2}+1)\\ I\otimes(-\partial_{x_2}+1) &(\partial_{x_1}+1)\otimes I\end{array}\right]
\in L^{1,1}_{(1,1),(0,0)}(\mathbb{T}^2;\mathsf{Mat}_2(\mathbb{C})).$$
Since the symbol of $\partial_{x_j}$ is $\sigma_{\partial_{x_j}}(k_1,k_2)=ik_j$, for $(k_1,k_2)\in\mathbb{Z}^2$, then $P$ is a bielliptic operator and therefore possesses a parametrix $B\in L^{-1,-1}_{(1,1),(0,0)}(\mathbb{T}^2;\mathsf{Mat}_2(\mathbb{C})).$ 

To see that $P$ is a realization of $\bar\partial +\frac{1}{2}(1+i)$ it suffices to identify a function $f=\mathrm{Re}f+i\mathrm{Im}f$ with the vector $^t\![\mathrm{Re}f\quad  \mathrm{Im}f]$ and observe that
 
$$P \left[\begin{array}{cc}
    \mathrm{Re}f \\ \mathrm{Im}f \end{array}\right]=
\frac12
\begin{bmatrix}
\partial_{x_1}\mathrm{Re}f+\partial_{x_2}\mathrm{Im}f+\mathrm{Re}f+\mathrm{Im}f
\\[8pt]
\partial_{x_1}\mathrm{Im}f-\partial_{x_2}\mathrm{Re}f+\mathrm{Re}f+\mathrm{Im}f
\end{bmatrix}=\bar\partial f+\frac{1}{2}(1+i)f.
$$

\vspace{.3cm}

\noindent{\bf Example 4.} To end the section and the paper, we consider the following example
on $\mathbb T^2=\mathbb T^1_x\times\mathbb T^1_y$: $$\partial_x\boxtimes D_y=\left[\begin{array}{cc}
    \partial_x\otimes 1 & 1\otimes(-D_y)\\ 1\otimes D_y &(-\partial_x)\otimes 1
\end{array}\right]\in L^{1,1}_{(1,1),(0,0)}(\mathbb T^2,\mathsf{Mat}_2(\mathbb C)),$$ where $D_y=-i\partial_y$. If $f\in C^{\infty}(\mathbb{T}^2,\mathbb C)$, then we have $$(\partial_x\boxtimes D_y)f=\left[\begin{array}{cc}
    \partial_x & -D_y\\ D_y &-\partial_x
\end{array}\right]\left[\begin{array}{cc}
    \mathrm{Re}f \\ \mathrm{Im}f \end{array}\right]=\left[\begin{array}{cc}
    \partial_x\mathrm{Re}f+i\partial_y\mathrm{Im}f\\ -i\partial_y\mathrm{Re}f-\partial_x\mathrm{Im}f
\end{array}\right]=$$ $$=\partial_x\mathrm{Re}f+i\partial_y\mathrm{Im}f+i(-i\partial_y\mathrm{Re}f-\partial_x\mathrm{Im}f)=\partial_x\overline f+\partial_yf.$$ Similarly, we have
$$(D_x\boxtimes\partial_y)f=i(\partial_y\overline f-\partial_xf).$$ Therefore
$$(D_x\boxtimes\partial_y)\circ(\partial_x\boxtimes D_y)f=(D_x\boxtimes\partial_y)(\partial_x\overline f+\partial_y f)=i(\partial_y(\overline{\partial_x\overline f+\partial_yf})-\partial_x(\partial_x\overline f+\partial_yf))=i(\partial_y^2-\partial_x^2)\overline f,$$ so that $(D_x\boxtimes\partial_y)\circ(\partial_x\boxtimes D_y)$ can be viewed as a realization of the wave operator acting on $i\bar f$. \ We also have $$(\partial_x\boxtimes D_y)^2-(D_x\boxtimes\partial_y)^2=2(\partial_x^2+\partial_y^2).$$ 
Hence, we may also recover the classical Laplace operator on the torus $\mathbb T^2$ by using the bisingular vector tensor product $\boxtimes$.

\vspace{.3cm}

\noindent\textbf{Data Availability Statement.} All data are provided in full in this paper.

\vspace{.3cm}

\noindent\textbf{Declarations.}
\textbf{Conflict of interest.} There is no conflict of interest.

\bibliographystyle{alpha}

\begin{thebibliography}{90}
\bibitem{AS}
  M.~F.~Atiyah, I.~M.~Singer.
  \newblock The index of elliptic operators, I.
  \newblock Ann. of Math. \textbf{87}(1968), 484--530. DOI: https://doi.org/10.2307/1970715

\bibitem{BS}
 M.~Borsero, R.~Schulz.
\newblock Microlocal properties of bisingular operators.
\newblock J. Pseudo-Differ. Oper. Appl. \textbf{5}(2014), 43--67. 

\bibitem{F}
  V.~Fischer.
  \newblock Intrinsic pseudo-differential calculi on any compact Lie group.
  \newblock J. Funct. Anal. \textbf{268}(2015), 3404--3477.DOI: https://doi.org/10.1016/j.jfa.2015.03.015
  
  
\bibitem{FP}
  S.~Federico, A.~Parmeggiani.
  \newblock On a class of pseudodifferential operators on the product of compact Lie groups.
  \newblock  Math. Nachr. \textbf{296}(2023), 217--242. DOI:https://doi.org/10.1002/mana.202100400

\bibitem{GW}
  C.~G\'erard, M.~Wrochna.
  \newblock Construction of Hadamard states by characteristic Cauchy problem.
  \newblock Anal. PDE \textbf{9}(2016), 111--149.  DOI: 10.2140/apde.2016.9.111 

\bibitem{H}
  L.~H\"ormander.
  \newblock Pseudo-differential operators and hypoelliptic equations.
  \newblock Singular Integrals (Proc. Sympos. Pure Math., Chicago, Ill., 1966), 138--183, Amer. Math. Soc., Providence, RI, 1967. 

\bibitem{H1}
  L.~H\"ormander.
  \newblock \textit{Linear partial differential operators.} Third revised printing.
  \newblock Grundlehren der mathematischen Wissenschaften, Band 116. Springer-Verlag New York, Inc., New York, 1969. vii+288 pp. 

\bibitem{NR}
F.~Nicola, L.~Rodino.
\newblock Residues and index for bisingular operators.
\newblock $C^*$-algebras and elliptic theory, 187--202, Trends Math., Birkh\"auser, Basel, 2006.

\bibitem{P}
  V.~S.~Pilidi.
  \newblock Multidimensional bisingular operators.
  \newblock Soviet Math. Dokl. \textbf{12}(1971), 1723--1726. 

\bibitem{R}
  L.~Rodino.
  \newblock A class of pseudo differential operators on the product of two manifolds and applications.
  \newblock Ann. Scuola Norm. Sup. Pisa Cl. Sci. (4) \textbf{2}(1975), 287--302. DOI: https://www.numdam.org/item/ASNSP\_1975\_4\_2\_2\_287\_0/

\bibitem{RT}
  M.~Ruzhansky, V.~Turunen.
  \newblock \textit{Pseudo-differential operators and symmetries. Background Analysis and Advanced Topics}.
  \newblock  Pseudo-Differential Operators. Theory and Applications, \textbf{2}. Birkh\"auser Verlag, Basel, 2010. xiv+709 pp. DOI: 10.1007/978-3-7643-8514-9


\bibitem{S}
 R.~T.~Seeley.
 \newblock Singular integrals and boundary value problems. 
 \newblock Amer. J. Math. \textbf{88}(1966), 781--809.
 
\end{thebibliography}
\end{document}